\documentclass[fleqn,reqno,10pt,a4paper,final]{amsart}

\usepackage{amsmath}
\usepackage{amssymb}
\usepackage{amsthm}
\usepackage{amscd}
\usepackage[ansinew]{inputenc}
\usepackage{cite}
\usepackage{bbm}
\usepackage{color}
\usepackage[english=american]{csquotes}
\usepackage[final]{graphicx}

\graphicspath{{Pictures/}}

\numberwithin{equation}{section}

\newtheoremstyle{thmlemcorr}{10pt}{10pt}{\itshape}{}{\bfseries}{.}{10pt}{{\thmname{#1}\thmnumber{ #2}\thmnote{ (#3)}}}
\newtheoremstyle{thmlemcorr*}{10pt}{10pt}{\itshape}{}{\bfseries}{.}\newline{{\thmname{#1}\thmnumber{ #2}\thmnote{ (#3)}}}
\newtheoremstyle{defi}{10pt}{10pt}{\itshape}{}{\bfseries}{.}{10pt}{{\thmname{#1}\thmnumber{ #2}\thmnote{ (#3)}}}
\newtheoremstyle{remexample}{10pt}{10pt}{}{}{\bfseries}{.}{10pt}{{\thmname{#1}\thmnumber{ #2}\thmnote{ (#3)}}}
\newtheoremstyle{ass}{10pt}{10pt}{}{}{\bfseries}{.}{10pt}{{\thmname{#1}\thmnumber{ A#2}\thmnote{ (#3)}}}

\theoremstyle{thmlemcorr}
\newtheorem{theorem}{Theorem}
\numberwithin{theorem}{section}
\newtheorem{lemma}[theorem]{Lemma}
\newtheorem{corollary}[theorem]{Corollary}
\newtheorem{proposition}[theorem]{Proposition}

\theoremstyle{thmlemcorr*}
\newtheorem{theorem*}{Theorem}
\newtheorem{lemma*}[theorem]{Lemma}
\newtheorem{corollary*}[theorem]{Corollary}
\newtheorem{proposition*}[theorem]{Proposition}
\newtheorem{problem*}[theorem]{Problem}
\newtheorem{conjecture*}[theorem]{Conjecture}

\theoremstyle{defi}

\theoremstyle{remexample}
\newtheorem{remark}[theorem]{Remark}
\newtheorem{example}[theorem]{Example}

\theoremstyle{ass}

\newcommand{\Crm}{\mathrm{C}}

\newcommand{\Lrm}{\mathrm{L}}

\newcommand{\Wrm}{\mathrm{W}}

\newcommand{\Acal}{\mathcal{A}}

\newcommand{\Ecal}{\mathcal{E}}
\newcommand{\Fcal}{\mathcal{F}}
\newcommand{\Gcal}{\mathcal{G}}
\newcommand{\Hcal}{\mathcal{H}}

\newcommand{\Lcal}{\mathcal{L}}

\newcommand{\Pcal}{\mathcal{P}}

\newcommand{\Ucal}{\mathcal{U}}

\newcommand{\Wcal}{\mathcal{W}}

\newcommand{\Ebf}{\mathbf{E}}

\newcommand{\Mbf}{\mathbf{M}}

\newcommand{\Ybf}{\mathbf{Y}}

\renewcommand{\Bbb}{\mathbb{B}}

\newcommand{\Sbb}{\mathbb{S}}

\DeclareMathOperator{\id}{id}

\DeclareMathOperator*{\wslim}{w*-lim}

\DeclareMathOperator{\supmod}{sup}

\DeclareMathOperator{\curl}{curl}

\DeclareMathOperator{\rank}{rank}
\DeclareMathOperator{\trace}{trace}
\DeclareMathOperator{\spn}{span}
\DeclareMathOperator{\ran}{ran}
\DeclareMathOperator{\supp}{supp}
\newcommand{\ee}{\mathrm{e}}
\newcommand{\ii}{\mathrm{i}}
\newcommand{\set}[2]{\left\{\, #1 \ \ \textup{\textbf{:}}\ \ #2 \,\right\}}
\newcommand{\setn}[2]{\{\, #1 \ \ \textup{\textbf{:}}\ \ #2 \,\}}
\newcommand{\setb}[2]{\bigl\{\, #1 \ \ \textup{\textbf{:}}\ \ #2 \,\bigr\}}
\newcommand{\setB}[2]{\Bigl\{\, #1 \ \ \textup{\textbf{:}}\ \ #2 \,\Bigr\}}
\newcommand{\setBB}[2]{\biggl\{\, #1 \ \ \textup{\textbf{:}}\ \ #2 \,\biggr\}}

\newcommand{\norm}[1]{\|#1\|}

\newcommand{\normb}[1]{\bigl\|#1\bigr\|}

\newcommand{\normBB}[1]{\biggl\|#1\biggr\|}

\newcommand{\abs}[1]{|#1|}
\newcommand{\abslr}[1]{\left|#1\right|}
\newcommand{\absn}[1]{|#1|}
\newcommand{\absb}[1]{\bigl|#1\bigr|}
\newcommand{\absB}[1]{\Bigl|#1\Bigr|}
\newcommand{\absBB}[1]{\biggl|#1\biggr|}

\newcommand{\floorBB}[1]{\biggl\lfloor #1\biggr\rfloor}

\newcommand{\dpr}[1]{\langle #1 \rangle}	

\newcommand{\dprn}[1]{\langle #1 \rangle}
\newcommand{\dprb}[1]{\bigl\langle #1 \bigr\rangle}

\newcommand{\ddpr}[1]{\langle\!\langle #1 \rangle\!\rangle}

\newcommand{\ddprn}[1]{\langle\!\langle #1 \rangle\!\rangle}
\newcommand{\ddprb}[1]{\bigl\langle\hspace{-2.5pt}\bigl\langle #1 \bigr\rangle\hspace{-2.5pt}\bigr\rangle}
\newcommand{\ddprB}[1]{\Bigl\langle\!\!\Bigl\langle #1 \Bigr\rangle\!\!\Bigr\rangle}

\newcommand{\cl}[1]{\overline{#1}}
\newcommand{\di}{\mathrm{d}}
\newcommand{\dd}{\;\mathrm{d}}

\newcommand{\N}{\mathbb{N}}
\newcommand{\R}{\mathbb{R}}
\newcommand{\C}{\mathbb{C}}

\newcommand{\loc}{\mathrm{loc}}
\newcommand{\sym}{\mathrm{sym}}
\newcommand{\skw}{\mathrm{skew}}

\newcommand{\ONE}{\mathbbm{1}}

\newcommand{\toweakstar}{\overset{*}\rightharpoondown}

\newcommand{\toup}{\uparrow}
\newcommand{\todown}{\downarrow}

\newcommand{\conv}{\star}

\newcommand{\sbullet}{\begin{picture}(1,1)(-0.5,-2)\circle*{2}\end{picture}}
\newcommand{\frarg}{\,\sbullet\,}
\newcommand{\BV}{\mathrm{BV}}
\newcommand{\BD}{\mathrm{BD}}
\newcommand{\LD}{\mathrm{LD}}
\newcommand{\BDY}{\mathbf{BDY}}

\newcommand{\toY}{\overset{\Ybf}{\to}}

\DeclareMathOperator{\Tan}{Tan}

\newcommand{\term}[1]{\textbf{#1}}

\newcommand{\proofstep}[1]{\textit{#1}}

\def\Xint#1{\mathchoice 
{\XXint\displaystyle\textstyle{#1}}%
{\XXint\textstyle\scriptstyle{#1}}%
{\XXint\scriptstyle\scriptscriptstyle{#1}}%
{\XXint\scriptscriptstyle\scriptscriptstyle{#1}}%
\!\int} 
\def\XXint#1#2#3{{\setbox0=\hbox{$#1{#2#3}{\int}$} 
\vcenter{\hbox{$#2#3$}}\kern-.5\wd0}} 
\def\dashint{\,\Xint-}

\newcommand{\restrict}{\begin{picture}(10,8)\put(2,0){\line(0,1){7}}\put(1.8,0){\line(1,0){7}}\end{picture}}

\renewcommand{\epsilon}{\varepsilon}
\renewcommand{\phi}{\varphi}

\title[Lower semicontinuity in BD]{Lower semicontinuity for integral functionals in the space of functions of bounded deformation via rigidity and Young measures}

\author{Filip Rindler}
\address{Mathematical Institute, University of Oxford, 24--29 St Giles', Oxford OX1 3LB, United Kingdom.}
\email{rindler@maths.ox.ac.uk}

\begin{document}

\begin{abstract}
We establish a general weak* lower semicontinuity result in the space $\BD(\Omega)$ of functions of bounded deformation for functionals of the form
\begin{align*}
  \Fcal(u) &:= \int_\Omega f \bigl( x, \Ecal u \bigr) \dd x + \int_\Omega f^\infty \Bigl( x,
    \frac{\di E^s u}{\di \abs{E^s u}} \Bigr) \dd \abs{E^s u} \\
  &\qquad + \int_{\partial \Omega} f^\infty \bigl( x, u|_{\partial \Omega} \odot n_\Omega \bigr)
    \dd \Hcal^{d-1},  \qquad u \in \BD(\Omega).
\end{align*}
The main novelty is that we allow for non-vanishing Cantor-parts in the symmetrized derivative $Eu$. The proof is accomplished via Jensen-type inequalities for generalized Young measures and a construction of good blow-ups, which is based on local rigidity arguments for some differential inclusions involving symmetrized gradients, and an iteration of the blow-up construction. This strategy allows us to establish the lower semicontinuity result without an Alberti-type theorem in $\BD(\Omega)$, which is not available at present. We also include existence and relaxation results for variational problems in $\BD(\Omega)$, as well as a complete discussion of some differential inclusions for the symmetrized gradient in two dimensions.

\vspace{8pt}



\vspace{8pt}

\noindent\textsc{Date:} \today.
\end{abstract}

\maketitle



\section{Introduction}

The space $\BD(\Omega)$ of functions of bounded deformation, where $\Omega \subset \R^d$ is a bounded Lipschitz domain, was introduced in~\cite{Suqu78ERSE,Suqu79EFEP,MaStCh79SPDP} in order to treat variational problems from the mathematical theory of plasticity, and has been investigated by various authors, see for example~\cite{Kohn79NEDT,Kohn82NIED,Tema85MPP,TemStr80FBD,AmCoDa97FPFB,FucSer00VMPP}. This space consists of all functions $u \in \Lrm^1(\Omega;\R^d)$ with the property that the distributional symmetrized gradient $Eu$ (defined by duality with the symmetrized gradient $\Ecal u := (\nabla u + \nabla u^T)/2$) is a finite matrix-valued Radon measure on $\Omega$.

Several lower semicontinuity theorems in the space $\BD(\Omega)$ are available, see for example~\cite{BeCoDa98CLSP,BaFoTo00RTSF,Ebob05LSRS,GarZap08LSRS}, but they are all restricted to \emph{special} functions of bounded deformation, i.e.\ such that in the Lebesgue--Radon--Nikod\'{y}m decomposition
\[
  Eu = \Ecal u \, \Lcal^d + E^s u,  \qquad \Ecal u \in \Lrm^1(\Omega;\R_\sym^{d \times d}),
\]
the singular part $E^s u$ originates only from jumps and not from Cantor-type measures.

The aim of this work is to prove the following general lower semicontinuity theorem (this is Theorem~\ref{thm:BD_lsc}, see Section~\ref{sc:setup} for notation):

\begin{theorem} \label{thm:BD_lsc_teaser}
Let $\Omega \subset \R^d$ be a bounded Lipschitz domain, and let $f \colon \cl{\Omega} \times \R_\sym^{d \times d} \to \R$ satisfy the following assumptions:
\begin{itemize}
  \item[(i)] $f$ is a Carath\'{e}odory function,
  \item[(ii)] $\abs{f(x,A)} \leq M(1+\abs{A})$ for some $M > 0$ and all $x \in \cl{\Omega}$, $A \in \R_\sym^{d \times d}$,
  \item[(iii)] $f(x,\frarg)$ is symmetric-quasiconvex for all $x \in \cl{\Omega}$, that is,
\[
  \qquad f(x, A) \leq \dashint_{\omega} f \bigl( x, A + \Ecal \psi(z) \bigr) \dd z
\]
for all $A \in \R_\sym^{d \times d}$ and all $\psi \in \Crm_c^\infty(\omega;\R^d)$, where $\omega \subset \R^d$ is an arbitrary bounded Lipschitz domain,
  \item[(iv)] the (strong) recession function
\[
  \qquad f^\infty(x,A) := \lim_{\substack{\!\!\!\! x' \to x \\ \; t \to \infty}}
    \frac{f(x',tA)}{t}
  \qquad\text{exists for all $x \in \cl{\Omega}$, $A \in \R_\sym^{d \times d}$}
\]
and is (jointly) continuous on $\cl{\Omega} \times \R_\sym^{d \times d}$.
\end{itemize}
Then, the functional
\[
\begin{aligned}
  \Fcal(u) &:= \int_\Omega f \bigl( x, \Ecal u \bigr) \dd x + \int_\Omega f^\infty \Bigl( x,
    \frac{\di E^s u}{\di \abs{E^s u}} \Bigr) \dd \abs{E^s u} \\
  &\qquad + \int_{\partial \Omega} f^\infty \bigl( x, u|_{\partial \Omega} \odot n_\Omega \bigr)
    \dd \Hcal^{d-1},  \qquad u \in \BD(\Omega),
\end{aligned}
\]
is sequentially lower semicontinuous with respect to weak*-convergence in the space $\BD(\Omega)$.
\end{theorem}

In the above definition of $\Fcal$, the function $u|_{\partial \Omega} \in \Lrm^1(\partial \Omega,\Hcal^{d-1};\R^d)$ is the (inner) boundary trace of $u$ onto $\partial \Omega$, while $n_\Omega \colon \partial \Omega \to \Sbb^{d-1}$ is the boundary unit inner normal. If the boundary values of any admissible weakly* converging sequence are the same as the boundary values of the limit, then the boundary term may be omitted. The same is true if $f \geq 0$ since then we can only lose mass in the limit.

It follows from Reshetnyak's Continuity Theorem that the functional $\Fcal$ as defined above is the \enquote{correct} density extension to $\BD(\Omega)$ of the functional
\[
  \Fcal(u) := \int_\Omega f \bigl( x, \Ecal u \bigr) \dd x
    + \int_{\partial \Omega} f^\infty \bigl( x, u|_{\partial \Omega} \odot n_\Omega \bigr) \dd \Hcal^{d-1},
\]
defined for all $u \in \LD(\Omega)$, i.e.\ $u \in \BD(\Omega)$ with $E^s u = 0$. This statement is made precise in Corollary~\ref{cor:F_strictly_cont_ext}.

As immediate consequences of Theorem~\ref{thm:BD_lsc_teaser}, we can prove existence for some variational problems and show a relaxation theorem in $\BD(\Omega)$, see Corollaries~\ref{cor:min_existence} and~\ref{cor:relaxation}.

The strategy for the proof hinges on an idea that was first used in~\cite{Rind10?LSYM} to re-prove the standard lower semicontinuity theorem in the space $\BV$ of functions of bounded variation (see~\cite{AmbDal92RBVQ,FonMul93RQFB}) without Alberti's Rank-One Theorem~\cite{Albe93ROPD}. While still employing the celebrated blow-up technique of Fonseca and M\"{u}ller~\cite{FonMul92QCIL}, the proof in~\cite{Rind10?LSYM} replaces Alberti's Theorem with a rigidity result about solutions to the (under-determined) differential inclusion
\begin{equation}  \label{eq:gradient_DI}
  \nabla v \in \spn\{P\}  \qquad\text{pointwise a.e.,}\qquad  u \in \Wrm_\loc^{1,1}(\R^d;\R^m),
\end{equation}
where $P \in \R^{m \times d}$ is a \emph{fixed} matrix. While for $\BV$ this strategy merely provides a new proof of a known result, in $\BD$ we do not have an Alberti-type theorem at our disposal, and so we need to rely on this new approach in order to prove a general lower semicontinuity theorem.

The key point about Alberti's Theorem is that it provides us with crucial information about blow-ups of BV-functions at singular points. More precisely, this fundamental result ascertains that for $u \in \BV(\Omega;\R^m)$ we have
\[
  \rank \Bigl( \frac{\di D^s u}{\di \abs{D^s u}}(x_0) \Bigr) \leq 1
  \qquad\text{for $\abs{D^s u}$-almost every $x_0 \in \Omega$.}
\]
This allows us to conclude that at such points $x_0$, every blow-up limit can be written as a function which depends only on $x \cdot \xi$ for some direction $\xi \in \Sbb^{d-1}$ (in fact it is the same $\xi$ as in $\frac{\di D^s u}{\di \abs{D^s u}}(x_0) = a \otimes \xi$). The blow-up limit needs to be averaged in order to achieve affine boundary conditions for the application of quasiconvexity, and without the one-directionality of the blow-ups this would incur jumps over the gluing boundaries, which destroy the argument.

The central new observation in~\cite{Rind10?LSYM} is that all blow-ups at points $x_0$ where $\rank \bigl( \frac{\di D^s u}{\di \abs{D^s u}}(x_0) \bigr) \geq 2$ must in fact be affine, so we may apply quasiconvexity in this case as well (we do not even need the additional averaging step). This was called a \enquote{rigidity} argument, because at its heart is the phenomenon that all solutions to the differential inclusion~\eqref{eq:gradient_DI} have a very special structure, and hence we are in a \enquote{rigid} situation.

In $\BD(\Omega)$ the strategy is roughly similar, but faces the additional complication that the rigidity is much weaker: The natural distinction is whether $\frac{\di E^s u}{\di \abs{E^s u}}(x_0)$ can be written as a symmetric tensor product $a \odot b := (a \otimes b + b \otimes a)/2$ for some $a,b \in \R^d$ or not. However, in contrast to the gradient case it turns out that
\[
  \Ecal u = \frac{1}{2} \bigl( \nabla u + \nabla u^T \bigr) \in \spn\{P\}  \qquad\text{pointwise a.e.,}\qquad  u \in \LD_\loc(\R^d),
\]
where the fixed matrix $P \in \R_\sym^{d \times d}$ cannot be written in the form $a \odot b$, does \emph{not} imply that $u$ is affine (see Example~\ref{ex:weak_rigidity}). In particular, blow-ups $v$ of $u$ at points $x_0$ where $\frac{\di E^s u}{\di \abs{E^s u}}(x_0) \neq a \odot b$ for all $a,b \in \R^d$, do not necessarily have a constant multiple of Lebesgue measure as its symmetrized derivative $v$. Using Fourier Analysis and an ellipticity argument, it is however possible to show that $Ev$ is \emph{absolutely continuous} with respect to Lebesgue measure, and as regards blow-ups \enquote{an $\Lcal^d$-absolutely continuous measure is as good as a constant multiple of $\Lcal^d$}. This is so, because we may take a blow-up of the blow-up, which still is a blow-up to the original function (this will be used in the form that tangent measures to tangent measures are tangent measures), and this particular blow-up now indeed has a constant multiple of Lebesgue measure as symmetrized derivative, hence it is affine.

On the other hand, at points $x_0 \in \Omega$ where $\frac{\di E^s u}{\di \abs{E^s u}}(x_0) = a \odot b$ for some $a,b \in \R^d \setminus \{0\}$ with $a \neq b$, it turns out that the symmetrized derivative of any blow-up is the sum of a measure invariant under translations orthogonal to both $a$ and $b$, and possibly an absolutely continuous part with linear density. If this linear part is non-zero, we can use the same \enquote{iterated blow-up trick} mentioned before to get an affine blow-up, so we are again in the above case. If the linear part is zero, we can show that the blow-up limit is the sum of two one-directional functions (depending only on $x \cdot a$ and $x \cdot b$, respectively), and so again we have a well-behaved blow-up limit at our disposal, which may then be averaged (using parallelotopes with face normals $a$ and $b$ instead of the usual cubes) to get an affine function. The case $\frac{\di E^s u}{\di \abs{E^s u}}(x_0) = a \odot a$ for some $a \in \R^d \setminus \{0\}$ is somewhat degenerate, but can also be treated with essentially the same methods (in this case, the remainder is not necessarily linear, not even smooth, but still vanishes in a second blow-up). The pivotal Theorem~\ref{thm:good_blowups} details the construction of good blow-ups and can be considered the core of the present work.

Having thus arrived at an affine function in all of the above cases, we can apply the symmetric-quasiconvexity locally. Figure~\ref{fig:singular_blowups} (p.~\pageref{fig:singular_blowups}) gives an overview over the blow-up contruction whereas Figure~\ref{fig:averaging} (p.~\pageref{fig:averaging}) shows the averaging procedure. The case of two space dimensions is explored in greater detail in Section~\ref{ssc:rigidity_2D} to provide a few more concrete results and examples, even though this is not needed elsewhere.

Like in~\cite{Rind10?LSYM}, the proof is set in the framework of generalized Young measures (or DiPerna--Majda measures), as presented in~\cite{KriRin10CGGY}, the original idea is in~\cite{DiPMaj87OCWS,AliBou97NUIG}. We prove localization principles for Young measures in terms of so-called regular and singular \emph{tangent Young measures}, which encapsulate the blow-up process and contain local information about the Young measure under investigation at the blow-up point, see Propositions~\ref{prop:localize_reg},~\ref{prop:localize_sing}.

Young measures allow to express (the effect of) quasiconvexity locally in a very concise way, namely as Jensen-type inequalities, see Theorem~\ref{thm:BDY_Jensen} for a precise statement. Having established these with the aid of the construction of good blow-ups, the final step to conclude lower semicontinuity in $\BD$ is essentially a straightforward computation (see Theorem~\ref{thm:BD_lsc}). The final Section~\ref{sc:concl_remarks} contains some further remarks on why the use of Young measures (as opposed to a more elementary presentation) is advantageous in this work.

The paper is organized as follows: After fixing notation and proving some auxiliary results in Section~\ref{sc:setup}, the localization principles are the topic of Section~\ref{sc:localization}. Then, Section~\ref{sc:good_blowups} is devoted to proving the existence of good blow-ups and to investigate in more detail some differential inclusions involving $\Ecal u$ in two space dimensions. After the proof of the Jensen-type inequalities in Section~\ref{sc:Jensen}, finally in Section~\ref{sc:lsc} we establish the lower semicontinuity and relaxation theorems and state an existence result for minimizers of variational problems in $\BD(\Omega)$. We end with concluding remarks in Section~\ref{sc:concl_remarks}, and for the convenience of the reader in an appendix we give in full detail (and our notation) Preiss' existence proof for non-zero tangent measures.

\section*{Acknowledgements}

The author wishes to extend many thanks to Jan Kristensen for numerous stimulating discussions related to the topic of this paper and for reading preliminary versions of the manuscript. He is also indebted to Robert V.\ Kohn for a hardcopy of his PhD thesis. The support of the Oxford Centre for Nonlinear PDE (OxPDE) through the EPSRC Science and Innovation award to OxPDE (EP/E035027/1) is gratefully acknowledged. The results in this paper are part of the author's DPhil thesis at the University of Oxford.

\section{Setup and auxiliary results} \label{sc:setup}

\subsection{Notation and linear algebra}

In all of the following, $d \in \N$ will be the number of space dimensions, which we consider fixed. By $B(x_0,r)$ we denote the open ball around $x_0 \in \R^d$ with radius $r > 0$, the open unit ball in $\R^d$ is $\Bbb^d$, its volume is $\omega_d$, and $\Sbb^{d-1}$ is the unit sphere. By $\Omega$ we designate a generic open set in $\R^d$ on which no boundedness or boundary regularity is assumed, unless otherwise stated.

We equip the space $\R^{d \times d}$ of $(d \times d)$-dimensional square matrices with the Frobenius norm $\abs{A} := \sqrt{\sum_{i,j} (A_j^i)^2} = \sqrt{\trace(A^T A)}$ (the Euclidean norm in $\R^{d^2}$), where $A_j^i$ denotes the entry of $A$ in the $i$th row and $j$th column. The Frobenius norm is generated by the scalar product $A : B := \sum_{i,j} A_j^i B_j^i$, under which the space $\R^{d \times d}$ becomes a (real) Hilbert space. By $\R_\sym^{d \times d}$ and $\R_\skw^{d \times d}$ we denote the subspaces of symmetric and skew-symmetric matrices, respectively.

The \term{tensor product} of vectors $a,b \in \R^d$ is $a \otimes b := ab^T$ and the \term{symmetric tensor product} is $a \odot b := (a \otimes b + b \otimes a)/2$. We record the following lemma about symmetric tensor products in $\R^{2 \times 2}$:

\begin{lemma} \label{lem:sym_tensor_prod}
Let $M \in \R_\sym^{2 \times 2}$ be a non-zero symmetic matrix.
\begin{itemize}
  \item[(i)] If $\rank M = 1$, then $M = \pm a \odot a = \pm a \otimes a$ for some vector $a \in \R^2$.
  \item[(ii)] If $\rank M = 2$, then $M = a \odot b$ for some vectors $a,b \in \R^2$ if and only if the two (non-zero) eigenvalues of $M$ have opposite signs.
\end{itemize}
\end{lemma}
\begin{proof}
\proofstep{Ad (i).} Every rank-one matrix $M$ can be written as a tensor product $M = c \otimes d$ for some vectors $c,d \in \R^2 \setminus \{0\}$. By the symmetry, we get $c_1 d_2 = c_2 d_1$, which implies that the vectors $c$ and $d$ are multiples of each other. We therefore find $a \in \R^2$ with $M = \pm a \otimes a$.

\proofstep{Ad (ii).} Assume first that $M = a \odot b$ for some vectors $a,b \in \R^2$ and take an orthogonal matrix $Q \in \R^{2 \times 2}$ such that $QMQ^T$ is diagonal. Moreover,
\[
  QMQ^T = \frac{1}{2}Q \bigl( a \otimes b + b \otimes a \bigr) Q^T
  = \frac{1}{2} \bigl( Qa \otimes Qb + Qb \otimes Qa \bigr) = Qa \odot Qb,
\]
whence we may always assume without loss of generality that $M$ is already diagonal,
\[
  a \odot b = M = \begin{pmatrix} \lambda_1 &  \\  & \lambda_2 \end{pmatrix},
\]
where $\lambda_1,\lambda_2 \neq 0$ are the two eigenvalues of $M$. Writing this out componentwise, we get
\[
  a_1 b_1 = \lambda_1, \qquad a_2 b_2 = \lambda_2, \qquad a_1 b_2 + a_2 b_1 = 0.
\]
As $\lambda_1, \lambda_2 \neq 0$, also $a_1, a_2, b_1, b_2 \neq 0$, and hence
\[
  0 = a_1 b_2 + a_2 b_1 = \frac{a_1}{a_2} \lambda_2 + \frac{a_2}{a_1} \lambda_1.
\]
Thus, $\lambda_1$ and $\lambda_2$ must have opposite signs.

For the other direction, by transforming as before we may assume again that $M$ is diagonal, $M = \Bigl( \begin{smallmatrix} \lambda_1 &  \\ & \lambda_2 \end{smallmatrix} \Bigr)$, and that $\lambda_1$ and $\lambda_2$ do not have the same sign. Then, with $\gamma := \sqrt{-\lambda_1/\lambda_2}$,
we define
\[
  a := \begin{pmatrix} \gamma \\ 1 \end{pmatrix},  \qquad
  b := \begin{pmatrix} \lambda_1 \gamma^{-1} \\ \lambda_2 \end{pmatrix}.
\]
For $\lambda_1 > 0$, $\lambda_2 < 0$ say (the other case is analogous),
\[
  \lambda_1 \gamma^{-1} + \lambda_2 \gamma
  = \lambda_1 \sqrt{\frac{\abs{\lambda_2}}{\lambda_1}} - \abs{\lambda_2} \sqrt{\frac{\lambda_1}{\abs{\lambda_2}}}
  = 0,
\]
and therefore
\[
  a \odot b
  = \frac{1}{2} \begin{pmatrix} \lambda_1 & \lambda_2 \gamma \\ \lambda_1 \gamma^{-1} & \lambda_2 \end{pmatrix}
  + \frac{1}{2} \begin{pmatrix} \lambda_1 & \lambda_1 \gamma^{-1} \\ \lambda_2 \gamma & \lambda_2 \end{pmatrix}
  = M.
\]
This proves the claim.
\end{proof}


\subsection{Measure theory}

In the following, we briefly gather some of the notions from measure theory employed in this paper. More information can for example be found in~\cite{FonLeo07MMCV,AmFuPa00FBVF,Matt95GSME}.

The space $\Mbf_\loc(\R^d;\R^N)$ contains all $\R^N$-valued set functions that are defined on the relatively compact Borel subsets, and that are $\sigma$-additive and finite when restricted to the Borel $\sigma$-algebra on a compact subset of $\R^d$. We call its elements vector-valued \term{local (Radon) measures}. Most often, in the previous notation $\R^N$ is just as a placeholder for \enquote{$\R^{d \times d}$}. The subspace $\Mbf(\R^d;\R^N)$ contains all vector-valued \term{finite (Radon) measures} on the Borel $\sigma$-algebra on $\R^d$ with values in $\R^N$. Positive measures are contained in the analogous spaces $\Mbf_\loc^+(\R^d)$ and $\Mbf^+(\R^d)$, respectively. A probability measure is a positive measure $\mu \in \Mbf^+(\R^d)$ with $\mu(\R^d) = 1$, we write $\mu \in \Mbf^1(\R^d)$. We will also employ the spaces $\Mbf(V;\R^N)$, $\Mbf^+(V;\R^N)$, $\Mbf^1(V;\R^N)$ with a bounded Borel set $V \subset \R^d$ replacing $\R^d$; all of the following statements, with the appropriate adjustments, also hold for these spaces.

For every local measure $\mu \in \Mbf_\loc(\R^d;\R^N)$, we denote by $\abs{\mu} \in \Mbf_\loc^+(\R^d)$ its \term{total variation measure}. The restriction of a (local) measure $\mu \in \Mbf_\loc(\R^d;\R^N)$ to a Borel set $A \subset \R^d$ is written as $\mu \restrict A$ and defined by $(\mu \restrict A)(B) := \mu(B \cap A)$ for all relatively compact Borel sets $B \subset \R^d$. For a positive measure $\mu \in \Mbf_\loc^+(\R^d)$, the support $\supp \mu$ is the set of all $x \in \R^d$ such that $\mu(B(x,r)) > 0$ for all $r > 0$, which is always a closed set. For a vector measure $\mu \in \Mbf_\loc(\R^d;\R^N)$, the support of $\mu$ is simply the support of $\abs{\mu}$.

Lebesgue measure in $\R^d$ is denoted by $\Lcal^d$, sometimes augmented to $\Lcal_x^d$ to give a name to the integration variable. For a Lebesgue-measurable set $A \subset \R^d$, we will often simply write $\abs{A}$ instead of $\Lcal^d(A)$. The symbol $\Hcal^k$ stands for the $k$-dimensional Hausdorff outer measure, $k \in [0,\infty)$. When restricted to a $\Hcal^k$-rectifiable set $S \subset \R^d$ (see Section~2.9 of~\cite{AmFuPa00FBVF}, we only need the fact that Lipschitz boundaries are $\Hcal^k$-rectifiable), $\Hcal^k \restrict S$ is a local Radon measure.

The pairing $\dpr{f,\mu}$ between a Borel measurable function $f \colon \R^d \to \R^N$ and a positive measure $\mu \in \Mbf^+(\R^d)$, or, if $f$ has compact support also with $\mu \in \Mbf_\loc^+(\R^d)$, is defined as
\[
  \dpr{f,\mu} := \int f \dd \mu   \quad\text{($\in \R^N$),}
\]
provided this integral exists.

Every measure $\mu \in \Mbf_\loc(\R^d;\R^N)$ has a (unique) \term{Lebesgue--Radon--Nikod\'{y}m decomposition} $\mu = \frac{\di \mu}{\di \lambda} \lambda + \mu^s$ with respect to a positive measure $\lambda \in \Mbf_\loc^+(\R^d)$, i.e.\
\[
  \mu(A) = \int_A \frac{\di \mu}{\di \lambda} \dd \lambda + \mu^s(A)
\]
for all relatively compact Borel sets $A \subset \R^d$. In this decomposition, $\mu^s$ and $\lambda$ are mutually singular, i.e.\ concentrated on mutually negligible sets. The function $\frac{\di \mu}{\di \lambda} \in \Lrm_\loc^1(\R^d,\lambda;\R^N)$ is called the \term{density} of $\mu$ with respect to $\lambda$ and may be computed by
\[
  \frac{\di \mu}{\di \lambda}(x_0) = \lim_{r \todown 0} \frac{\mu(B(x_0,r))}{\lambda(B(x_0,r))}
    \qquad\text{for $\lambda$-a.e.\ $x_0 \in \supp \lambda$.}
\]
If not otherwise specified, $\mu^s$ will always mean the singular part of the measure $\mu$ with respect to Lebesgue measure. The function $\frac{\di \mu}{\di \abs{\mu}} \in \Lrm_\loc^1(\R^d,\abs{\mu};\R^N)$ is called the \term{polar function} of $\mu$ and satisfies $\absb{\frac{\di \mu}{\di \abs{\mu}}(x)} = 1$ at $\abs{\mu}$-almost every $x \in \R^d$.

Several times we will employ the \term{pushforward} $T_* \mu := \mu \circ T^{-1} \in \Mbf_\loc(\R^d;\R^N)$ of a local measure $\mu \in \Mbf_\loc(\R^d;\R^N)$ under an affine map $T \colon \R^d \to \R^d$, $x \mapsto x_0 + L x$, where $x_0 \in \R^d$ and $L \in \R^{d \times d}$ is an invertible matrix (of course, pushforwards are defined for more general $T$, but we will not need those). For a measurable function $f \colon \R^d \to \R^N$ and $\mu \in \Mbf_\loc^+(\R^d)$, we have the transformation rule
\[
  \dprb{f, T_*\mu} = \int f \dd (T_*\mu) = \int f \circ T \dd \mu = \dprb{f \circ T, \mu}
\]
provided one, hence both, of these integrals are well-defined. Also, with $\det T := \det L$, we have the following formulas for densities:
\begin{equation} \label{eq:pushforward_density_trafo}
  \frac{\di T_*\mu}{\di \Lcal^d} = \abslr{\det T}^{-1}\frac{\di \mu}{\di \Lcal^d} \circ T^{-1},  \qquad
  \frac{\di T_*\mu}{\di \abs{T_*\mu}} = \frac{\di \mu}{\di \abs{\mu}} \circ T^{-1}.
\end{equation}
Mostly, we will use pushforwards under the blow-up transformation $T^{(x_0,r)}(x) := (x-x_0)/r$, where $x_0 \in \R^d$ and $r > 0$. For this particular transformation we have $\abs{\det T^{(x_0,r)}}^{-1} = r^d$.

The preceding spaces of measures have several different notions of convergence that are relevant for our theory: The \term{norm (or strong) convergence} of a sequence $(\mu_j) \subset \Mbf(\R^d;\R^N)$ to $\mu \in \Mbf(\R^d;\R^N)$ means that $\abs{\mu_j - \mu}(\R^d) \to 0$.

By the Riesz Representation Theorem, we may consider $\Mbf_\loc(\R^d;\R^N)$ as the dual space to the locally convex space $\Crm_c(\R^d;\R^N)$, and $\Mbf(\Omega;\R^N)$ as the dual space to the Banach space $\Crm_0(\Omega;\R^N)$. These dualities induce the \term{(local) weak* convergence} $\mu_j \toweakstar \mu$ in $\Mbf_\loc(\R^d;\R^N)$ defined as $\dpr{\psi,\mu_j} \to \dpr{\psi,\mu}$ (in $\R^N$) for all $\psi \in \Crm_c(\R^d)$ as well as the \term{weak* convergence} $\mu_j \toweakstar \mu$ in $\Mbf(\Omega;\R^N)$ meaning $\dpr{\psi,\mu_j} \to \dpr{\psi,\mu}$ for all $\psi \in \Crm_0(\Omega)$. Both convergences (we only work with convergences here, not with topologies) have good compactness properties. In particular, every sequence $(\mu_j) \subset \Mbf_\loc(\R^d;\R^N)$ satisfying $\sup_j \abs{\mu_j}(K) < \infty$ for all compact $K \subset \R^d$ has a (locally) weakly* converging subsequence. Likewise, if for a sequence $(\mu_j) \subset \Mbf(\Omega;\R^N)$ we have $\sup_j \abs{\mu_j}(\Omega) < \infty$, then this sequence is (sequentially) weakly* relatively compact.

Finally, with the \term{area functional} $\langle\frarg\rangle \colon \Mbf(\Omega;\R^N) \to \R$, defined by 
\[
  \langle \mu \rangle := \int_\Omega \sqrt{1+\absB{\frac{\di \mu}{\di \Lcal^d}}^2} \dd x + \abs{\mu^s}(\Omega),
  \qquad \mu \in \Mbf(\Omega;\R^N),
\]
we define \term{$\langle\frarg\rangle$-strict convergence} in $\Mbf(\Omega;\R^N)$ to comprise $\mu_j \toweakstar \mu$ and $\langle \mu_j \rangle \to \langle \mu \rangle$, see~\cite{KriRin10RSIF,KriRin10CGGY} and also the Reshetnyak Continuity Theorem~\ref{thm:reshetnyak} for a discussion why $\langle\frarg\rangle$-strict convergence is important here. It can be shown (by mollification) that smooth measures are dense in $\Mbf(\Omega;\R^N)$ with respect to the $\langle\frarg\rangle$-strict convergence. Notice that by Reshetnyak's Continuity Theorem~\ref{thm:reshetnyak} below, $\langle\frarg\rangle$-strict convergence is stronger than the usual notion of strict convergence.

\subsection{Tangent measures} \label{ssc:Tan}

Tangent measures are a powerful tool in Geometric Measure Theory for investigating the local structure of Radon measures. In contrast to the previous work~\cite{Rind10?LSYM}, which employed the restricted notion of tangent measures from Section~2.7 in~\cite{AmFuPa00FBVF}, we here use Preiss' original definition~\cite{Prei87GMDR}. This has several advantages from a technical point of view (in particular, we can use the general theory for tangent measures), and is also the more elegant approach from the conceptual point of view. General information on tangent measures can for example be found in Chapter~14 of~\cite{Matt95GSME} and also in~\cite{Prei87GMDR}.

Let $T^{(x_0,r)}(x) := (x-x_0)/r$ for $x_0 \in \R^d$ and $r > 0$. For a vector-valued Radon measure $\mu \in \Mbf_\loc(\R^d;\R^N)$ and $x_0 \in \R^d$, a \term{tangent measure} to $\mu$ at $x_0$ is any weak* limit in the space $\Mbf_\loc(\R^d;\R^N)$ of the rescaled measures $c_n T^{(x_0,r_n)}_* \mu$ for some sequence $r_n \todown 0$ of radii and some strictly positive rescaling constants $c_n > 0$. The set of all such tangent measures is denoted by $\Tan(\mu,x_0)$ and the sequence $c_n T^{(x_0,r_n)}_* \mu$ is called a \term{blow-up sequence}. From the definition it follows that $\Tan(\mu,x_0) = \{0\}$ for all $x_0 \notin \supp \mu$. Preiss originally excluded the zero measure from $\Tan(\mu,x_0)$ explicitly, but for us it has some technical advantages to include it.

Is is a fundamental result of Preiss that the set $\Tan(\mu,x_0)$ contains non-zero measures at $\abs{\mu}$-almost every $x_0 \in \supp \mu$ (or, equivalently, at $\abs{\mu}$-almost every $x_0 \in \R^d$). This is proved in Theorem~2.5 of~\cite{Prei87GMDR}, but since this is the only result from Preiss' paper needed here, a fully-detailed proof is given in the appendix for the convenience of the reader.

One can show, see Remark~14.4~(i) of~\cite{Matt95GSME}, that for any non-zero $\tau \in \Tan(\mu,x_0)$ we may always choose the rescaling constants $c_n$ in the blow-up sequence $c_n T_*^{(x_0,r_n)} \mu \toweakstar \tau$ to be
\[
  c_n := c \bigl[ \mu(x_0 + r_n \cl{U}) \bigr]^{-1}
\]
for any bounded open set $U \subset \R^d$ containing the origin such that $\tau(U) > 0$, and some constant $c = c(U) > 0$. This involves passing to a (non-relabeled) subsequence if necessary.

A very special property of tangent measures is that at $\abs{\mu}$-almost every $x_0 \in \R^d$ and for all sequences $r_n \todown 0$, $c_n > 0$, it holds that
\begin{equation} \label{eq:Tan_abs}
  \tau = \wslim_{n\to\infty} c_n T_*^{(x_0,r_n)} \mu  \quad\text{if and only if}\quad
  \abs{\tau} = \wslim_{n\to\infty} c_n T_*^{(x_0,r_n)} \abs{\mu},
\end{equation}
which in particular implies
\begin{equation} \label{eq:Tan_density}
  \Tan(\mu,x_0) = \frac{\di \mu}{\di \abs{\mu}}(x_0) \cdot \Tan(\abs{\mu},x_0).
\end{equation}
See e.g.\ Theorem~2.44 in~\cite{AmFuPa00FBVF} for a proof (with a different definition for tangent measures; the proof, however, carries over).

If $\mu \in \Mbf_\loc^+(\R^d)$ is absolutely continuous with respect to a positive measure $\lambda \in \Mbf_\loc^+(\R^d)$, then $\Tan(\mu,x_0) = \Tan(\lambda,x_0)$ for $\lambda$-almost all $x_0 \in \R^d$. This fact is proved in Lemma~14.6 of~\cite{Matt95GSME} and is particularly powerful in conjunction with the following result, see Lemma~14.5 of~\cite{Matt95GSME}: For a Borel set $E \subset \R^d$, at all $\mu$-density points $x_0 \in \supp \mu$ of $E$, i.e.\ all points $x_0 \in \supp \mu$ such that
\[
  \lim_{r \todown 0} \frac{\mu(B(x_0,r) \setminus E)}{\mu(B(x_0,r))} = 0,
\]
it holds that
\[
  \Tan(\mu,x_0) = \Tan(\mu \restrict E,x_0).
\]
In particular, this relation holds for $\mu$-almost every $x_0 \in E$.

As an application, we can first cut off the singular part of an arbitrary measure $\mu \in \Mbf_\loc(\R^d;\R^N)$, then use the first fact on the remaining (absolutely continuous) part, and also~\eqref{eq:Tan_density}, to see
\[
  \Tan(\mu,x_0) = \setBB{ \alpha \, \frac{\di \mu}{\di \Lcal^d}(x_0) \, \Lcal^d }{ \alpha \in \R }
  \qquad \text{for $\Lcal^d$-a.e.\ $x_0 \in \R^d$.}
\]
In particular, at such $x_0$ there exists a sequence $r_n \todown 0$ satisfying
\[
  r_n^{-d} T_*^{(x_0,r_n)} \mu  \quad\toweakstar\quad  \frac{\di \mu}{\di \Lcal^d}(x_0) \, \Lcal^d
  \quad \in \Tan(\mu,x_0).
\]

The next fact, that tangent measures to tangent measures are again tangent measures, is very important for our theory and we state it explicitly as a lemma:

\begin{lemma} \label{lem:TanTan_eq_Tan}
Let $\mu \in \Mbf_\loc(\R^d;\R^N)$. For $\abs{\mu}$-almost every $x_0 \in \R^d$ and every $\tau \in \Tan(\mu,x_0)$, it holds that $\Tan(\tau,y_0) \subset \Tan(\mu,x_0)$ for all $y_0 \in \R^d$.
\end{lemma}

A proof of this can be found in Theorem~14.16 of~\cite{Matt95GSME}. Note that since we imposed that $\Tan(\mu,x_0)$ contains the zero-measure for \emph{every} $x_0 \in \R^d$, in the statement above we can allow $y_0$ arbitrary instead of just from $\supp \tau$ as in \emph{loc.~cit.}

\subsection{Functions of bounded deformation}  \label{ssc:BD}

From now on, let $\Omega$ be an open domain with Lipschitz boundary (some extensions to non-Lipschitz $\Omega$ are possible, but we omit details here for simplicity). For a function $u \in \Wrm_\loc^{1,1}(\Omega;\R^d)$ define the \term{symmetrized gradient} (or \term{deformation tensor}) via
\begin{equation} \label{eq:sym_grad}
  \Ecal u := \frac{1}{2} \bigl( \nabla u + \nabla u^T \bigr),  \qquad \Ecal u \in \Lrm_\loc^1(\Omega;\R_\sym^{d \times d}).
\end{equation}
The space $\BD(\Omega)$ of \term{functions of bounded deformation} is the space of functions $u \in \Lrm^1(\Omega;\R^d)$ such that the distributional \term{symmetrized derivative}
\[
  Eu = \frac{1}{2} \bigl( Du + Du^T \bigr)
\]
is (representable as) a finite Radon measure, $Eu \in \Mbf(\Omega;\R_\sym^{d \times d})$. The space $\BD(\Omega)$ is a Banach space under the norm
\[
  \norm{u}_{\BD(\Omega)} := \norm{u}_{\Lrm^1(\Omega;\R^d)} + \abs{Eu}(\Omega).
\]
Of course, technically we work with equivalence classes of functions equal almost everywhere, but this will be mostly implicit.

We split $Eu$ according to the Lebesgue--Radon--Nikod\'{y}m decomposition
\[
  Eu = E^a u + E^s u = \Ecal u \, \Lcal^d + E^s u,
\]
where (in analogy to before) $\Ecal u = \frac{\di Eu}{\di \Lcal^d} \in \Lrm^1(\Omega;\R_\sym^{d \times d})$ denotes the Radon--Nikod\'{y}m derivative of $Eu$ with respect to Lebesgue measure and $E^s u$ is singular. We call $\Ecal u$ the \term{approximate symmetrized gradient} (the reason for the word \enquote{approximate} can be found in Section~4 of~\cite{AmCoDa97FPFB}).

The subspace $\LD(\Omega)$ of $\BD(\Omega)$ consists of all $\BD$-functions such that $Eu$ is absolutely continuous with respect to Lebesgue measure (i.e.\ $E^s u = 0$). Note that even in this case one has to distinguish between the \emph{measure} $Eu = E^a u$ and its \emph{density} $\Ecal u$, in particular with respect to pushforwards, cf.~\eqref{eq:pushforward_density_trafo}. The space $\BD_\loc(\R^d)$ is the space of functions $u \in \Lrm_\loc^1(\R^d;\R^d)$ such that the restriction of $u$ to every relatively compact open subset $U \subset \R^d$ lies in $\BD(U)$.

Since there is no Korn inequality in $\Lrm^1$, see~\cite{Orns62NIDO,CoFaMa05NACE}, it follows that $\Wrm^{1,1}(\Omega;\R^d)$ is a proper subspace of $\LD(\Omega)$ and also that the space $\BV(\Omega;\R^d)$ of functions of bounded variation, i.e.\ the space of $\Lrm^1$-functions such that the distributional derivative is representable as a finite Radon measure (see~\cite{AmFuPa00FBVF}), is a proper subspace of $\BD(\Omega)$.

A \term{rigid deformation} is a skew-symmetric affine map $u \colon \R^d \to \R^d$, i.e.\ $u$ is of the form
\[
  u(x) = u_0 + Rx,  \qquad \text{where $u_0 \in \R^d$, $R \in \R_\skw^{d \times d}$.}
\]
The following lemma is well-known and will be used many times in the sequel, usually without mentioning. We reproduce its proof here, because the central formula~\eqref{eq:Wu_identity} will be of use later.

\begin{lemma} \label{lem:E_kernel}
The kernel of the linear operator $\Ecal \colon \Crm^1(\R^d;\R^d) \to \Crm(\R^d;\R_\sym^{d \times d})$ given in~\eqref{eq:sym_grad} is the space of rigid deformations.
\end{lemma}

\begin{proof}
It is obvious that $\Ecal u$ vanishes for a rigid deformation $u$.

For the other direction, let $u \in \Crm^1(\R^d;\R^d)$ with $\Ecal u \equiv 0$, and define
\[
  \Wcal u := \frac{1}{2} \bigl( \nabla u - \nabla u^T \bigr).
\]
Then, for all $i,j,k = 1,\ldots,d$, we have in the sense of distributions,
\begin{equation} \label{eq:Wu_identity}
\begin{aligned}
  \partial_k \Wcal u_j^i &= \frac{1}{2} \bigl( \partial_{kj} u^i - \partial_{ki} u^j \bigr) \\
  &= \frac{1}{2} \bigl( \partial_{jk} u^i + \partial_{ji} u^k \bigr)
    - \frac{1}{2} \bigl( \partial_{ij} u^k + \partial_{ik} u^j \bigr) \\
  &= \partial_j \Ecal u_k^i - \partial_i \Ecal u_k^j \equiv 0.
\end{aligned}
\end{equation}
As $\nabla u = \Ecal u + \Wcal u$, this entails that $\nabla u$ is a constant, hence $u$ is affine and it is clear that it in fact must be a rigid deformation.
\end{proof}

It is an easy consequence of the previous lemma that $u \in \BD_\loc(\R^d)$ with $Eu = A \Lcal^d$, where $A \in \R_\sym^{d \times d}$ is a fixed symmetric matrix, is an affine function. More precisely, $u(x) = u_0 + (A+R)x$ for some $u_0 \in \R^d$ and $R \in \R_\skw^{d \times d}$ skew-symmetric.

As notions of convergence in $\BD(\Omega)$ we have the \term{norm convergence}, the \term{weak* convergence} $u_j \toweakstar u$ in $\BD(\Omega)$ if $u_j \to u$ strongly in $\Lrm^1$ and $Eu_j \toweakstar Eu$ in the sense of finite measures, and the \term{$\langle\frarg\rangle$-strict convergence}, defined like weak* convergence, but additionally requiring that $\langle Eu_j \rangle(\Omega) \to \langle Eu \rangle(\Omega)$. If $\sup_j \norm{Eu_j}_{\BD(\Omega)} < \infty$, then there exists a weakly* converging subsequence.

In $\BD_\loc(\R^d)$ we let \term{weak* convergence} mean $u_j \to u$ in $\Lrm_\loc^1$ (i.e.\ in $\Lrm^1$ on all compact subsets of $\R^d)$ and $Eu_j \toweakstar Eu$ in $\Mbf_\loc(\R^d;\R^{d \times d})$. If $(u_j) \subset \BD_\loc(\R^d)$ and $\sup_j \norm{u_j}_{\BD(U)} < \infty$ for all relatively compact open $U \subset \R^d$, then there exists a weakly* converging subsequence.

Since $\Omega$ has a Lipschitz boundary, the \term{trace} $u|_{\partial \Omega}$ of $u$ onto $\partial \Omega$ is well-defined in the sense that there exists a bounded linear operator $u \mapsto u|_{\partial \Omega}$ mapping $\BD(\Omega)$ (surjectively) onto $\Lrm^1(\partial \Omega,\Hcal^{d-1};\R^d)$ (the space of $\Hcal^{d-1}$-integrable functions on $\partial \Omega$ with values in $\R^d$) that coincides with the natural trace for all $u \in \BD(\Omega) \cap \Crm(\Omega;\R^d)$, see Theorem~II.2.1 of~\cite{TemStr80FBD}.

If $u \in \BD(\Omega)$ with $u|_{\partial \Omega} = 0$, then we also have the Poincar\'{e} inequality
\[
  \norm{u}_{\BD(\Omega)} \leq C \abs{Eu}(\Omega),
\]
where $C = C(\Omega)$ only depends on the domain $\Omega$, see Proposition~II.2.4 in~\cite{Tema85MPP}. Moreover, it is shown for example in~\cite{TemStr80FBD} (or see Remark~II.2.5 of~\cite{Tema85MPP}) that for each $u \in \BD(\Omega)$ there exists a rigid deformation $r$ such that
\[
  \norm{u+r}_{\Lrm^{d/(d-1)}(\Omega;\R^d)} \leq C \abs{Eu}(\Omega),
\]
where again $C = C(\Omega)$.

More information on $\BD(\Omega)$ and applications can be found in~\cite{Tema85MPP,AmCoDa97FPFB,FucSer00VMPP} and also in~\cite{Suqu78ERSE,Suqu79EFEP,MaStCh79SPDP,Kohn79NEDT,TemStr80FBD,Kohn82NIED}.

\subsection{Integrands} \label{ssc:integrands}

For $f \in \Crm(\Omega \times \R^N)$, where $\Omega \subset \R^d$ is an open set, define the transformation
\[
  (Sf)(x,\hat{A}) := (1-\abs{\hat{A}})f \biggl( x, \frac{\hat{A}}{1-\abs{\hat{A}}} \biggr),
  \qquad\text{$x \in \Omega$, $\hat{A} \in \Bbb^N$.}
\]
Then $Sf \in \Crm(\Omega \times \Bbb^N)$, and we let
\begin{align*}
  \Ebf(\Omega;\R^N) := \setb{ f \in \Crm(\Omega \times \R^N) }{ &\text{$Sf$ extends into a bounded,} \\
  &\text{continuous function on $\cl{\Omega \times \Bbb^N}$} }.
\end{align*}
In particular, all $f \in \Ebf(\Omega;\R^N)$ have \term{linear growth at infinity}, that is there exists $M > 0$ such that
\[
  \abs{f(x,A)} \leq M(1+\abs{A})  \qquad\text{for all $x \in \cl{\Omega}$, $A \in \R^N$;}
\]
the smallest such $M$ is called the \term{linear growth constant} of $f$. Also, by definition, for each $f \in \Ebf(\Omega;\R^N)$ the limit
\begin{equation} \label{eq:f_infty_def}
  \qquad f^\infty(x,A) := \lim_{\substack{\!\!\!\! x' \to x \\ \!\!\!\! A' \to A \\ \; t \to \infty}}
    \frac{f(x',tA')}{t},  \qquad \text{$x \in \cl{\Omega}$, $A \in \R^N$,}
\end{equation}
exists and defines a positively $1$-homogeneous function (i.e.\ $f(x,\theta A) = \theta f(x,A)$ for all $\theta \geq 0$), called the \term{recession function} of $f$. The norm
\[
  \norm{f}_{\Ebf(\Omega;\R^N)} := \norm{Sf \colon \cl{\Omega \times \Bbb^N}}_\infty,
  \qquad f \in \Ebf(\Omega;\R^N)
\]
turns $\Ebf(\Omega;\R^N)$ into a Banach space.

More generally, for functions $h \colon \R^N \to \R$ with linear growth at infinity, we define the \term{generalized recession function} $h^\# \colon \R^N \to \R$ by
\begin{equation} \label{eq:f_sharp_def}
  h^\#(A) := \limsup_{\substack{\!\!\!\! A' \to A \\ \; t \to \infty}} \frac{h(tA')}{t},  \qquad A \in \R^N,
\end{equation}
which again is positively $1$-homogeneous. We also use the recession function $h^\infty$ as in~\eqref{eq:f_infty_def} (without $x$-dependence) if it is defined.

As shown in Lemma~2.3 of~\cite{AliBou97NUIG}, for an upper semicontinuous function $f \colon \Omega \times \R^N \to \R$ with linear growth at infinity, we may find a decreasing sequence $(f_k) \subset \Ebf(\Omega;\R^N)$ with
\[
  \inf_{k \in \N} f_k = \lim_{k \to \infty} f_k = f, \qquad
  \inf_{k \in \N} f_k^\infty = \lim_{k \to \infty} f_k^\infty = f^\#
  \qquad\text{(pointwise).}
\]
Furthermore, the linear growth constants of the $f_k$ can be chosen to be bounded by the linear growth constant of $f$.

The space $\Ebf_c(\R^d;\R^N)$ is defined similarly to $\Ebf(\Omega;\R^N)$, but additionally we require that for each element $f \in \Ebf_c(\R^d;\R^N) \subset \Crm(\R^d \times \R^N)$ there exists a compact set $K \subset \R^d$ such that $\supp f(\frarg,A) \subset K$ for all $A \in \R^N$. In this work, we will mostly employ the spaces $\Ebf(\Omega;\R_\sym^{d \times d})$ and $\Ebf_c(\R^d;\R_\sym^{d \times d})$, where $\R^N$ is replaced by $\R_\sym^{d \times d}$. Clearly, all the aforementioned results also hold for these spaces.

The following is a variant of the well-known Reshetnyak Continuity Theorem, see~\cite{Resh68WCCA} for the original version and the appendix of~\cite{KriRin10RSIF} for a proof of the present extension.

\begin{theorem}[Reshetnyak Continuity Theorem] \label{thm:reshetnyak}
Let $\mu_j,\mu \in \Mbf(\R^d;\R^N)$ and assume $\mu_j \to \mu$ with respect to the $\langle\frarg\rangle$-strict convergence. Then,
\begin{align*}
  &\int f \Bigl( x, \frac{\di \mu_j}{\di \Lcal^d}(x) \Bigr) \dd x + \int f^\infty \Bigl( x,
    \frac{\di \mu_j^s}{\di \abs{\mu_j^{s}}}(x) \Bigr) \dd \abs{\mu_j^{s}}(x) \\
  &\qquad \to \int f \Bigl( x, \frac{\di \mu}{\di \Lcal^d}(x) \Bigr) \dd x + \int f^\infty \Bigl(
    x, \frac{\di \mu^s}{\di \abs{\mu^{s}}}(x) \Bigr) \dd \abs{\mu^{s}}(x)
\end{align*}
for all $f \in \Ebf(\R^d;\R^N)$.
\end{theorem}

Since $\LD(\Omega)$ is $\langle\frarg\rangle$-strictly dense in $\BD(\Omega)$ (by a mollification argument), this immediately implies the following result:

\begin{corollary} \label{cor:F_strictly_cont_ext}
Let $f \in \Ebf(\Omega;\R_\sym^{d \times d})$. Then, the $\langle\frarg\rangle$-strictly continuous extension of the functional
\[
  \Fcal(u) := \int_\Omega f \bigl( \Ecal u(x) \bigr) \dd x,  \qquad u \in \LD(\Omega)
\]
onto the space $u \in \BD(\Omega)$ is
\[
  \overline{\Fcal}(u) := \int_\Omega f \bigl( x,\Ecal u(x) \bigr) \dd x + \int_\Omega f^\infty \Bigl( x,
  \frac{\di E^s u}{\di \abs{E^s u}}(x) \Bigr) \dd \abs{E^s u}(x),
\]
where now $u \in \BD(\Omega)$.
\end{corollary}

Of course, the previous result also holds with an additional boundary term.

\begin{remark}
Corollary~\ref{cor:F_strictly_cont_ext} strongly suggests that $\overline{\Fcal}$ as defined above is the right candidate for the weakly* lower semicontinuous envelope of $\Fcal$. That this is indeed true is the content of Corollary~\ref{cor:relaxation}.
\end{remark}

Finally, a function $f \colon \cl{\Omega} \times \R^N \to \R$ is a \term{Carath\'{e}odory integrand} if it is Borel measurable in its first and continuous in its second argument.

\subsection{Symmetric quasiconvexity}  \label{ssc:sym_qc}

A locally bounded Borel function $h \colon \R_\sym^{d \times d} \to \R$ is called \term{symmetric-quasiconvex} if
\[
  h(A) \leq \dashint_{\omega} h \bigl( A + \Ecal \psi(z) \bigr) \dd z
\]
for all $A \in \R_\sym^{d \times d}$ and all $\psi \in \Crm_c^\infty(\omega;\R^d)$, where $\omega \subset \R^d$ is an arbitrary bounded Lipschitz domain (by standard covering arguments it suffices to check this for one particular choice of $\omega$ only). Notice that if $h$ is upper semicontinuous and has linear growth at infinity, we may replace the space $\Crm_c^\infty(\omega;\R^d)$ by $\LD_0(\omega)$ ($\LD$-functions with zero boundary values in the sense of trace) in the above definition, see~\cite[Remark~3.2]{BaFoTo00RTSF}. Section~4 of~\cite{Ebob00LSIF} contains an example of a symmetric-quasiconvex function that is not convex.

Using one-directional oscillations one can prove that if the function $h \colon \R_\sym^{d \times d} \to \R$ is symmetric-quasiconvex, then it holds that
\begin{equation} \label{eq:sym_qc_aodotb_conv}
  h( \theta A_1 + (1-\theta) A_2 ) \leq \theta h(A_1) + (1-\theta) h(A_2)
\end{equation}
whenever $A_1,A_2 \in \R_\sym^{d \times d}$ with $A_2 - A_1 = a \odot b$ for some $a,b \in \R^d$ and $\theta \in [0,1]$; also see Proposition~3.4 in~\cite{FonMul99AQLS} for a more general statement in the framework of $\Acal$-quasiconvexity.

If we consider $\R_\sym^{d \times d}$ to be identified with $\R^{d(d+1)/2}$ and $h \colon \R_\sym^{d \times d} \to \R$ with $\tilde{h} \colon \R^{d(d+1)/2} \to \R$, then the convexity in~\eqref{eq:sym_qc_aodotb_conv} implies that $\tilde{h}$ is separately convex and so, by a well-known result, even locally Lipschitz, see for example Lemma~2.2 in~\cite{BaKiKr00RQE}. If additionally $h$ has linear growth at infinity, then the formula from \emph{loc.~cit.} even implies that $h$ is globally Lipschitz. In particular,~\eqref{eq:f_sharp_def} becomes
\[
  h^\#(A) := \limsup_{t \to \infty} \frac{h(tA)}{t},  \qquad A \in \R_\sym^{d \times d}.
\]
Likewise, for $f \colon \cl{\Omega} \times \R_\sym^{d \times d} \to \R$ that is symmetric-quasiconvex in its second variable and has linear growth at infinity, the definition of the recession function $f^\infty$ from~\eqref{eq:f_infty_def} simplifies to
\begin{equation} \label{eq:f_infty_def_Lip}
  \qquad f^\infty(x,A) := \lim_{\substack{\!\!\!\! x' \to x \\ \; t \to \infty}}
    \frac{f(x',tA)}{t},  \qquad \text{$x \in \cl{\Omega}$, $A \in \R_\sym^{d \times d}$.}
\end{equation}

Notice that from Fatou's Lemma we get that the recession function $f^\#$, and hence also $f^\infty$ if it exists, is symmetric-quasiconvex whenever $f$ is, this is completely analogous to the situation for ordinary quasiconvexity. Hence, $f^\#$ and $f^\infty$ are also continuous on $\R_\sym^{d \times d} \setminus \{0\}$ in this situation.

\subsection{Young measures} \label{ssc:YM}

Generalized Young measures were introduced by DiPerna and Majda in~\cite{DiPMaj87OCWS}, we here follow the framework of~\cite{KriRin10CGGY}, which itself is based upon Alibert and Bouchitt\'{e}'s reformulation~\cite{AliBou97NUIG} of the theory.

A \term{(generalized) Young measure} on the open set $\Omega \subset \R^d$ and with values in $\R^N$ is a triple $(\nu_x,\lambda_\nu,\nu_x^\infty)$ consisting of
\begin{itemize}
  \item[(i)] a parametrized family of probability measures $(\nu_x)_{x \in \cl{\Omega}} \subset \Mbf^1(\R^N)$,
  \item[(ii)] a positive finite measure $\lambda_\nu \in \Mbf^+(\cl{\Omega})$ and
  \item[(iii)] a parametrized family of probability measures $(\nu_x^\infty)_{x \in \cl{\Omega}} \subset \Mbf^1(\Sbb^{N-1})$.
\end{itemize}
Moreover, we require that
\begin{itemize}
  \item[(iv)] the map $x \mapsto \nu_x$ is weakly* measurable with respect to $\Lcal^d$, i.e.\ the function $x \mapsto \dpr{f(x,\frarg),\nu_x}$ is $\Lcal^d$-measurable for all bounded Borel functions $f \colon \cl{\Omega} \times \R^N \to \R$,
  \item[(v)] the map $x \mapsto \nu_x^\infty$ is weakly* measurable with respect to $\lambda_\nu$, and
  \item[(vi)] $x \mapsto \dprn{\abs{\frarg},\nu_x} \in \Lrm^1(\Omega)$.
\end{itemize}
The set $\Ybf(\Omega;\R^N)$ contains all these Young measures. Similarly, we define the space $\Ybf_\loc(\R^d;\R^N)$, but with $\lambda_\nu$ only a local measure and $x \mapsto \dprn{\abs{\frarg},\nu_x} \in \Lrm_\loc^1(\Omega)$.

The \term{duality product} between a function $f \in \Ebf(\Omega;\R^N)$ and a Young measure $\nu \in \Ybf(\Omega;\R^N)$, or $f \in \Ebf_c(\R^d;\R^N)$ and $\nu \in \Ybf_\loc(\R^d;\R^N)$, is given by
\[
  \ddprb{f,\nu} := \int \dprb{f(x,\frarg),\nu_x} \dd x
  + \int \dprb{f^\infty(x,\frarg),\nu_x^\infty} \dd \lambda_\nu(x).
\]
Via this duality product, the space $\Ybf(\Omega;\R^N)$ is part of the dual space to $\Ebf(\Omega;\R^N)$, and hence we say that a sequence of Young measures $(\nu_j) \subset \Ybf(\Omega;\R^N)$ \term{converges weakly*} to $\nu \in \Ybf(\Omega;\R^N)$ if $\ddprn{f,\nu_j} \to \ddprn{f,\nu}$ for all $f \in \Ebf(\Omega;\R^N)$. In $\Ybf_\loc(\R^d;\R^N)$, we use weak* convergence relative to $\Ebf_c(\R^d;\R^N)$, i.e.\ $\nu_j \toweakstar \nu$ in $\Ybf_\loc(\R^d;\R^N)$ if $\ddprn{f,\nu_j} \to \ddprn{f,\nu}$ for all $f \in \Ebf_c(\R^d;\R^N)$.

Fundamental for all Young measure theory are the following two compactness statements, for which a proof can be found in~\cite[Corollary~2]{KriRin10CGGY} (the proof only covers (i), but easily generalizes to (ii) as well):

\begin{lemma}[Compactness] \label{lem:YM_compactness}
The following two statements are true:
\begin{itemize}
  \item[(i)] Let $(\nu_j) \subset \Ybf(\Omega;\R^N)$ be a sequence of Young measures satisfying
\[
  \qquad \supmod_j \, \ddprb{\ONE \otimes \abs{\frarg}, \nu_j} < \infty.
\]
Then, there exists a subsequence (not relabeled) with $\nu_j \toweakstar \nu$ in $\Ybf(\Omega;\R^N)$.
  \item[(ii)] Let $(\nu_j) \subset \Ybf_\loc(\R^d;\R^N)$ be a sequence of Young measures satisfying
\[
  \qquad \supmod_j \, \ddprb{\phi \otimes \abs{\frarg}, \nu_j} < \infty
  \qquad\text{for all $\phi \in \Crm_c(\R^d)$.}
\]
Then, there exists a subsequence (not relabeled) with $\nu_j \toweakstar \nu$ in $\Ybf_\loc(\R^d;\R^N)$.
\end{itemize}
\end{lemma}

As proved in Lemma~3 of~\cite{KriRin10CGGY}, there exists a countable set of functions $\{f_k\} = \setn{ \phi_k \otimes h_k \in \Crm(\cl{\Omega} \times \R^N) }{ k \in \N } \subset \Ebf(\Omega;\R^N)$ such that $\ddprn{f_k,\nu_1} = \ddpr{f_k,\nu_2}$ for two Young measures $\nu_1,\nu_2 \in \Ybf(\Omega;\R^N)$ and all $k \in \N$ implies $\nu_1 = \nu_2$. A similar statement holds in $\Ybf_\loc(\R^d;\R^N)$, but this time with $\{f_k\} \subset \Ebf_c(\R^d;\R^N)$, i.e.\ $\phi_k \in \Crm_c(\R^d)$. An immediate consequence is that to uniquely identify the limit in the weak* convergence $\nu_j \toweakstar \nu$ in $\Ebf(\Omega;\R^N)$, it suffices to test with the collection $\{f_k\}$, we say that the $f_k$ \enquote{determine} the Young measure convergence.

Each measure $\mu \in \Mbf(\cl{\Omega};\R^N)$ with Lebesgue--Radon--Nikod\'{y}m decomposition $\mu = a \Lcal^d \restrict \Omega + p \abs{\mu^s}$, where $a \in \Lrm^1(\Omega)$, $p \in \Lrm^1(\cl{\Omega},\abs{\mu^s};\Sbb^{N-1})$, induces an \term{elementary Young measure} $\epsilon_\mu \in \Ybf(\Omega;\R^N)$ through
\[
  (\epsilon_\mu)_x := \delta_{a(x)},  \qquad \lambda_{\epsilon_\mu} := \abs{\mu^s},
    \qquad (\epsilon_\mu)_x^\infty := \delta_{p(x)}.
\]
If $\epsilon_{\mu_j} \toweakstar \nu$ in $\Ybf(\Omega;\R^N)$, then we say that the $\mu_j$ \term{generate} $\nu$ and we write $\mu_j \toY \nu$. Similarly, if $\epsilon_{\mu_j} \toweakstar \nu$ in $\Ybf_\loc(\Omega;\R^N)$, then we write $\mu_j \toY \nu$, the ambient space being clear from the context.

For a Young measure $\nu \in \Ybf(\Omega;\R^N)$, we define its \term{barycenter} $[\nu] \in \Mbf(\Omega;\R^N)$ to be
\[
  [\nu] := \dprb{\id,\nu_x} \, \Lcal^d + \dprb{\id,\nu_x^\infty} \, \lambda_\nu,
\]
and yet again similarly for $\nu \in \Ybf_\loc(\R^d;\R^N)$ (now of course $[\nu] \in \Mbf_\loc(\R^d;\R^N)$). Clearly, weak* convergence of Young measures implies the corresponding weak* convergence of the barycenters.

A Young measure $\nu \in \Ybf(\Omega;\R_\sym^{d \times d})$ (so $\R^N$ is replaced by $\R_\sym^{d \times d}$ in the definitions above) is called a \term{BD-Young measure}, in symbols $\nu \in \BDY(\Omega)$, if it can be generated by a sequence of elementary Young measures corresponding to symmetrized derivatives. That is, for all $\nu \in \BDY(\Omega)$, there exists a sequence $(u_j) \subset \BD(\Omega)$ with $Eu_j \toY \nu$. It is easy to see that for a BD-Young measure $\nu \in \BDY(\Omega)$, there exists $u \in \BD(\Omega)$ satisfying $Eu = [\nu]$, this $u$ is called the \term{underlying deformation} of $\nu$. Similarly, define $\BDY_\loc(\R^d)$ by replacing $\Ybf(\Omega;\R_\sym^{d \times d})$ and $\BD(\Omega)$ by their respective local counterparts in the previous definition. When working with $\BDY(\Omega)$ or $\BDY_\loc(\R^d)$, the appropriate spaces of integrands are $\Ebf(\Omega;\R_\sym^{d \times d})$ and $\Ebf_c(\R^d;\R_\sym^{d \times d})$, respectively, since it is clear that both $\nu_x$ and $\nu_x^\infty$ only take values in $\R_\sym^{d \times d}$ whenever $\nu \in \BD(\Omega)$ or $\nu \in \BD_\loc(\R^d)$.

On several occasions we will invoke the following lemma about boundary adjustments, see Lemma~4 of~\cite{KriRin10CGGY} for the corresponding result in $\BV$ (the proof is the same).

\begin{lemma} \label{lem:boundary_adjust}
Let $\nu \in \BDY(\Omega)$ be a BD-Young measure with $\lambda_\nu(\partial \Omega) = 0$ and barycenter $[\nu] = Eu$, where $u \in \BD(\Omega)$. Then, there exists a generating sequence $(v_j) \subset (\Wrm^{1,1} \cap \Crm^{\infty})(\Omega;\R^d)$ with $Ev_j \toY \nu$, and $v_j|_{\partial \Omega} = u|_{\partial \Omega}$ (in the sense of trace) for all $j \in \N$.
\end{lemma}

Finally, we also mention the following results on \enquote{extended representation}, which can be found in Proposition~2 of~\cite{KriRin10CGGY}: Let $\nu_j \to \nu$ in $\Ybf(\Omega;\R^N)$. Then, also for $g(x,A) := \ONE_B(x) f(x,A)$, where $f \in \Ebf(\Omega;\R^N)$ and $B \subset \Omega$ is a Borel set, it holds that
\begin{equation} \label{eq:ext_repr}
  \ddprb{g,\nu_j} \to \ddprb{g,\nu}  \qquad\text{as long as $(\Lcal^d + \lambda_\nu)(\partial B) = 0$.}
\end{equation}
Moreover, even for a Carath\'{e}odory function $f \colon \cl{\Omega} \times \R^N \to \R$ such that the recession function $f^\infty$ exists in the sense of~\eqref{eq:f_infty_def} and is (jointly) continuous on $\cl{\Omega} \times \R^N$, we have
\begin{equation} \label{eq:ext_repr_Caratheodory}
  \ddprb{f,\nu_j} \to \ddprb{f,\nu}.
\end{equation}
Similar statements hold for $\Ybf_\loc(\R^d;\R^N)$ if we additionally assume that $B \subset \R^d$ is relatively compact.

\section{Localization principles for Young measures} \label{sc:localization}

This section presents two localization principles for Young measures, one at regular and one at singular points. These results are essentially adaptations of Propositions~4.1 and~4.2 in~\cite{Rind10?LSYM}, but some modifications had to be incorporated owing to the different notion of tangent measures employed here.

Notice that the following two propositions are formulated in the space $\BDY(\Omega)$ for convenience only. Since taking tangent Young measures is a local operation, they clearly also hold in $\BDY_\loc(\R^d)$ and in fact also in $\Ybf_\loc(\R^d;\R^{m \times d})$ by an obvious generalization.

\subsection{Localization principle at regular points}

We start with \enquote{regular} points.

\begin{proposition}[Localization at regular points] \label{prop:localize_reg}
Let $\nu \in \BDY(\Omega)$ be a BD-Young measure. Then, for $\Lcal^d$-almost every $x_0 \in \Omega$ there exists a \term{regular tangent Young measure} $\sigma \in \BDY_\loc(\R^d)$ satisfying
\begin{align}
  [\sigma] &\in \Tan([\nu],x_0),  &\sigma_y &= \nu_{x_0} \quad\text{a.e.,}  \label{eq:loc_reg_1} \\
  \lambda_\sigma &= \frac{\di \lambda_\nu}{\di \Lcal^d}(x_0) \, \Lcal^d  \in \Tan(\lambda_\nu,x_0),
    &\sigma_y^\infty &= \nu_{x_0}^\infty \quad\text{a.e.} \label{eq:loc_reg_2}
\end{align}
In particular, for all compact sets $K \subset \R^d$ with $\Lcal^d(\partial K) = 0$, and all $h \in \Crm(\R^{d \times d})$ such that the recession function $h^\infty$ exists in the sense of~\eqref{eq:f_infty_def}, it holds that
\begin{equation} \label{eq:loc_reg_ddpr}
  \ddprb{\ONE_K \otimes h, \sigma} = \biggl[ \dprb{h,\nu_{x_0}} + \dprb{h^\infty,\nu_{x_0}^\infty}
    \frac{\di \lambda_\nu}{\di \Lcal^d}(x_0) \biggr] \abs{K}.
\end{equation}
\end{proposition}

\begin{proof}
Take a set $\{ \phi_k \otimes h_k \} \subset \Ebf_c(\R^d;\R^{d \times d})$ determining the (local) Young measure convergence as in Section~\ref{ssc:YM}, and let $x_0 \in \R^d$ be as follows:
\begin{itemize}
  \item[(i)] There exists a sequence $r_n \todown 0$ such that (with $T^{(x_0,r)}(x) := (x-x_0)/r$)
\[
  \qquad \gamma_n := r_n^{-d} T_*^{(x_0,r_n)} [\nu] \quad\toweakstar\quad
  \frac{\di [\nu]}{\di \Lcal^d}(x_0) \, \Lcal^d
  \quad \in \Tan([\nu],x_0).
\]
  \item[(ii)] It holds that
\[
  \qquad \lim_{r \todown 0} \frac{\lambda_\nu^s(B(x_0,r))}{r^d} = 0  \qquad\text{and}\qquad
  \frac{\di \lambda_\nu}{\di \Lcal^d}(x_0) \, \Lcal^d \in \Tan(\lambda_\nu,x_0).
\]
  \item[(iii)] The point $x_0$ is a Lebesgue point for the functions
\[
  \qquad x \mapsto \dprb{h_k,\nu_x} + \dprb{h_k^\infty,\nu_x^\infty} \frac{\di \lambda_\nu}{\di \Lcal^d}(x),
  \qquad k \in \N.
\]
\end{itemize}
By results recalled in Section~\ref{ssc:Tan} and standard results in measure theory, the above three conditions can be satisfied simultaneously at $\Lcal^d$-almost every $x_0 \in \R^d$.

Take a BD-norm bounded generating sequence $(u_j) \subset \LD(\Omega)$ for $\nu$, i.e.\ $Eu_j \toY \nu$ (see for instance Lemma~\ref{lem:boundary_adjust}) and denote by $\tilde{u}_j \in \BD(\R^d)$ the extension by zero of $u_j$ onto all of $\R^d$. For each $n \in \N$ set
\[
  v_j^{(n)}(y) := \frac{\tilde{u}_j(x_0 + r_n y)}{r_n},  \qquad y \in \R^d.
\]
Testing with $\phi \in \Crm_c^1(\R^d)$, we perform a change of variables to see for all $k,l = 1,\ldots,d$,
\begin{align*}
  &\int \phi \dd \bigl[ Ev_j^{(n)} \bigr]_l^k
    = \frac{1}{2} \int \phi \dd \bigl[ \partial_k (v_j^{(n)})^l + \partial_l (v_j^{(n)})^k \bigr] \\
  &\qquad = - \frac{1}{2r_n} \int \partial_k \phi(y) \cdot \tilde{u}_j^l(x_0 + r_n y)
    + \partial_l \phi(y) \cdot \tilde{u}_j^k(x_0 + r_n y) \dd y \\
  &\qquad= - \frac{1}{2r_n^{d+1}} \int \partial_k \phi \Bigl( \frac{x-x_0}{r_n} \Bigr) \cdot \tilde{u}_j^l(x)
    + \partial_l \phi \Bigl( \frac{x-x_0}{r_n} \Bigr) \cdot \tilde{u}_j^k(x) \dd x \\
  &\qquad = \frac{1}{2r_n^d} \int \phi \Bigl( \frac{\frarg-x_0}{r_n} \Bigr)
    \dd \bigl( \partial_k \tilde{u}_j^l + \partial_l \tilde{u}_j^k \bigr)
    = \frac{1}{r_n^d} \int \phi \dd \bigl[ T_*^{(x_0,r_n)} E\tilde{u}_j \bigr]_l^k.
\end{align*}
Thus, also employing~\eqref{eq:pushforward_density_trafo},
\begin{align*}
  Ev_j^{(n)} &= r_n^{-d} T_*^{(x_0,r_n)} E \tilde{u}_j \\
  &= \Ecal u_j(x_0 + r_n \frarg) \, \Lcal^d + r_n^{-1} \bigl( u_j(x_0 + r_n \frarg)|_{\partial \Omega_n}
    \odot n_{\Omega_n} \bigr) \, \Hcal^{d-1} \restrict \partial \Omega_n,
\end{align*}
where $\Omega_n := r_n^{-1} (\Omega - x_0)$, $u_j(x_0 + r_n \frarg)|_{\partial \Omega_n}$ is the (inner) trace of the function $y \mapsto u_j(x_0 + r_n y)$ onto $\partial \Omega_n$, and $n_{\Omega_n} \colon \partial \Omega_n \to \Sbb^{d-1}$ is the unit inner normal to $\partial \Omega_n$.

We can use the previous formula together with a Poincar\'{e} inequality in $\BD$ and the boundedness of the $\BD$-trace operator, see Section~\ref{ssc:BD}, to get
\begin{equation} \label{eq:uj_Poincare_est}
\begin{aligned}
  \normb{v_j^{(n)}}_{\BD(\R^d)} &\leq C(n) \absb{Ev_j^{(n)}}(\R^d) = C(n) \absb{E \tilde{u}_j}(\R^d) \\
  &\leq C(n) \norm{u_j}_{\BD(\Omega)},
\end{aligned}
\end{equation}
where $C(n)$ absorbs all $n$-dependent constants (including $r_n^{-d}$). For fixed $n$, this last expression is $j$-uniformly bounded. Hence, we may select a subsequence of the $j$s (not explicitly named and depending on $n$) such that the sequence $(Ev_j^{(n)})_j$ generates a Young measure $\sigma^{(n)} \in \BDY(\R^d)$.

For every $\phi \otimes h \in \Ebf_c(\R^d;\R^{d \times d})$ let $n \in \N$ be so large that $\supp \phi \subset\subset \Omega_n$ (then the boundary measure in $Ev_j^{(n)}$ can be neglected), and calculate
\begin{align*}
  \ddprb{\phi \otimes h,\sigma^{(n)}}
    &= \lim_{j \to \infty} \int \phi(y) h \bigl( \Ecal v_j^{(n)}(y) \bigr) \dd y \\
  &= \lim_{j \to \infty} \int \phi(y) h \bigl( \Ecal u_j(x_0 + r_n y) \bigr) \dd y \\
  &= \lim_{j \to \infty} \frac{1}{r_n^d} \int \phi \Bigl( \frac{x - x_0}{r_n} \Bigr)
    h \bigl( \Ecal u_j(x) \bigr) \dd x \\
  &= \frac{1}{r_n^d} \, \ddprB{\phi \Bigl( \frac{\frarg- x_0}{r_n} \Bigr) \otimes h, \nu}.
\end{align*}
First, we examine the regular part of the last expression:
\begin{align*}
  &\frac{1}{r_n^d} \int \phi \Bigl( \frac{x - x_0}{r_n} \Bigr) \biggl[ \dprb{h,\nu_x}
    + \dprb{h^\infty,\nu_x^\infty} \frac{\di \lambda_\nu}{\di \Lcal^d}(x) \biggr] \dd x \\
  &\qquad = \int \phi(y) \biggl[ \dprb{h,\nu_{x_0 + r_n y}} + \dprb{h^\infty,\nu_{x_0 + r_n y}^\infty}
    \frac{\di \lambda_\nu}{\di \Lcal^d}(x_0 + r_n y) \biggr] \dd y,
\end{align*}
which as $n \to \infty$ (hence $r_n \todown 0$) converges to
\[
  \int \phi(y) \biggl[ \dprb{h,\nu_{x_0}} + \dprb{h^\infty,\nu_{x_0}^\infty}
    \frac{\di \lambda_\nu}{\di \Lcal^d}(x_0) \biggr] \dd y.
\]
The latter convergence first holds for the collection of $\phi_k \otimes h_k$ by the corresponding Lebesgue point properties of $x_0$, and then also for all $\phi \otimes h \in \Ebf_c(\R^d;\R^{d \times d}$) by density.

For the singular part, let $\beta > 0$ be so large that $\supp \phi \subset B(0,\beta)$ and observe by virtue of assumption~(ii) on $x_0$ that as $n \to \infty$,
\begin{align*}
  &\absBB{\frac{1}{r_n^d} \int \phi \Bigl( \frac{x - x_0}{r_n} \Bigr)
    \dprb{h^\infty,\nu_x^\infty} \dd \lambda_\nu^s(x) } 
  \leq M \norm{\phi}_\infty \cdot \frac{\lambda_\nu^s(B(x_0,\beta r_n))}{r_n^d}  \quad\to\quad 0,
\end{align*}
where $M := \sup \setn{\absn{h^\infty(A)}}{A \in \partial \Bbb^{d \times d}}$ is the linear growth constant of $h$.

In particular, we have proved so far that
\[
  \sup_{n \in \N} \, \absb{\ddprb{\phi \otimes \abs{\frarg}, \sigma^{(n)}}} < \infty
  \qquad\text{for all $\phi \in \Crm_c(\R^d)$.}
\]
Thus, by the Young measure compactness, see Lemma~\ref{lem:YM_compactness}~(ii), selecting a further subsequence if necessary, we may assume that $\sigma^{(n)} \toweakstar \sigma$ for some Young measure $\sigma \in \Ybf_\loc(\R^d;\R^{d \times d})$. From a diagonal argument we get that in fact $\sigma \in \BDY_\loc(\R^d)$. We also have $[\sigma] \in \Tan([\nu],x_0)$, because $[\sigma^{(n)}] = \gamma_n$ plus a jump part that moves out to infinity in the limit. This proves the first assertion in~\eqref{eq:loc_reg_1}.

Our previous considerations yield
\[
  \ddprb{\phi \otimes h,\sigma} = \int \phi(y) \biggl[ \dprb{h,\nu_{x_0}}
    + \dprb{h^\infty,\nu_{x_0}^\infty} \frac{\di \lambda_\nu}{\di \Lcal^d}(x_0) \biggr] \dd y
\]
for all $\phi \otimes h \in \Ebf_c(\R^d;\R^{d \times d})$. Varying first $\phi$ and then $h$, we see that $\sigma_y = \nu_{x_0}$ and $\sigma_y^\infty = \nu_{x_0}^\infty$ hold for $\Lcal^d$-almost every $y \in \R^d$, i.e.\ the second assertions in~\eqref{eq:loc_reg_1} and~\eqref{eq:loc_reg_2}, respectively. The first assertion from~\eqref{eq:loc_reg_2} follows, since the previous formula also implies $\lambda_\sigma = \frac{\di \lambda_\nu}{\di \Lcal^d}(x_0) \, \Lcal^d$ and the latter measure lies in $\Tan(\lambda_\nu,x_0)$. Finally, as an immediate consequence of~\eqref{eq:loc_reg_1} and~\eqref{eq:loc_reg_2}, in conjunction with~\eqref{eq:ext_repr}, we get~\eqref{eq:loc_reg_ddpr}. This concludes the proof.
\end{proof}

\subsection{Localization principle at singular points}

We now turn to \enquote{singular} points, i.e.\ points in the support of the singular part of the concentration measure $\lambda_\nu$ of a Young measure $\nu$.

\begin{proposition}[Localization at singular points] \label{prop:localize_sing}
Let $\nu \in \BDY(\Omega)$ be a BD-Young measure. Then, for $\lambda_\nu^s$-almost every $x_0 \in \Omega$, there exists a \term{singular tangent Young measure} $\sigma \in \BDY_\loc(\R^d)$ satisfying
\begin{align}
  [\sigma] &\in \Tan([\nu],x_0),  &\sigma_y &= \delta_0 \quad\text{a.e.,} \label{eq:loc_sing_1} \\
  \lambda_\sigma &\in \Tan(\lambda_\nu^s,x_0) \setminus \{0\}, &\sigma_y^\infty &= \nu_{x_0}^\infty
    \quad\text{$\lambda_\sigma$-a.e.} \label{eq:loc_sing_2}
\end{align}
In particular, for all bounded open sets $U \subset \R^d$ with $(\Lcal^d + \lambda_\sigma)(\partial U) = 0$ and all positively $1$-homogeneous $g \in \Crm(\R_\sym^{d \times d})$ it holds that
\begin{equation} \label{eq:loc_sing_ddpr}
  \ddprb{\ONE_U \otimes g, \sigma} = \dprb{g,\nu_{x_0}^\infty} \, \lambda_\sigma(U).
\end{equation}
\end{proposition}

\begin{proof}
Take a dense and countable set $\{ g_k \} \subset \Crm(\partial\Bbb^{d \times d})$ and consider all $g_k$ to be extended to $\R^{d \times d}$ by positive $1$-homogeneity. Then, let $x_0 \in \supp \lambda_\nu^s$ be such that:
\begin{itemize}
  \item[(i)] There exist sequences $r_n \todown 0$, $c_n > 0$ and $\lambda_0 \in \Tan(\lambda_\nu^s,x_0) \setminus \{0\}$ such that
\begin{equation} \label{eq:lambda_TM}
  \qquad c_n T_*^{(x_0,r_n)} \lambda_\nu^s  \quad\toweakstar\quad  \lambda_0.
\end{equation}
  \item[(ii)] It holds that
\begin{equation} \label{eq:singular_blowup}
  \qquad \lim_{r \todown 0} \frac{1}{\lambda_\nu^s(B(x_0,r))} \int_{B(x_0,r)} 1 + \dprb{\abs{\frarg},\nu_x}
    + \frac{\di \lambda_\nu}{\di \Lcal^d}(x) \dd x = 0.
\end{equation}
  \item[(iii)] The point $x_0$ is a $\lambda_\nu^s$-Lebesgue point for the functions
\[
  \qquad x \mapsto \dprb{\id,\nu_x^\infty}  \qquad\text{and}\qquad
  x \mapsto \dprb{g_k,\nu_x^\infty},  \quad k \in \N.
\]
\end{itemize}
By the usual measure-theoretic results and Preiss's existence theorem for non-zero tangent measures, see Theorem~2.5 in~\cite{Prei87GMDR} or the appendix, this can be achieved at $\lambda_\nu^s$-almost every $x_0 \in \Omega$.

The constants $c_n$ in~\eqref{eq:lambda_TM} can always be chosen as
\[
  c_n = c \bigl[ \lambda_\nu^s(\cl{B(x_0,Rr_n)}) \bigr]^{-1}  \qquad\text{($< \infty$)}
\]
for some fixed $R > 0$, $c > 0$, such that $\lambda_0(B(0,R)) > 0$, see Remark~14.4~(i) in~\cite{Matt95GSME} (recalled in Section~\ref{ssc:Tan}). Also notice that we may increase $R$ such that $\lambda_\nu^s(\partial B(x_0,Rr_n)) = 0$ for all $n$. In conjunction with~\eqref{eq:lambda_TM} this further yields for each $n \in \N$ the existence of a constant $\beta_N > 0$ satisfying
\[
  \limsup_{n\to\infty} \, c \cdot \frac{\lambda_\nu^s(B(x_0,Nr_n))}{\lambda_\nu^s(B(x_0,Rr_n))} \leq \beta_N.
\]
Combining this with~\eqref{eq:singular_blowup}, we get
\begin{align*}
  &\limsup_{n\to\infty} \, c_n \ddprb{\ONE_{B(x_0,Nr_n)} \otimes \abs{\frarg},\nu} \\
  &\qquad = \limsup_{n\to\infty} \Biggl[ \frac{c}{\lambda_\nu^s(B(x_0,Rr_n))} \int_{B(x_0,Nr_n)} \dprb{\abs{\frarg},\nu_x}
    + \frac{\di \lambda_\nu}{\di \Lcal^d}(x) \dd x \\
  &\qquad\qquad\qquad\qquad + \frac{c}{\lambda_\nu^s(B(x_0,Rr_n))} \, \lambda_\nu^s(B(x_0,Nr_n)) \Biggr] \\
  &\qquad\leq 0 + \beta_N.
\end{align*}
Hence, for all $N \in \N$,
\begin{equation} \label{eq:nu_blowup_bdd}
  \limsup_{n\to\infty} \, c_n \lambda_\nu^s(B(x_0,Nr_n))
  \leq \limsup_{n\to\infty} \, c_n \ddprb{\ONE_{B(x_0,Nr_n)} \otimes \abs{\frarg},\nu} \leq \beta_N
\end{equation}
Furthermore,
\[
  \limsup_{n\to\infty} \, \bigl( c_n T_*^{(x_0,r_n)} \abs{[\nu]} \bigr)(B(0,N)) \leq \beta_N,
  \qquad\text{for all $N \in \N$,}
\]
and so, taking a (non-relabeled) subsequence of the $r_n$, we may assume
\begin{equation} \label{eq:tau_TM}
  c_n T_*^{(x_0,r_n)} [\nu]  \quad\toweakstar\quad  \tau  \quad \in \Tan([\nu],x_0).
\end{equation}
Notice that $\tau$ might be the non-zero ($\tau \neq 0$ could only be ensured for $[\nu]$-almost every $x_0 \in \supp \, [\nu]$, but not necessarily for $\lambda_\nu^s$-almost every $x_0 \in \Omega$).

For a norm-bounded generating sequence $(u_j) \subset \LD(\Omega)$ of $\nu$, that is $Eu_j \toY \nu$, we denote by $\tilde{u} \in \BD(\R^d)$ the extension by zero, and set
\[
  v_j^{(n)}(y) := r_n^{d-1} c_n \tilde{u}_j(x_0 + r_n y),  \qquad y \in \R^d.
\]
We can then compute, similary to the localization principle for regular points,
\begin{align*}
  Ev_j^{(n)} &= c_n T_*^{(x_0,r_n)} E \tilde{u}_j \\
  &= r_n^d c_n \Ecal u_j(x_0 + r_n \frarg) \, \Lcal^d \\
  &\qquad + r_n^{d-1} c_n \bigl( u_j(x_0 + r_n \frarg)|_{\partial \Omega_n} \odot n_{\Omega_n} \bigr) \, \Hcal^{d-1} \restrict \Omega_n,
\end{align*}
where as before $\Omega_n := r_n^{-1}(\partial \Omega - x_0)$. Completely analogously to~\eqref{eq:uj_Poincare_est} we may also derive
\[
  \normb{v_j^{(n)}}_{\BD(\R^d)} \leq C(n) \norm{u_j}_{\BD(\Omega)}.
\]
The latter estimate implies that up to an $n$-dependent subsequence of $j$s, $(Ev_j^{(n)})_j$ generates a Young measure $\sigma^{(n)} \in \BDY(\R^d)$.

Let $g \in \Crm(\R^{d \times d})$ be positively $1$-homogeneous and let $\phi \in \Crm_c(\R^d)$. Then we have for all $n$ so large that $\supp \phi \subset\subset \Omega_n$ (and hence we may neglect the boundary jump part of $Ev_j^{(n)}$),
\begin{equation} \label{eq:sigma_n_trafo}
\begin{aligned}
  \ddprb{\phi \otimes g, \sigma^{(n)}} &= \lim_{j\to\infty} \int \phi(y) g \bigl( \Ecal v_j^{(n)}(y) \bigr) \dd y \\
  &= \lim_{j\to\infty} r_n^d c_n \int \phi(y) g \bigl( \Ecal u_j(x_0 + r_n y) \bigr) \dd y \\
  &= \lim_{j\to\infty} c_n \int \phi \Bigl( \frac{x-x_0}{r_n} \Bigr) g \bigl( \Ecal u_j(x) \bigr) \dd y \\
  &= c_n \, \ddprB{ \phi \Bigl( \frac{\frarg-x_0}{r_n} \Bigr) \otimes g, \nu }.
\end{aligned}
\end{equation}
For the regular part of the last expression, set $M := \sup \setn{\abs{g(A)}}{A \in \partial \Bbb^{d \times d}}$ and choose $N \in \N$ so large that $\supp \phi \subset B(0,N)$. Possibly increasing $n$ as to ensure
\[
  c_n \lambda_\nu^s(B(x_0,N r_n)) \leq \beta_N + 1,
\]
see~\eqref{eq:nu_blowup_bdd}, we have
\begin{equation} \label{eq:reg_part_est}
\begin{aligned}
  &\absBB{c_n \int \phi \Bigl( \frac{x-x_0}{r_n} \Bigr) \biggl[ \dprb{g,\nu_x} 
    + \dprb{g,\nu_x^\infty} \frac{\di \lambda_\nu}{\di \Lcal^d}(x) \biggr] \dd x } \\
  &\qquad \leq c_n M \norm{\phi}_\infty \int_{B(x_0,Nr_n)} \dprb{\abs{\frarg},\nu_x}
    + \frac{\di \lambda_\nu}{\di \Lcal^d}(x) \dd x \\
  &\qquad \leq \frac{M \norm{\phi}_\infty (\beta_N + 1)}{\lambda_\nu^s(B(x_0,Nr_n))}
    \int_{B(x_0,Nr_n)} \dprb{\abs{\frarg},\nu_x}
    + \frac{\di \lambda_\nu}{\di \Lcal^d}(x) \dd x \\
  &\to 0  \qquad\text{as $n \to \infty$,}
\end{aligned}
\end{equation}
the convergence following by virtue of~\eqref{eq:singular_blowup}. Hence, we get from~\eqref{eq:sigma_n_trafo},
\begin{equation} \label{eq:lim_sigma_n}
  \limsup_{n\to\infty} \, \ddprb{\phi \otimes g, \sigma^{(n)}} = \limsup_{n\to\infty} \,
  c_n \int \phi \Bigl( \frac{x-x_0}{r_n} \Bigr) \dprb{g,\nu_x^\infty} \dd \lambda_\nu^s(x).
\end{equation}
Taking $g = \abs{\frarg}$ in the previous equality,
\begin{align*}
  \limsup_{n\to\infty} \, \ddprb{\phi \otimes \abs{\frarg}, \sigma^{(n)}}
  &= \limsup_{n\to\infty} \, c_n \int \phi \Bigl( \frac{x-x_0}{r_n} \Bigr) \dd \lambda_\nu^s(x) \\
  &= \limsup_{n\to\infty} \, \int \phi \dd \bigl( c_n T_*^{(x_0,r_n)} \lambda_\nu^s \bigr)
    = \int \phi \dd \lambda_0,
\end{align*}
where the convergence follows from~\eqref{eq:lambda_TM}. In particular, $\ddprn{\phi \otimes \abs{\frarg}, \sigma^{(n)}}$ is uniformly bounded by $\norm{\phi}_\infty \lambda_0(\supp \phi)$, and hence by the Young measure compactness there exists a subsequence of the $r_n$s (not relabeled) with
\[
  \sigma^{(n)}  \quad\toweakstar\quad  \sigma  \qquad\text{in $\Ybf_\loc(\R^d;\R^{d \times d})$.}
\]
Again by a diagonal argument, we see $\sigma \in \BDY_\loc(\R^d)$. From~\eqref{eq:lim_sigma_n} we also get
\begin{equation} \label{eq:sigma_lim}
  \ddprb{\phi \otimes g, \sigma} = \lim_{n\to\infty} c_n \int \phi \Bigl( \frac{x-x_0}{r_n} \Bigr)
  \dprb{g,\nu_x^\infty} \dd \lambda_\nu^s(x).
\end{equation}

We now turn to the verification of~\eqref{eq:loc_sing_1} and~\eqref{eq:loc_sing_2}. The barycenters of the $\sigma^{(n)}$ satisfy
\[
  [\sigma^{(n)}] = c_n T_*^{(x_0,r_n)} [\nu] + \mu_n,
\]
where $\mu_n \in \Mbf(\partial \Omega_n;\R_\sym^{d \times d})$ are boundary measures satisfying $\mu_n \toweakstar 0$. Hence, by~\eqref{eq:tau_TM}, $[\sigma^{(n)}] \toweakstar \tau$ as $n \to \infty$ and so $[\sigma] = \tau \in \Tan([\nu],x_0)$, which is the first assertion of~\eqref{eq:loc_sing_1}.

For the second assertion of~\eqref{eq:loc_sing_1}, take cut-off functions $\phi \in \Crm_c(\R^d;[0,1])$, $\chi \in \Crm_c(\R_\sym^{d \times d};[0,1])$ and calculate similarly to~\eqref{eq:sigma_n_trafo},
\begin{align*}
  \ddprb{\phi \otimes \abs{\frarg}\chi(\frarg), \sigma^{(n)}} &= \lim_{j\to\infty}
    c_n \int \phi \Bigl( \frac{x-x_0}{r_n} \Bigr)
    \absb{ \Ecal u_j(x) } \chi \bigl( r_n^d c_n \Ecal u_j(x) \bigr) \dd x \\
  &= c_n \ddprB{ \phi \Bigl( \frac{\frarg-x_0}{r_n} \Bigr) \otimes \abs{\frarg} \chi(r_n^d c_n \frarg), \nu }.
\end{align*}
Then use a reasoning analogous to~\eqref{eq:reg_part_est} to see that the regular part of the previous expression converges to zero as $n \to \infty$. On the other hand, because $\chi$ has compact support in $\R_\sym^{d \times d}$, the singular part is identically zero. So we have shown
\[
  \ddprb{\phi \otimes \abs{\frarg}\chi(\frarg), \sigma} = 0
\]
for all $\phi, \chi$ as above. Hence, $\sigma_y = \delta_0$ for $\Lcal^d$-almost every $y \in \R^d$.

To see the first assertion from~\eqref{eq:loc_sing_2}, plug $g := \abs{\frarg}$ into~\eqref{eq:sigma_lim} and use $\sigma_y = \delta_0$ almost everywhere to derive for any $\phi \in \Crm_c(\R^d)$,
\begin{align*}
  \int \phi \dd \lambda_\sigma &= \ddprb{\phi \otimes \abs{\frarg}, \sigma}
    = \lim_{n\to\infty} c_n \int \phi \Bigl( \frac{x-x_0}{r_n} \Bigr) \dd \lambda_\nu^s(x) \\
  &= \lim_{n\to\infty} \, \int \phi \dd \bigl( c_n T_*^{(x_0,r_n)} \lambda_\nu^s \bigr)
    = \int \phi \dd \lambda_0,
\end{align*}
the last equality by~\eqref{eq:lambda_TM}. Hence, $\lambda_\sigma = \lambda_0 \in \Tan(\lambda_\nu^s,x_0)$.

We postpone the verification of the second assertion of~\eqref{eq:loc_sing_2} for a moment and instead turn to the verification of~\eqref{eq:loc_sing_ddpr} first. Let $U \subset \R^d$ be a bounded open set with $(\Lcal^d + \lambda_\sigma)(\partial U) = 0$. If $\lambda_\sigma(U) = 0$, then~\eqref{eq:loc_sing_ddpr} holds trivially, so assume $\lambda_\sigma(U) > 0$. Use $\phi = \ONE_U$ in~\eqref{eq:sigma_lim}, which is allowed by virtue of~\eqref{eq:ext_repr}, to get
\[
  \int_U \dprb{g,\sigma_y^\infty} \dd \lambda_\sigma(y) = \ddprb{\ONE_U \otimes g, \sigma}
  = \lim_{n\to\infty} c_n \int_{x_0 + r_n U} \dprb{g,\nu_x^\infty} \dd \lambda_\nu^s(x).
\]
Because $\lambda_\sigma \in \Tan(\lambda_\nu^s,x_0)$ and $\lambda_\sigma(U) > 0$, well-known results on tangent measures (see Section~\ref{ssc:Tan}) imply that $c_n = \tilde{c}(U) [\lambda_\nu^s(x_0 + r_n U)]^{-1}$ for some constant $\tilde{c}(U) > 0$. With this, the right hand side is
\begin{align*}
  \lim_{n\to\infty} c_n \int_{x_0 + r_n U} \dprb{g,\nu_x^\infty} \dd \lambda_\nu^s(x)
    &= \lim_{n\to\infty} \tilde{c}(U) \, \dashint_{x_0 + r_n U} \dprb{g,\nu_x^\infty} \dd \lambda_\nu^s(x) \\
  &= \tilde{c}(U) \, \dprb{g,\nu_{x_0}^\infty}
\end{align*}
by the Lebesgue point properties of $x_0$ (first ascertain this for the collection $\{g_k\}$ and then for the general case). Hence we have shown
\[
  \ddprb{\ONE_U \otimes g, \sigma} = \int_U \dprb{g,\sigma_y^\infty} \dd \lambda_\sigma(y) = \tilde{c}(U) \, \dprb{g,\nu_{x_0}^\infty}
\]
and testing this with $g = \abs{\frarg}$, we get $\tilde{c}(U) = \lambda_\sigma(U)$. Thus we have proved~\eqref{eq:loc_sing_ddpr}. But clearly, varying $U$ and $g$, this also implies $\sigma_y^\infty = \nu_{x_0}^\infty$ for $\lambda_\sigma$-almost every $y \in \R^d$, which is the second assertion of~\eqref{eq:loc_sing_2}.
\end{proof}

\section{Construction of good singular blow-ups} \label{sc:good_blowups}

This section combines the localization principles with rigidity arguments to show that among the possibly many singular tangent Young measures of a BD-Young measure $\nu \in \BDY(\Omega)$, there are always \enquote{good} ones at $\lambda_\nu^s$-almost every point $x_0 \in \Omega$. More concretely, we will construct blow-ups that are either affine or that are sums of one-directional functions, see Figure~\ref{fig:singular_blowups}. Some concrete differential inclusions involving the symmetrized gradient $\Ecal u$ in two dimensions are treated more elaborately in Section~\ref{ssc:rigidity_2D} for illustration purposes.

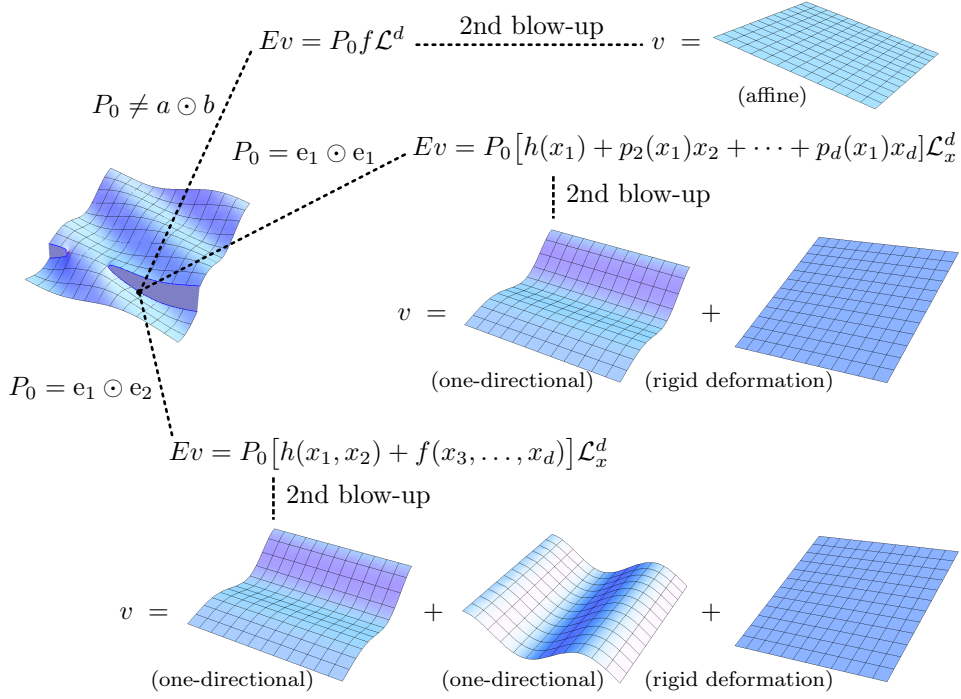
\begin{figure}[t]
\centering
\begingroup
  \makeatletter
  \providecommand\color[2][]{%
    \errmessage{(Inkscape) Color is used for the text in Inkscape, but the package 'color.sty' is not loaded}
    \renewcommand\color[2][]{}%
  }
  \providecommand\transparent[1]{%
    \errmessage{(Inkscape) Transparency is used (non-zero) for the text in Inkscape, but the package 'transparent.sty' is not loaded}
    \renewcommand\transparent[1]{}%
  }
  \providecommand\rotatebox[2]{#2}
  \ifx\svgwidth\undefined
    \setlength{\unitlength}{356.76689453pt}
  \else
    \setlength{\unitlength}{\svgwidth}
  \fi
  \global\let\svgwidth\undefined
  \makeatother
  \begin{picture}(1,0.78751704)%
    \put(0,0){\includegraphics[width=\unitlength]{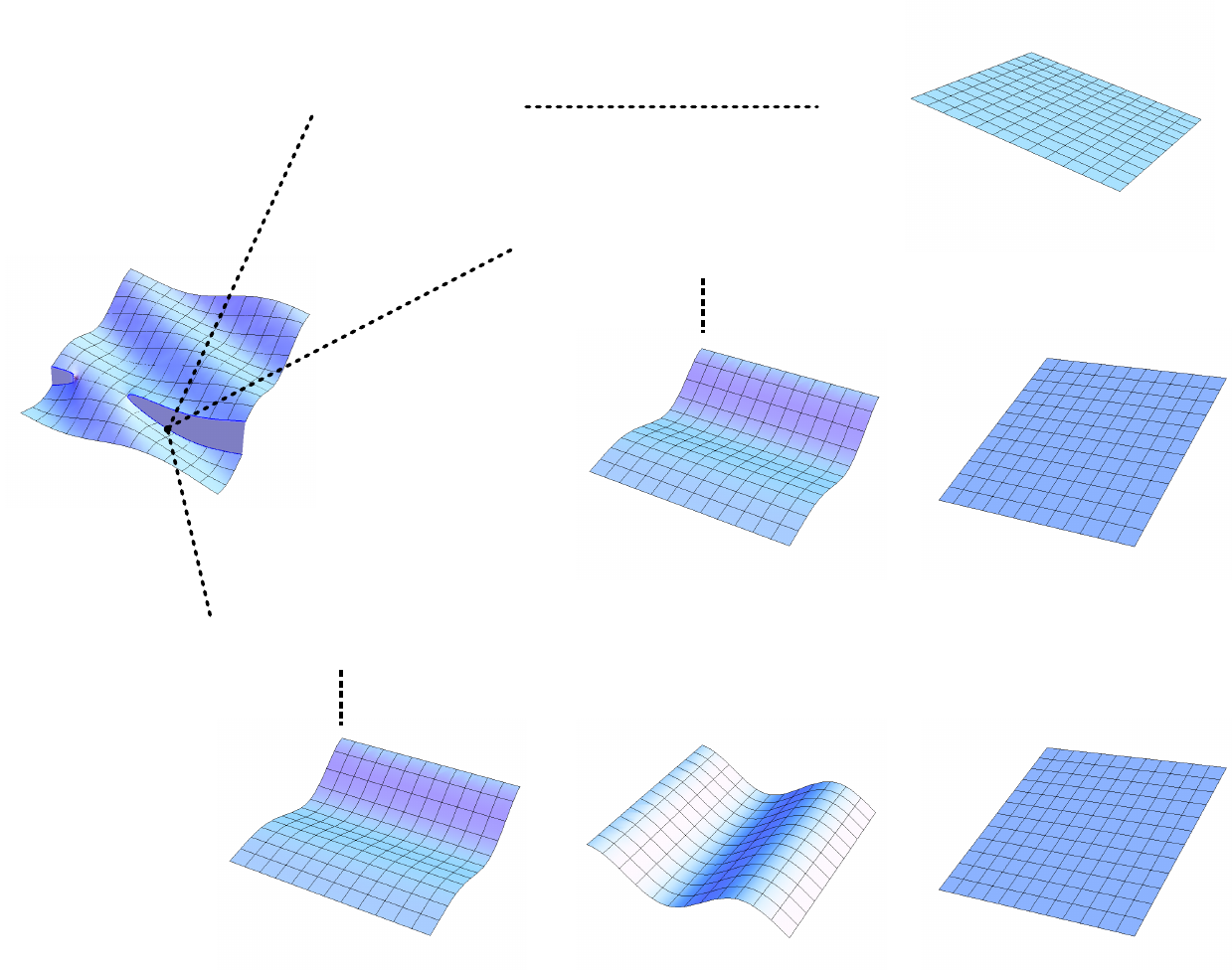}}%
    \put(1.80528226,-0.02045116){\color[rgb]{0,0,0}\makebox(0,0)[lb]{\smash{
}}}%
    \put(0.16706025,0.2610885){\color[rgb]{0,0,0}\makebox(0,0)[lb]{\smash{$Ev = P_0 \bigl[ h(x_1,x_2) + f(x_3,\ldots,x_d) \bigr] \Lcal^d_x$}}}%
    \put(0.42268935,0.58398841){\color[rgb]{0,0,0}\makebox(0,0)[lb]{\smash{$Ev = P_0 \bigl[ h(x_1) + p_2(x_1)x_2 + \cdots + p_d(x_1)x_d] \Lcal^d_x$}}}%
    \put(0.5841393,0.53465648){\color[rgb]{0,0,0}\makebox(0,0)[lb]{\smash{2nd blow-up}}}%
    \put(0.40699282,0.40908429){\color[rgb]{0,0,0}\makebox(0,0)[lb]{\smash{$v \; =$}}}%
    \put(0.23311273,0.57852902){\color[rgb]{0,0,0}\makebox(0,0)[lb]{\smash{$P_0 = \ee_1 \odot \ee_1$}}}%
    \put(0.43390115,0.09291146){\color[rgb]{0,0,0}\makebox(0,0)[lb]{\smash{$+$}}}%
    \put(0.72540802,0.09291146){\color[rgb]{0,0,0}\makebox(0,0)[lb]{\smash{$+$}}}%
    \put(0.45286407,0.02242361){\color[rgb]{0,0,0}\makebox(0,0)[lb]{\smash{{\footnotesize (one-directional)}}}}%
    \put(0.14790304,0.02242361){\color[rgb]{0,0,0}\makebox(0,0)[lb]{\smash{{\footnotesize (one-directional)}}}}%
    \put(0.08633526,0.62883562){\color[rgb]{0,0,0}\makebox(0,0)[lb]{\smash{$P_0 \neq a \odot b$}}}%
    \put(0.26123939,0.69610644){\color[rgb]{0,0,0}\makebox(0,0)[lb]{\smash{$Ev = P_0 f \Lcal^d$ }}}%
    \put(0.47202566,0.71225144){\color[rgb]{0,0,0}\makebox(0,0)[lb]{\smash{2nd blow-up}}}%
    \put(0.67383372,0.69610644){\color[rgb]{0,0,0}\makebox(0,0)[lb]{\smash{$v \; =$ }}}%
    \put(0.66588829,0.02242361){\color[rgb]{0,0,0}\makebox(0,0)[lb]{\smash{{\footnotesize (rigid deformation)}}}}%
    \put(0.11548595,0.09291146){\color[rgb]{0,0,0}\makebox(0,0)[lb]{\smash{$v \; =$}}}%
    \put(0.66588829,0.33859645){\color[rgb]{0,0,0}\makebox(0,0)[lb]{\smash{{\footnotesize (rigid deformation)}}}}%
    \put(0.76352815,0.64004742){\color[rgb]{0,0,0}\makebox(0,0)[lb]{\smash{{\footnotesize (affine)}}}}%
    \put(-0.0011168,0.33060167){\color[rgb]{0,0,0}\makebox(0,0)[lb]{\smash{$P_0 = \ee_1 \odot \ee_2$}}}%
    \put(0.44165227,0.33859645){\color[rgb]{0,0,0}\makebox(0,0)[lb]{\smash{{\footnotesize (one-directional)}}}}%
    \put(0.72540802,0.40908429){\color[rgb]{0,0,0}\makebox(0,0)[lb]{\smash{$+$}}}%
    \put(0.29039007,0.21624129){\color[rgb]{0,0,0}\makebox(0,0)[lb]{\smash{2nd blow-up}}}%
  \end{picture}%
\endgroup

\caption{Constructing good singular blow-ups.}
\label{fig:singular_blowups}
\end{figure}

\begin{theorem}[Good singular blow-ups] \label{thm:good_blowups}
Let $\nu \in \BDY(\Omega)$ be a BD-Young measure. For $\lambda_\nu^s$-almost every $x_0 \in \Omega$, there exists a singular tangent Young measure $\sigma \in \BDY_\loc(\R^d)$ as in Proposition~\ref{prop:localize_sing} such that additionally for any $v \in \BD_\loc(\R^d)$ with $Ev = [\sigma]$:
\begin{itemize}
  \item[(i)] If $\dprn{\id,\nu_{x_0}^\infty} \notin \setn{a \odot b}{a,b \in \R^d \setminus \{0\}}$ (this includes the case $\dprn{\id,\nu_{x_0}^\infty} = 0$), then $v$ is equal to an affine function almost everywhere.
  \item[(ii)] If $\dprn{\id,\nu_{x_0}^\infty} = a \odot b$ ($a,b \in \R^d \setminus \{0\}$) with $a \neq b$, then there exist functions $h_1,h_2 \in \BV_\loc(\R)$, $v_0 \in \R^d$, and a skew-symmetric matrix $R \in \R_\skw^{d \times d}$ such that
\[
  \qquad v(x) = v_0 + h_1(x \cdot a)b + h_2(x \cdot b)a + Rx,  \qquad \text{$x \in \R^d$ a.e.}
\]
  \item[(iii)] If $\dprn{\id,\nu_{x_0}^\infty} = a \odot a$ ($a \in \R^d \setminus \{0\}$), then there exists a function $h \in \BV_\loc(\R)$, $v_0 \in \R^d$ and a skew-symmetric matrix $R \in \R_\skw^{d \times d}$ such that
\[
  \qquad v(x) = v_0 + h(x \cdot a)a + Rx,  \qquad \text{$x \in \R^d$ a.e.}
\]
\end{itemize}
\end{theorem}

\begin{remark}
In contrast to the situation for the space $\BV$, where \emph{all} blow-ups could be shown to have a good structure, in $\BD$ we may only ascertain that there \emph{exists} at least one good blow-up. Moreover, in $\BV$ we know from Alberti's Rank-One Theorem~\cite{Albe93ROPD} that the case corresponding to~(i), that is $\frac{\di Du}{\di \abs{Du}}(x_0)$ cannot be written as a tensor product, in fact occurs only on a negligible set. However, no such theorem is available for $\BD$, so we need all cases of the above theorem.
\end{remark}

\begin{example}
Let $\Omega = (-1,1)^2 \subset \R^2$ and let
\[
  u := \begin{pmatrix} \ONE_{\{x_2 > 0\}} \\ \ONE_{\{x_1 > 0\}} \end{pmatrix}.
\]
Then, $u \in \BDY(\Omega)$ and
\[
  Eu = (\ee_1 \odot \ee_2) \bigl[ \Hcal^1 \restrict \{ x_1 = 0 \} + \Hcal^1 \restrict \{ x_2 = 0 \} \bigr].
\]
Hence, for the elementary BD-Young measure $\epsilon_{Eu}$ at the origin, case~(ii) of the preceding theorem is applicable; notice that indeed we need both $h_1$ and $h_2$ for the result to be true.
\end{example}

With the notation of the theorem we set
\[
  P_0 := \begin{cases}
    \frac{\dprn{\id,\nu_{x_0}^\infty}}{\absn{\dprn{\id,\nu_{x_0}^\infty}}} 
      & \text{if $\dprn{\id,\nu_{x_0}^\infty} \neq 0$,} \\
    0 & \text{if $\dprn{\id,\nu_{x_0}^\infty} = 0.$}
  \end{cases}
\]
The proof will be accomplished in the following three sections, its main scheme is shown in Figure~\ref{fig:singular_blowups}.

\subsection{The case $P_0 \neq a \odot b$}  \label{ssc:neq_a_odot_b}

The proof technique for this case consists of using Fourier multipliers and projections together with an iterated blow-up argument and is an adaptation of the idea for the proof of Lemma~2.7 in~\cite{Mull99VMMP}.

\begin{proof}[Proof of Theorem~\ref{thm:good_blowups}~(i)]
Take a singular tangent Young measure $\nu \in \BDY_\loc(\R^d)$ at a point $x_0 \in \Omega$ as in Proposition~\ref{prop:localize_sing} and let $v \in \BD_\loc(\R^d)$ with $Ev = [\sigma]$. This $v$ then satisfies (by the properties of singular tangent Young measures, see e.g.~\eqref{eq:loc_sing_ddpr})
\[
  Ev = P_0 \abs{Ev}.
\]

If $P_0 = 0$ (i.e.\ $\dprn{\id,\nu_{x_0}^\infty} = 0$), then we immediately have that $v$ is affine. Hence from now on we assume $P_0 \neq 0$.

\proofstep{Step~1.} Suppose first that $v$ is smooth. By assumption, $P_0 \neq a \odot b$ for any $a,b \in \R^d$. Let $A \colon \R^{d \times d} \to \R^{d \times d}$ be the orthogonal projection onto $(\spn\{P_0\})^\perp$. Then,
\begin{equation} \label{eq_AEv0}
  A(\Ecal v) \equiv 0.
\end{equation}
For every smooth cut-off function $\phi \in \Crm_c^\infty(\R^d;[0,1])$, the function $w := \phi v$ satisfies (here exceptionally considering $\nabla \phi$ as a column vector)
\[
  \Ecal w = \phi \Ecal v + v \odot \nabla \phi.
\]
Combining this with~\eqref{eq_AEv0}, we get
\begin{equation} \label{eq:Aew_f}
  A(\Ecal w) = A(v \odot \nabla \phi) =: f,
\end{equation}
where by means of an embedding result in $\BD$~\cite{TemStr80FBD}, $f \in \Lrm^{d/(d-1)}(\R^d;\R^{d \times d})$ ($\Lrm^\infty$ if $d=1$).

If for the Fourier transform $\hat{g}$ of a function $g \in \Lrm^1(\R^d;\R^N)$ we use the definition
\[
  \hat{g}(\xi) := \int g(x) \ee^{-2\pi\ii x \cdot \xi} \dd x,  \qquad \xi \in \R^d,
\]
then it can be checked easily that
\[
  \widehat{\Ecal w}(\xi) = (2\pi\ii) \, \hat{w}(\xi) \odot \xi.
\]
Hence, applying the Fourier transform to both sides of~\eqref{eq:Aew_f}, and considering $A$ to be identified with its complexification (that is, $A(M+\ii N) = AM + \ii AN$ for $M,N \in \R^{d \times d}$), we arrive at
\begin{equation} \label{eq:AeFw_Ff}
  (2\pi\ii) \, A(\hat{w}(\xi) \odot \xi) = \hat{f}(\xi)  \qquad \text{for all $\xi \in \R^d$.}
\end{equation}

\proofstep{Step~2.}
We will now use some linear algebra to rewrite~\eqref{eq:AeFw_Ff} as a Fourier multiplier equation and then apply a version of the Mihlin multiplier theorem.

Notice first that (the complexification of) the projection $A \colon \C^{d \times d} \to \C^{d \times d}$ has kernel $\spn\{P_0\}$ (here and in the following all spans are understood in $\C^{d \times d}$) and hence descends to the quotient
\[
  [A] \colon \C^{d \times d} / \spn\{P_0\} \to \ran A,
\]
and $[A]$ is an invertible linear map. Then, for $\xi \in \R^d \setminus \{0\}$, let
\[
  \bigl\{ P_0, \ee_1 \odot \xi, \ldots, \ee_d \odot \xi, G_{d+1}(\xi),
  \ldots, G_{d^2-1}(\xi) \bigr\} \quad\subset \R^{d \times d}
\]
be a basis of $\C^{d \times d}$ with the property that the matrices $G_{d+1}(\xi), \ldots, G_{d^2-1}(\xi)$ depend smoothly on $\xi$ and are positively $1$-homogeneous in $\xi$. For all $\xi \in \R^d \setminus \{0\}$, denote by $B(\xi) \colon \C^{d \times d} \to \C^{d \times d}$ the (non-orthogonal) projection with
\begin{align*}
  \ker B(\xi) &= \spn \{P_0\}, \\
  \ran B(\xi) &= \spn \bigl\{ \ee_1 \odot \xi, \ldots, \ee_d \odot \xi, G_{d+1}(\xi),
  \ldots, G_{d^2-1}(\xi) \bigr\}.
\end{align*}
If we interpret $\ee_1 \odot \xi, \ldots, \ee_d \odot \xi, G_{d+1}(\xi), \ldots, G_{d^2-1}(\xi)$ as vectors in $\R^{d^2}$, collect them into the columns of the matrix $X(\xi) \in \R^{d^2 \times (d^2-1)}$, and also let $Y \in \R^{d^2 \times (d^2-1)}$ be a matrix whose columns comprise an orthonormal basis of $(\spn\{P_0\})^\perp$, then $B(\xi)$ can be written explicitly as (it is elementary to see that $Y^T X(\xi)$ is invertible)
\[
  B(\xi) = X(\xi)(Y^T X(\xi))^{-1}Y^T.
\]
This implies that $B(\xi)$ is positively $0$-homogeneous, and using Cramer's Rule, we also see that $B(\xi)$ depends smoothly on $\xi \in \R^d \setminus \{0\}$ (if $\det (Y^T X(\xi))$ was not bounded away from zero for $\xi \in \Sbb^{d-1}$, then by compactness there would exist $\xi_0 \in \Sbb^{d-1}$ with $\det (Y^T X(\xi_0)) = 0$, a contradiction). Of course, also $B(\xi)$ descends to a quotient
\[
  [B(\xi)] \colon \C^{d \times d} / \spn\{P_0\} \to \ran B(\xi),
\]
which is now invertible. It is not difficult to see that $\xi \mapsto [B(\xi)]$ is still positively $0$-homogeneous and smooth in $\xi \neq 0$ (for example by utilizing the basis given above).

Since $\hat{w}(\xi) \odot \xi \in \ran B(\xi)$, we notice that $[B(\xi)]^{-1}(\hat{w}(\xi) \odot \xi) = [\hat{w}(\xi) \odot \xi]$, the equivalence class of $\hat{w}(\xi) \odot \xi$ in $\C^{d \times d} / \spn\{P_0\}$. This allows us to rewrite~\eqref{eq:AeFw_Ff} in the form
\[
  (2\pi\ii) \, [A] [B(\xi)]^{-1}(\hat{w}(\xi) \odot \xi) = \hat{f}(\xi),
\]
or equivalently as
\[
  (2\pi\ii) \, \hat{w}(\xi) \odot \xi = [B(\xi)] [A]^{-1} \hat{f}(\xi).
\]
The function $M \colon \R^d \setminus \{0\} \to \R^{d^2 \times d^2}$ given by $\xi \mapsto [B(\xi)] [A]^{-1}$ is smooth and positively $0$-homogeneous, and we have the multiplier equation
\[
  \widehat{\Ecal w}(\xi) = (2\pi\ii) \, \hat{w}(\xi) \odot \xi = M(\xi) \hat{f}(\xi).
\]

A matrix-version of the Mihlin Multiplier Theorem, see Theorem~6.1.6 in~\cite{BerLof76IS}, now yields
\begin{equation} \label{eq:Ew_L1star_estimate}
  \norm{\Ecal w}_{\Lrm^{d/(d-1)}} \leq C \norm{f}_{\Lrm^{d/(d-1)}}
  \leq C \norm{v}_{\Lrm^{d/(d-1)}(K;\R^d)},
\end{equation}
where $K := \supp \phi$ and $C = C(K,\norm{A},\norm{\nabla \phi}_\infty)$ is a constant.

\proofstep{Step~3.}
If $v \in \BD_\loc(\R^d)$ is not smooth, we take a family of mollifiers $(\rho_\delta)_{\delta > 0}$ and define by convolution $v_\delta := \rho_\delta \conv v \in \Crm^\infty(\R^d;\R^d)$. Correspondingly, with a fixed cut-off function $\phi$ as above we define $w_\delta := \phi v_\delta$. This mollification preserves the property $Ev = P_0 \abs{Ev}$ and so~\eqref{eq:Ew_L1star_estimate} gives for $\delta < 1$,
\[
  \norm{\Ecal w_\delta}_{\Lrm^{d/(d-1)}} \leq C \norm{v_\delta}_{\Lrm^{d/(d-1)}(K;\R^d)}
  \leq C \norm{v}_{\Lrm^{d/(d-1)}(K_1;\R^d)},
\]
where again $K := \supp \phi$ and $K_1 := K + \Bbb^d$.

Since $\Ecal w_\delta \Lcal^d \toweakstar Ew$ as $\delta \todown 0$, the previous $\delta$-uniform estimate implies that $Ew$ is absolutely continuous with respect to Lebesgue measure, $Ew = \Ecal w \Lcal^d$ for $\Ecal w \in \Lrm^{d/(d-1)}(K;\R_\sym^{d \times d})$. Finally, varying $\phi$, we get that also $Ev$ is absolutely continuous with respect to Lebesgue measure and $\Ecal v \in \Lrm_\loc^{d/(d-1)}(\R^d;\R_\sym^{d \times d})$.

\proofstep{Step~4.}
We have shown so far that $[\sigma] = Ev$ is absolutely continuous with respect to Lebesgue measure. Now apply Proposition~\ref{prop:localize_reg} and Preiss' existence result for non-zero tangent measures to $\sigma$ in order to infer the existence of a \emph{regular} tangent Young measure $\kappa$ to $\sigma$ at $\Lcal^d$-almost every point $y_0 \in \supp [\sigma]$ with $[\kappa] \neq 0$. It is not difficult to see that $\kappa$ is still a singular tangent measure to $\nu$ in the sense of Proposition~\ref{prop:localize_sing}. Indeed, one may observe first that~\eqref{eq:loc_sing_1},~\eqref{eq:loc_sing_2} with $\kappa$ in place of $\sigma$ still hold by the conclusion of Propositon~\ref{prop:localize_reg} and~\eqref{eq:loc_sing_1},~\eqref{eq:loc_sing_2} for $\sigma$ together with the fact that tangent measures to tangent measures are tangent measures, see Lemma~\ref{lem:TanTan_eq_Tan} (we need to select $x_0 \in \Omega$ according to that lemma, which is still possible $\lambda_\nu^s$-almost everywhere). Finally, we see that~\eqref{eq:loc_sing_ddpr} also holds with $\kappa$ in place of $\sigma$, because this assertion always follows from~\eqref{eq:loc_sing_1},~\eqref{eq:loc_sing_2}.

On the other hand, by the absolute continuity of $Ev$ with respect to $\Lcal^d$ and standard results on tangent measures, we may in fact choose $y_0$ such that $[\kappa] \in \Tan(Ev,y_0)$ is a constant multiple of Lebesgue measure, see Section~\ref{ssc:Tan}. Thus, any $\tilde{v} \in \BD_\loc(\R^d)$ with $E\tilde{v} = [\kappa]$ is affine. This shows the claim of Theorem~\ref{thm:good_blowups}~(i) with $\kappa$ in place of $\sigma$ and $\tilde{v}$ in place of $v$.
\end{proof}

\subsection{The case $P_0 = a \odot b$}  \label{ssc:a_odot_b}

This case is more involved, yet essentially elementary. We first examine the situation in two dimensions and then, via a dimension reduction lemma, extend the result to an arbitrary number of dimensions.

\begin{lemma}[2D rigidity] \label{lem:BD_rigidity_2D}
A function $u \in \BD_\loc(\R^2)$ satisfies
\begin{equation} \label{eq:Eu_P_2D}
  Eu = \frac{a \odot b}{\abs{a \odot b}} \abs{Eu}
  \qquad\text{for fixed $a,b \in \R^2 \setminus \{0\}$ with $a \neq b$,}
\end{equation}
if and only if $u$ has the form
\begin{equation} \label{eq:P_symtens_u}
  u(x) = h_1(x \cdot a)b + h_2(x \cdot b)a,  \qquad\text{$x \in \R^2$ a.e.,}
\end{equation}
where $h_1, h_2 \in \BV_\loc(\R)$.
\end{lemma}

Notice that we are only imposing a condition on the symmetric derivative, which only determines a function up to a rigid deformation. In the above case, however, since $a$ and $b$ are linearly independent, we may absorb this rigid deformation into $h_1$ and $h_2$.

\begin{proof}
By the chain rule in $\BV$, it is easy to see that all $u$ of the form~\eqref{eq:P_symtens_u} satisfy~\eqref{eq:Eu_P_2D}.

For the other direction, without loss of generality we suppose that $a = \ee_1$, $b = \ee_2$ (see Step~2 in the proof of Theorem~\ref{thm:good_blowups}~(ii) below for an explicit reduction; in fact, this lemma will only be used in the case $a = \ee_1$, $b = \ee_2$ anyway).

We will use a slicing result, Proposition~3.2 in~\cite{AmCoDa97FPFB}, which essentially follows from Fubini's Theorem: If for $\xi \in \R^2 \setminus \{0\}$ we define
\begin{align*}
  H_\xi &:= \setb{ x \in \R^2 }{ x \cdot \xi = 0 }, \\
  u_y^\xi(t) &:= \xi^T u(y + t\xi),  \qquad\text{where $t \in \R$, $y \in H_\xi$,}
\end{align*}
then the result in \emph{loc.~cit.\ }states
\begin{equation} \label{eq:Eu_slicing}
  \absb{\xi^T Eu \xi} = \int_{H_\xi} \absb{Du_y^\xi} \dd \Hcal^1(y)
  \qquad\text{as measures.}
\end{equation}
By assumption, $Eu = \sqrt{2} (\ee_1 \odot \ee_2) \abs{Eu}$ with $\abs{Eu} \in \Mbf_\loc(\R^2)$, so if we apply~\eqref{eq:Eu_slicing} for $\xi = \ee_1$, we get
\[
  0 = \frac{\sqrt{2}}{2} \absb{\ee_1^T \bigl(\ee_1 \ee_2^T
    + \ee_2 \ee_1^T \bigr) \ee_1} \, \abs{Eu}
  = \int_{H_\xi} \absb{\partial_t u^1(y + t\ee_1)} \dd \Hcal^1(y),
\]
where we wrote $u = (u^1,u^2)^T$. This yields $\partial_1 u^1 \equiv 0$ distributionally, whence $u^1(x) = h_2(x_2)$ for some $h_2 \in \Lrm_\loc^1(\R)$. Analogously, we find that $u^2(x) = h_1(x_1)$ with $h_1 \in  \Lrm_\loc^1(\R)$. Thus, we may decompose
\[
  u(x) = \begin{pmatrix} 0 \\ h_1(x_1) \end{pmatrix}
    + \begin{pmatrix} h_2(x_2) \\ 0 \end{pmatrix}
  = h_1(x \cdot \ee_1)\ee_2 + h_2(x \cdot \ee_2)\ee_1
\]
and it only remains to show that $h_1, h_2 \in \BV_\loc(\R)$. For this, fix $\eta \in \Crm_c^1(\R;[-1,1])$ with $\int \eta \dd t = 1$ and calculate for all $\phi \in \Crm_c^1(\R;[-1,1])$ by Fubini's Theorem,
\begin{align*}
  2 \int \phi \otimes \eta \dd(Eu)_2^1 &= - \int u^2 (\phi' \otimes \eta) \dd x - \int u^1 (\phi \otimes \eta') \dd x \\
  &= - \int h_1 \phi' \dd x_1 \cdot \int \eta \dd x_2 - \int u^1 (\phi \otimes \eta') \dd x.
\end{align*}
So, with $K := \supp \phi \times \supp \eta$,
\[
  \absBB{ \int h_1 \phi' \dd x } \leq 2 \abs{Eu}(K) + \norm{u^1}_{\Lrm^1(K)} \cdot \norm{\eta'}_\infty < \infty
\]
for all $\phi \in \Crm_c^1(\R)$ with $\norm{\phi}_\infty \leq 1$, hence $h_1 \in \BV_\loc(\R)$. Likewise, $h_2 \in \BV_\loc(\R)$, and we have shown the lemma.
\end{proof}

Next we need to extend the preceding rigidity lemma to an arbitrary number of dimensions. This is the purpose of the following lemma, which we only formulate for the case $P = \ee_1 \odot \ee_2$ to avoid notational clutter (we will only need this special case later).

\begin{lemma}[Dimension reduction] \label{lem:dim_reduction_e1e2}
Let $u \in \BD_\loc(\R^d)$ be such that
\[
  Eu = \sqrt{2} (\ee_1 \odot \ee_2) \abs{Eu}.
\]
Then, there exist a Radon measure $\mu \in \Mbf_\loc(\R^2)$ and a linear function $f \colon \R^{d-2} \to \R$ such that
\[
  Eu = \sqrt{2} (\ee_1 \odot \ee_2) \Bigl[ \mu \otimes \Lcal^{d-2} + f(x_3,\ldots,x_d) \Lcal^d_x \Bigr].
\]
\end{lemma}

\begin{proof}
In all of the following, let
\[
  P_0 := \sqrt{2} (\ee_1 \odot \ee_2).
\]

\proofstep{Step~1.}
We first assume that $u$ is smooth. In this case, there exists $g \in \Crm^\infty(\R^d)$ such that
\[
  \Ecal u = P_0 g  \qquad\text{and}\qquad  E^s u = 0.
\]
Clearly, 
\[
  \Ecal u(x)_k^j = \frac{\sqrt{2}}{2} \begin{cases}
    g(x)  & \text{if $(j,k) = (1,2)$ or $(j,k) = (2,1)$,} \\
    0     & \text{otherwise.}
  \end{cases}
\]

Fix $i \geq 3$. With
\[
  \Wcal u := \frac{1}{2} \bigl( \nabla u - \nabla u^T \bigr)
\]
we have from~\eqref{eq:Wu_identity},
\[
  \partial_k \Wcal u_j^i = \partial_j \Ecal u_k^i - \partial_i \Ecal u_k^j,
  \qquad\text{for $j,k = 1,\ldots,d$.}
\]
Since $i \geq 3$, only the second term is possibly non-zero, and so
\[
  \nabla \Wcal u(x)_j^i = - \partial_i \Ecal u(x)^j = -\frac{\sqrt{2}}{2} \begin{cases}
    (0, \partial_i g(x),0,\ldots,0)  & \text{if $j = 1$,} \\
    (\partial_i g(x),0,0,\ldots,0)  & \text{if $j = 2$,} \\
    (0,\ldots,0)  & \text{if $j \geq 3$.}
  \end{cases}
\]
It is elementary to see that if a function $h \in \Crm^\infty(\R^d)$ satisfies $\partial_k h \equiv 0$ for all $k = 2,\ldots,d$, then, with a slight abuse of notation, $h(x) = h(x_1)$ and also $\partial_1 h(x) = \partial_1 h(x_1)$. In our situation this gives that $\partial_i g$ can be written both as a function of $x_1$ only, and as a function of $x_2$ only. But this is only possible if $\partial_i g$ is constant, say $\partial_i g \equiv a_i \in \R$ for $i = 3,\ldots,d$.

If we set
\[
  f(x) := a_3 x_3 + \cdots + a_d x_d,
\]
we have that the function $h(x) := g(x) - f(x)$ only depends on the first two components $x_1,x_2$ of $x$, and thus
\[
  Eu = P_0 \Bigl[ h(x_1,x_2) \Lcal^d_x + f(x_3,\ldots,x_d) \Lcal^d_x \Bigr].
\]

\proofstep{Step~2.}
Now assume that only $u \in \BD_\loc(\R^d)$. We will reduce this case to the previous one by a smoothing argument. Set $u_\delta := \rho_\delta \conv u \in \Crm^\infty(\R^d;\R^d)$, where $(\rho_\delta)_{\delta > 0}$ is a family of mollifying kernels. It can be seen that $Eu_\delta = \sqrt{2} (\ee_1 \odot \ee_2) \abs{Eu_\delta}$ still holds, so we may apply the first step to get a smooth function $h_\delta \in \Crm^\infty(\R^2)$ and a linear function $f_\delta \colon \R^{d-2} \to \R$ such that
\[
  Eu_\delta = P_0 \Bigl[ h_\delta(x_1,x_2) \Lcal^d_x
  + f_\delta(x_3,\ldots,x_d) \Lcal^d_x \Bigr].
\]

We will show that also the limit has an analogous form: With the cube $Q^k(R) := (-R,R)^k$ ($R > 0$), take $\phi \in \Crm_c(Q^2(R);[-1,1])$, and define the measures
\[
  \mu_\delta := h_\delta(x_1,x_2) \Lcal^2_{(x_1,x_2)}  \quad\in \Mbf_\loc(\R^2).
\]
We have from Fubini's Theorem,
\[
  \int \phi \otimes \ONE_{Q^{d-2}(R)} \dd Eu_\delta = P_0 \Biggl[ (2R)^{d-2} \int \phi \dd \mu_\delta
  + \int \phi \dd x \cdot \int_{Q^{d-2}(R)} f_\delta \dd x \Biggr]
\]
The second term on the right hand side is identically zero since $f_\delta$ is linear and $Q^{d-2}(R)$ is symmetric, so, with a constant $C = C(R)$,
\[
  \limsup_{\delta \todown 0} \int \phi \dd \mu_\delta \leq C \limsup_{\delta \todown 0} \abs{Eu_\delta}(Q^d(R))
  < \infty.
\]
Therefore, selecting a subsequence of $\delta$s, we may assume that $\mu_\delta \toweakstar \mu \in \Mbf_\loc(\R^2)$, which entails $\mu_\delta \otimes \Lcal^{d-2} \toweakstar \mu \otimes \Lcal^{d-2}$.  Moreover, if $f_\delta \Lcal^d \toweakstar \gamma \in \Mbf_\loc(\R^d)$, then $\gamma$ must be of the form $f \Lcal^d$ with $f = f(x_3,\ldots,x_d)$ linear, since the space of measures of this form is finite-dimensional and hence weakly* closed. Thus, we see that there exists a Radon measure $\mu \in \Mbf_\loc(\R^2)$ and a linear map $f \colon \R^{d-2} \to \R$ such that
\[
  Eu =  P_0 \Bigl[ \mu \otimes \Lcal^{d-2} + f(x_3,\ldots,x_d) \Lcal^d_x \Bigr].
\]
This proves the claim.
\end{proof}

We can now finish the proof of case (ii) of our theorem:

\begin{proof}[Proof of Theorem~\ref{thm:good_blowups}~(ii)]
Like in the proof of part~(i) of the theorem, take a singular tangent Young measure $\nu \in \BDY_\loc(\R^d)$ at a point $x_0 \in \Omega$ as in Proposition~\ref{prop:localize_sing} and let $v \in \BD_\loc(\R^d)$ with $Ev = [\sigma]$. As before, it holds from the properties of tangent Young measures that
\[
  Ev = P_0 \abs{Ev}.
\]

\proofstep{Step~1.}
We first show the result in the case $a = \ee_1$, $b = \ee_2$, i.e.\ $P_0 = \sqrt{2} (\ee_1 \odot \ee_2)$. Under this asumption we may apply the dimensional reduction result from Lemma~\ref{lem:dim_reduction_e1e2} to get a Radon measure $\mu \in \Mbf_\loc(\R^2)$ and a linear function $f \colon \R^{d-2} \to \R$ for which
\[
  Ev = P_0 \Bigl[ \mu \otimes \Lcal^{d-2} + f(x_3,\ldots,x_d) \Lcal^d_x \Bigr].
\]

If $f$ is non-zero, $[\sigma] = Ev$ cannot be purely singular and so there exists an $\Lcal^d$-negligible set $N \subset \R^d$ such that $[\sigma] \restrict (\R^d \setminus N) = g \Lcal^d$ for some non-zero $g \in \Lrm_\loc^1(\R^d;\R^{d \times d})$. Hence, by virtue of Proposition~\ref{prop:localize_reg} and Preiss' existence result for non-zero tangent measures, there is $y_0 \in \R^d$ and a regular tangent Young measure $\kappa \in \BDY_\loc(\R^d)$ to $\sigma$ at $y_0$ with $[\kappa]$ a non-zero constant multiple of Lebesgue measure, namely $[\kappa] = \alpha P_0 \Lcal^d$ for some $\alpha \neq 0$. Hence, any $\tilde{v} \in \BD_\loc(\R^d)$ with $[\kappa] = E\tilde{v}$ is affine and in particular of the form exhibited in case~(ii) of the theorem (with $h_1, h_2$ linear). As in Step~4 of the proof of part~(i) of the present theorem, we can show that $\kappa$ is a singular tangent measure to $\nu$ at $x_0$ as well (in the sense that it satisfies the conclusion of Proposition~\ref{prop:localize_sing}). Hence, in the case $f$ is not identically zero, we have already shown part~(ii) of the present theorem with $\tilde{v}$ and $\kappa$ in place of $v$ and $\sigma$, respectively.

Next we treat the other case where $f \equiv 0$ and $Ev$ might be purely singular, that is
\begin{equation} \label{eq:Ev_struct}
  Ev = P_0 \, \mu \otimes \Lcal^{d-2}.
\end{equation}
In this situation we have that there exists a function $h \in \BD_\loc(\R^2)$ and $v_0 \in \R^d$ as well as a skew-symmetric matrix $R \in \R_\skw^{d \times d}$ such that
\[
  v(x) = v_0 + \begin{pmatrix} h^1(x_1,x_2) \\ h^2(x_1,x_2) \\ 0 \\ \vdots \\ 0 \end{pmatrix} + R x.
\]
This can roughly be seen as follows: By a mollification argument, we may assume that $v$ is smooth. Then,~\eqref{eq:Ev_struct} means that $\Ecal v(x) = P_0 g(x_1,x_2)$ for some $g \in \Crm^\infty(\R^2)$, $x \in \R^d$. Hence, the function
\[
  h(x_1,x_2) := \begin{pmatrix} v^1(x_1,x_2,0,\ldots,0) \\ v^2(x_1,x_2,0,\ldots,0) \end{pmatrix},
\]
has symmetrized gradient $\Ecal h(x_1,x_2) = \tilde{P}_0 g(x_1,x_2)$, where $\tilde{P}_0$ is the leading principal minor of $P_0$. Considering $h$ to be extended to a function on $\R^d$ (constant in $x_3,\ldots,x_d$) and with $d$ components ($h^3,\ldots,h^d = 0$), we have that $\Ecal h = \Ecal v$ and so, $v$ equals $h$ modulo a rigid deformation.

But for $h$ we can invoke Lemma~\ref{lem:BD_rigidity_2D} to deduce that
\[
  h(x_1,x_2) = h_1(x_1)\ee_2 + h_2(x_2)\ee_1.
\]
where $h_1, h_2 \in \BV_\loc(\R)$. Thus, we arrive at
\[
  v(x) = v_0 + h_1(x_1)\ee_2 + h_2(x_2)\ee_1 + Rx.
\]
This proves the claim for $a = \ee_1$, $b = \ee_2$.

\proofstep{Step~2.}
For general $a,b \in \R^d$ with $a \neq b$ take an invertible matrix $G \in \R^{d \times d}$ with $Ga = \ee_1$, $Gb = \ee_2$. Then $G(a \odot b)G^T = \ee_1 \odot \ee_2$ and hence, replacing $v(x)$ by
\[
  \tilde{v}(x) := Gv(G^T x),
\]
we have $E\tilde{v} = \sqrt{2} (\ee_1 \odot \ee_2) \abs{E\tilde{v}}$. By the previous step, there exist $\tilde{v}_0 \in \R^d$ and a skew-symmetric matrix $\tilde{R} \in \R_\skw^{d \times d}$ such that
\[
  \tilde{v}(x) = \tilde{v}_0 + h_1(x_1)\ee_2 + h_2(x_2)\ee_1 + \tilde{R}x.
\]
We can now transform back to the original $v(x) = G^{-1}\tilde{v}(G^{-T}x)$. In this process, we get
\begin{align*}
  G^{-1}h_1(G^{-T}x \cdot \ee_1)\ee_2 &= h_1(x \cdot a)b, \qquad\text{and}\\
  G^{-1}h_2(G^{-T}x \cdot \ee_2)\ee_1 &= h_2(x \cdot b)a.
\end{align*}
Also setting $v_0 := G^{-1} \tilde{v}_0$ and $R := G^{-1}RG^{-T}$, which is still skew-symmetric, we have proved the claimed splitting in the general situation as well.
\end{proof}

\begin{remark}
As a by-product of the proof, we note the following dichotomy for a measure $\mu \in \Mbf(\R^d;\R^N)$: At $\abs{\mu}$-almost every $x_0 \in \R^d$, either all tangent measures are purely singular (with respect to $\Lcal^d$), or $A_0 \Lcal^d \in \Tan(\mu,x_0)$, where $A_0 = \frac{\di \mu}{\di \abs{\mu}}(x_0)$.
\end{remark}

\subsection{The case $P_0 = a \odot a$}  \label{ssc:a_odot_a}

For this degenerate case we can essentially use the same techniques as in the previous sections, but there are some differences.

\begin{proof}[Proof of Theorem~\ref{thm:good_blowups}~(iii)]
Again we take a singular tangent Young measure $\sigma \in \BDY_\loc(\R^d)$ at a point $x_0 \in \Omega$ from Proposition~\ref{prop:localize_sing} and $v \in \BD_\loc(\R^d)$ with
\[
  Ev = [\sigma] = (a \odot a) \abs{Ev}.
\]

\proofstep{Step~1.} In case that $v$ is smooth and $a = \ee_1$, i.e.\ there exists $g \in \Crm^\infty(\R^d)$ such that
\[
  \Ecal v = (\ee_1 \odot \ee_1) g  \qquad\text{and}\qquad
  E^s v = 0,
\]
we may proceed analogously to Step~1 in the proof of Lemma~\ref{lem:dim_reduction_e1e2}, to get for $i = 2,\ldots,d$,
\[
  \nabla \Wcal u(x)_1^i = - \partial_i Eu(x)^1 = (- \partial_i g(x),0,\ldots,0),
\]
where as before $\Wcal u$ is the skew-symmetric part of $\nabla u$. This gives that $\Wcal u_1^i$ and hence also $\partial_i g$ only depend on the first component $x_1$ of $x$, $\partial_i g(x) = p_i(x_1)$ say. Define
\[
  h(x) := g(x) - p_2(x_1) x_2 - \cdots - p_d(x_1) x_d
\]
and observe that $\partial_i h \equiv 0$ for $i = 2,\ldots,d$. Hence we may write $h(x) = h(x_1)$ and have now decomposed $g$ as
\begin{equation} \label{eq:g_decomp}
  g(x) = h(x_1) + p_2(x_1) x_2 + \cdots + p_d(x_1) x_d.
\end{equation}

\proofstep{Step~2.}
For $v$ only from $\BD_\loc(\R^d)$, but still $a = \ee_1$, we use a smoothing argument very similar to Step~2 in the proof to Lemma~\ref{lem:dim_reduction_e1e2} together with the first step to see that
\begin{equation} \label{eq:Ev_aodota_form}
\begin{aligned}
  Ev = (\ee_1 \odot \ee_1) \Bigl[ \mu \otimes \Lcal^{d-1}
    &+ \gamma_2 \otimes (x_2 \Lcal_{x_2}) \otimes \Lcal^{d-2} \\
  &+ \gamma_3 \otimes \Lcal^d \otimes (x_3 \Lcal_{x_3}) \otimes \Lcal^{d-3} \\
  &+ \cdots \\
  &+ \gamma_d \otimes \Lcal^{d-2} \otimes (x_d \Lcal_{x_d}) \Bigr],
\end{aligned}
\end{equation}
where $\mu, \gamma_2, \ldots, \gamma_d \in \Mbf_\loc(\R)$ are signed measures. In fact, mollify $v$ to get $v_\delta \in \Crm^\infty(\R^d;\R^d)$ and apply Step~1 to the $v_\delta$ to see that $Ev_\delta = (\ee_1 \odot \ee_1) g_\delta \Lcal^d$ with $g_\delta$ of the form exhibited in~\eqref{eq:g_decomp}. Then use test functions of the form
\[
  \phi(x_1) \ONE_{Q^d(R)},  \qquad  \phi(x_1) x_2 \ONE_{Q^d(R)},  \qquad \ldots, \qquad
  \phi(x_1) x_d \ONE_{Q^d(R)}
\]
for $\phi \in \Crm_c((-R,R);[-1,1])$, $R > 0$, in a similar argument as before to see that all parts of the measures $(\ee_1 \odot \ee_1) g_\delta \Lcal^d$ converge separately. Thus, $Ev = \wslim_{\delta \todown 0} Ev_\delta$ has the form~\eqref{eq:Ev_aodota_form}.

Let $y_0 \in \R^d$ be such that there exists another (non-zero) singular tangent Young measure $\kappa \in \BDY_\loc(\R^d)$ to $\sigma$ at $y_0$ (in the sense of Proposition~\ref{prop:localize_sing}). Since then $[\kappa] \in \Tan(Ev,y_0)$ and all parts of $Ev$ are smooth in the variables $x_2,\ldots,x_d$ by~\eqref{eq:Ev_aodota_form}, every tangent measure will be constant in these variables (one can see this for example by testing the blow-up sequence with tensor products of $\Crm_c(\R)$-functions). Hence, $[\kappa]$ can be written in the form
\[
  [\kappa] = \tilde{\mu} \otimes \Lcal^{d-1}
\]
for some $\tilde{\mu} \in \Mbf_\loc(\R)$. As before we have that $\kappa$ is also a singular tangent Young measure to $\nu$ at the point $x_0$.

\proofstep{Step~3.}
We may now argue similarly to Step~2 of the proof of part~(ii) of the theorem in the previous section to get that there exists $h \in \BV_\loc(\R)$ as well as $\tilde{v}_0 \in \R^d$ and a skew-symmetric matrix $\tilde{R} \in \R_\skw^{d \times d}$ with
\[
  \tilde{v}(x) = \tilde{v}_0 + h(x_1)\ee_1 + \tilde{R}x.
\]
This shows the claim of case~(iii) of the theorem for $a = \ee_1$. For general $a$, we use a transformation like in Step~3 of the proof in the previous section.
\end{proof}

\subsection{Rigidity in 2D} \label{ssc:rigidity_2D}

To illustrate the previous rigidity argument in a more concrete situation, this section gives a complete analysis of solutions for the differential inclusion
\begin{equation} \label{eq:incl}
  \Ecal u \in \spn \{ P \}  \qquad\text{pointwise a.e.,}\qquad   u \in \LD_\loc(\R^2),
\end{equation}
for a \emph{fixed} symmetric matrix $P \in \R_\sym^{2 \times 2}$. The results presented here are not needed in the sequel, and for convenience we restrict our analysis to the space $\LD_\loc(\R^2)$ and omit extensions to $\BD_\loc(\R^2)$.

First we notice that we may always reduce the above problem to an equivalent differential inclusion with $P$ diagonal. Indeed, let $Q \in \R^{2 \times 2}$ be an orthogonal matrix such that
\[
  QPQ^T = \begin{pmatrix} \lambda_1 & \\ & \lambda_2 \end{pmatrix} =: \tilde{P},
  \qquad \text{$\lambda_1,\lambda_2 \in \R$.}
\]
Clearly, $u \in \LD_\loc(\R^2)$ solves~\eqref{eq:incl} if and only if $\tilde{u}(x) := Q u(Q^T x)$ solves
\[
  \Ecal \tilde{u} \in \spn \{ \tilde{P} \}  \qquad\text{pointwise a.e.,}
\]
so we may always assume that $P$ in~\eqref{eq:incl} is already diagonal.

According to Lemma~\ref{lem:sym_tensor_prod} we have three non-trivial cases to take care of, corresponding to the signs of the eigenvalues $\lambda_1$, $\lambda_2$; the trivial case $\lambda_1 = \lambda_2 = 0$, i.e.\ $P = 0$, was already settled in Lemma~\ref{lem:E_kernel}.

We will formulate our results on solvability of~\eqref{eq:incl} in terms of conditions on $g \in \Lrm_\loc^1(\R^2)$ in the differential equation
\[
  \Ecal u = P g \quad\text{a.e.,}  \qquad  u \in \LD_\loc(\R^2).
\]
With $g$ as an additional unknown this is clearly equivalent to~\eqref{eq:incl}.

First, consider the situation that $\lambda_1, \lambda_2 \neq 0$ and that these two eigenvalues have opposite signs. Then, from (the proof of) Lemma~\ref{lem:sym_tensor_prod}, we know that $P = a \odot b$ ($a \neq b$) for
\[
  a := \begin{pmatrix} \gamma \\ 1 \end{pmatrix},  \qquad
  b := \begin{pmatrix} \lambda_1 \gamma^{-1} \\ \lambda_2 \end{pmatrix},
  \qquad\text{where}\qquad \gamma := \sqrt{-\frac{\lambda_1}{\lambda_2}}.
\]

The result about solvability of~\eqref{eq:incl} for this choice of $P$ is:

\begin{proposition}[Rigidity for $P = a \odot b$]
Let $P = \Bigl( \begin{smallmatrix} \lambda_1 & \\ & \lambda_2 \end{smallmatrix} \Bigr) = a \odot b$, where $\lambda_1, \lambda_2 \in \R$ have opposite signs. Then, there exists a function $u \in \LD_\loc(\R^2)$ solving the differential equation
\[
  \Ecal u = Pg \quad\text{a.e.}
\]
if and only if $g \in \Lrm_\loc^1(\R^2)$ is of the form
\[
  g(x) = h_1(x \cdot a) + h_2(x \cdot b),  \qquad\text{$x \in \R^2$ a.e.,}
\]
where $h_1, h_2 \in \Lrm_\loc^1(\R)$. In this case, 
\[
  u(x) = u_0 + H_1(x \cdot a)b + H_2(x \cdot b)a + Rx,  \qquad\text{$x \in \R^2$ a.e.,}
\]
with $u_0 \in \R^2$, $R \in \R_\skw^{d \times d}$ and $H_1, H_2 \in \Wrm_\loc^{1,1}(\R)$ satisfying $H_1' = h_1$ and $H_2' = h_2$.
\end{proposition}

\begin{proof}
This follows by virtue of Lemma~\ref{lem:BD_rigidity_2D} together with some elementary computations.
\end{proof}

In the case $\lambda_1 \neq 0$, $\lambda_2 = 0$, i.e.\ $P = \lambda_1 (\ee_1 \odot \ee_1)$, one could guess by analogy to the previous case that if $u \in \LD_\loc(\R^2)$ satisfies $\Ecal u = Pg$ for some $g \in \Lrm_\loc^1(\R)$, then $u$ and $g$ should only depend on $x_1$ up to a rigid deformation. This, however, is \emph{false}, as can be seen from the following example.

\begin{example}
Consider
\[
  P := \begin{pmatrix} 1 & \\ & 0 \end{pmatrix},  \qquad
  u(x) := \begin{pmatrix} 4 x_1^3 x_2 \\ -x_1^4 \end{pmatrix},  \qquad
  g(x) := 12x_1^2 x_2.
\]
Then, $u$ satisfies $\Ecal u = Pg$, but neither $u$ nor $g$ only depend on $x_1$.
\end{example}

The general statement reads as follows.

\begin{proposition}[Rigidity for $P = a \odot a$]
Let $P = \Bigl( \begin{smallmatrix} \lambda_1 & \\ & 0 \end{smallmatrix} \Bigr) = \lambda_1 (\ee_1 \odot \ee_1)$. Then, there exists a function $u \in \LD_\loc(\R^2)$ solving the differential equation
\[
  \Ecal u = Pg \quad\text{a.e.}
\]
if and only if $g \in \Lrm_\loc^1(\R^2)$ is of the form
\[
  g(x) = h(x_1) + p(x_1)x_2,  \qquad\text{$x \in \R^2$ a.e.,}
\]
where $h,p \in \Lrm_\loc^1(\R)$. In this case,
\[
  u(x) = u_0 + \lambda_1 \begin{pmatrix} H(x_1) + \Pcal'(x_1)x_2 \\ -\Pcal(x_1) \end{pmatrix} + Rx,
  \qquad\text{$x \in \R^2$ a.e.,}
\]
with $u_0 \in \R^2$, $R \in \R_\skw^{d \times d}$ and $H \in \Wrm_\loc^{1,1}(\R)$, $\Pcal \in \Wrm_\loc^{2,1}(\R)$ satisfying $H_1' = h_1$ and $\Pcal'' = p$.
\end{proposition}

\begin{proof}
From the arguments in Section~\ref{ssc:a_odot_a} we know that whenever $u \in \LD_\loc(\R^2)$ solves the differential equation $\Ecal u = Pg$, then $g$ (and hence also $u$) must have the form exhibited in the statement of the proposition. Conversely, it is elementary to check that $u$ as defined above satisfies $\Ecal u = P g$.
\end{proof}

Finally, we consider the case where the eigenvalues $\lambda_1$ and $\lambda_2$ are non-zero and have the same sign. Then, $P \neq a \odot b$ for any $a,b \in \R^2$ by Lemma~\ref{lem:sym_tensor_prod}. Define the differential operator
\[
  \Acal_P := \lambda_2 \partial_{11} + \lambda_1 \partial_{22}
\]
and notice that whenever a function $g \colon \R^2 \to \R$ satisfies $\Acal_P g \equiv 0$ distributionally, then by elliptic regularity (generalized Weyl's Lemma), we have that in fact $g \in \Crm^\infty(\R^2)$.

\begin{proposition}[Rigidity for $P \neq a \odot b$] \label{prop:P_neq_a_odot_b_2D}
Let $P = \Bigl( \begin{smallmatrix} \lambda_1 & \\ & \lambda_2 \end{smallmatrix} \Bigr)$, where $\lambda_1, \lambda_2 \in \R$ have the same sign. Then, there exists a function $u \in \LD_\loc(\R^2)$ solving the differential equation
\[
  \Ecal u = Pg \quad\text{a.e.}
\]
if and only if $g \in \Lrm_\loc^1(\R^2)$ satisfies
\[
  \Acal_P g \equiv 0.
\]
Moreover, in this case both $g$ and $u$ are smooth.
\end{proposition}

\begin{proof}
First assume that $g \in \Crm^\infty(\R^2)$ satisfies $\Acal_P g \equiv 0$. Define
\[
  F := \nabla g \begin{pmatrix} 0 & \lambda_2 \\ -\lambda_1 & 0 \end{pmatrix}
  = (- \lambda_1 \partial_2 g, \lambda_2 \partial_1 g)
\]
and observe (we use $\curl \, (h_1,h_2) = \partial_2 h_1 - \partial_1 h_2$)
\[
  \curl F = - \lambda_1 \partial_{22} g - \lambda_2 \partial_{11} g = - \Acal_P g \equiv 0.
\]
Hence, there exists $f \in \Crm^\infty(\R^2)$ with $\nabla f = F$, in particular
\begin{equation} \label{eq:solv_cond}
  \partial_1 f = - \lambda_1 \partial_2 g,  \qquad  \partial_2 f = \lambda_2 \partial_1 g.
\end{equation}

Put
\[
  \Ucal := \begin{pmatrix} \lambda_1 & 0 \\ 0 & \lambda_2 \end{pmatrix} g
  + \begin{pmatrix} 0 & -1 \\ 1 & 0 \end{pmatrix} f.
\]
We calculate (this time we apply the curl row-wise), using~\eqref{eq:solv_cond},
\begin{equation} \label{eq:curl_U}
  \curl \, \Ucal = \begin{pmatrix} \curl \, \bigl( \lambda_1 g, -f \bigr) \\
    \curl \, \bigl( f, \lambda_2 g \bigr) \end{pmatrix}
  = \begin{pmatrix} \lambda_1 \partial_2 g + \partial_1 f \\ 
    \partial_2 f - \lambda_2 \partial_1 g \end{pmatrix} \equiv 0.
\end{equation}
Let $u \in \Crm^\infty(\R^2;\R^2)$ be such that $\nabla u = \Ucal$. Then, $\Ecal u = P g$.

For the other direction, it suffices to show that $\Ecal u = Pg$ implies $\Acal_P g \equiv 0$, the smoothness of $u,g$ follows from the first step. Notice further that by a mollification argument we may in fact assume that $u \in \Crm^\infty(\R^2;\R^2)$, $g \in \Crm^\infty(\R^2)$, since the conditions $\Ecal u = Pg$ and $\Acal_P g \equiv 0$ are preserved under smoothing. So, splitting the gradient into its symmetric and skew-symmetric parts,
\[
  \nabla u = \begin{pmatrix} \lambda_1 & 0 \\ 0 & \lambda_2 \end{pmatrix} g
  + \begin{pmatrix} 0 & -1 \\ 1 & 0 \end{pmatrix} f
\]
for some function $f \in \Crm^\infty(\R^2)$. As in~\eqref{eq:curl_U}, this implies the conditions~\eqref{eq:solv_cond} for $\nabla g, \nabla f$. Hence,
\[
  \nabla f = \nabla g \begin{pmatrix} 0 & \lambda_2 \\ -\lambda_1 & 0 \end{pmatrix}
  = ( -\lambda_1 \partial_2 g, \lambda_2 \partial_1 g ).
\]
Since the curl of $\nabla f$ vanishes, we get
\[
  0 \equiv \curl \nabla f = - \lambda_1 \partial_{22} g - \lambda_2 \partial_{11} g = -\Acal_P g,
\]
so $g$ satisfies $\Acal_P g \equiv 0$.
\end{proof}

\begin{remark}[Harmonic functions]
By Lemma~\ref{lem:sym_tensor_prod}, the simplest matrix that cannot be written as a symmetric tensor product is the identity matrix $P = I_2 = \Bigl( \begin{smallmatrix} 1 &  \\  & 1 \end{smallmatrix} \Bigr)$. In this case $\Acal_P$ is the Laplacian and the differential equation $\Ecal u = I_2 g$ is solvable in $\LD_\loc(\R^2)$ if and only if $g$ is harmonic.
\end{remark}

\begin{remark}[Comparison to gradients]
Proposition~\ref{prop:P_neq_a_odot_b_2D} should be contrasted with the corresponding situation for gradients. If $u \in \Wrm_\loc^{1,1}(\R^2;\R^2)$ satisfies
\[
  \nabla u \in \spn \{ P \}  \qquad\text{pointwise a.e.,}
\]
and $\rank P = 2$, then necessarily $u$ is affine, a proof of which can be found in Lemma~3.2 of~\cite{Rind10?LSYM} (this rigidity result is closely related to Hadamard's jump condition, also see~\cite[Proposition~2]{BalJam87FPMM},~\cite[Lemma~1.4]{DeLe09NARO},~\cite[Lemma~2.7]{Mull99VMMP} for related results). Notice that this behavior for the gradient is in sharp contrast to the behavior for the symmetrized gradient, as can be seen from the following example.
\end{remark}

\begin{example} \label{ex:weak_rigidity}
Let
\[
  P := \begin{pmatrix} 1 & \\ & 1 \end{pmatrix},  \qquad
  u(x) := \begin{pmatrix} \ee^{x_1} \sin(x_2) \\ -\ee^{x_1} \cos(x_2) \end{pmatrix},  \qquad
  g(x) := \ee^{x_1} \sin(x_2).
\]
Then, one can check that $g$ is harmonic and $u$ satisfies $\Ecal u = Pg$. So, the fact that $P$ cannot be written as a symmetric tensor product does not imply that that any solution to the differential inclusion $\Ecal u \in \spn \{P\}$ must be affine.
\end{example}

\section{Jensen-type inequalities} \label{sc:Jensen}

In this section we establish the following necessary conditions for BD-Young measures, which will later yield general lower semicontinuity and relaxation results as corollaries.

\begin{theorem}[Jensen-type inequalities] \label{thm:BDY_Jensen}
Let $\nu \in \BDY(\Omega)$ be a BD-Young measure. Then, for all symmetric-quasiconvex $h \in \Crm(\R_\sym^{d \times d})$ with linear growth at infinity it holds that
\begin{align*}
  h \biggl( \dprb{\id,\nu_{x_0}} + \dprb{\id,\nu_{x_0}^\infty} \frac{\di \lambda_\nu}{\di \Lcal^d}(x_0) \biggr)
    &\leq \dprb{h,\nu_{x_0}} + \dprb{h^\#,\nu_{x_0}^\infty} \frac{\di \lambda_\nu}{\di \Lcal^d}(x_0)
\intertext{for $\Lcal^d$-a.e.\ $x_0 \in \Omega$, and}
  h^\# \bigl( \dprb{\id,\nu_{x_0}^\infty} \bigr) &\leq \dprb{h^\#,\nu_{x_0}^\infty}
\end{align*}
for $\lambda_\nu^s$-a.e.\ $x_0 \in \Omega$.
\end{theorem}

The proof is contained in Lemmas~\ref{lem:Jensen_regular} and~\ref{lem:Jensen_singular} below (notice that if $h$ is symmetric-quasiconvex, then so is its generalized recession function $h^\#$).

\subsection{Jensen-type inequality at regular points}

The proof at regular points is straightforward.

\begin{lemma} \label{lem:Jensen_regular}
Let $\nu \in \BDY(\Omega)$ be a BD-Young measure. Then, for $\Lcal^d$-a.e.\ $x_0 \in \Omega$ it holds that
\[
  h \biggl( \dprb{\id,\nu_{x_0}} + \dprb{\id,\nu_{x_0}^\infty} \frac{\di \lambda_\nu}{\di \Lcal^d}(x_0) \biggr)
  \leq \dprb{h,\nu_{x_0}} + \dprb{h^\#,\nu_{x_0}^\infty} \frac{\di \lambda_\nu}{\di \Lcal^d}(x_0)
\]
for all symmetric-quasiconvex $h \in \Crm(\R_\sym^{d \times d})$ with linear growth at infinity.
\end{lemma}

\begin{proof}
Use Proposition~\ref{prop:localize_reg} to get a regular tangent Young measure $\sigma \in \BDY(\Bbb^d)$ to $\nu$ at a suitable $x_0 \in \Omega$ (this is possible for $\Lcal^d$-almost every $x_0 \in \Omega$). With
\[
  A_0 := \dprb{\id,\nu_{x_0}} + \dprb{\id,\nu_{x_0}^\infty} \frac{\di \lambda_\nu}{\di \Lcal^d}(x_0),
\]
it holds that $[\sigma] = A_0 \Lcal^d$. From Lemma~\ref{lem:boundary_adjust} take a sequence $(v_n) \subset (\Wrm^{1,1} \cap \Crm^{\infty})(\Bbb^d;\R^d)$ with $Ev_n \toY \sigma$ in $\Ybf(\Bbb^d;\R_\sym^{d \times d})$ and $v_n|_{\partial \Bbb^d}(x) = A_0 x$ on $\partial \Omega$. Since the function $h$ is quasiconvex,
\[
  h(A_0) \leq \dashint_{\Bbb^d} h(\Ecal v_n) \dd z.
\]
By virtue of the approximation result cited in Section~\ref{ssc:integrands} we get a sequence $(\ONE_{\Bbb^d} \otimes h_k) \subset \Ebf(\Bbb^d;\R_\sym^{d \times d})$ with $h_k \todown h$, $h_k^\infty \todown h^\#$ pointwise and $\sup_k \norm{\ONE_{\Bbb^d} \otimes h_k}_{\Ebf} < \infty$. Thus, for all $k \in \N$,
\begin{align*}
  h(A_0) &\leq \limsup_{n \to \infty} \dashint_{\Bbb^d} h(\Ecal v_n) \dd z
    \leq \lim_{n \to \infty} \dashint_{\Bbb^d} h_k(\Ecal v_n) \dd z \\
  &= \frac{1}{\omega_d} \ddprb{\ONE_{\Bbb^d} \otimes h_k,\sigma}
    = \dprb{h_k,\nu_{x_0}} + \dprb{h_k^\infty,\nu_{x_0}^\infty} \frac{\di \lambda_\nu}{\di \Lcal^d}(x_0),
\end{align*}
where the last equality follows from~\eqref{eq:loc_reg_ddpr}. Now let $k \to \infty$ and invoke the monotone convergence theorem to conclude.
\end{proof}

\subsection{Jensen-type inequality at singular points}

We now prove a Jensen-type inequality at singular points, utilizing the good blow-ups from Theorem~\ref{thm:good_blowups}. At points where the (good) blow-up is affine, this is a straightforward application of the quasiconvexity. At (almost all) other points, we can decompose the blow-up into one or two one-directional functions and an affine part (cf.\ Figure~\ref{fig:singular_blowups}). This special structure allows us to average the functions into an affine function, which then allows the application of quasiconvexity, see Figure~\ref{fig:averaging} for an illustration of this averaging procedure.

\begin{figure}[t]
\centering
\begingroup
  \makeatletter
  \providecommand\color[2][]{%
    \errmessage{(Inkscape) Color is used for the text in Inkscape, but the package 'color.sty' is not loaded}
    \renewcommand\color[2][]{}%
  }
  \providecommand\transparent[1]{%
    \errmessage{(Inkscape) Transparency is used (non-zero) for the text in Inkscape, but the package 'transparent.sty' is not loaded}
    \renewcommand\transparent[1]{}%
  }
  \providecommand\rotatebox[2]{#2}
  \ifx\svgwidth\undefined
    \setlength{\unitlength}{298pt}
  \else
    \setlength{\unitlength}{\svgwidth}
  \fi
  \global\let\svgwidth\undefined
  \makeatother
  \begin{picture}(1,0.53062195)%
    \put(0,0){\includegraphics[width=\unitlength]{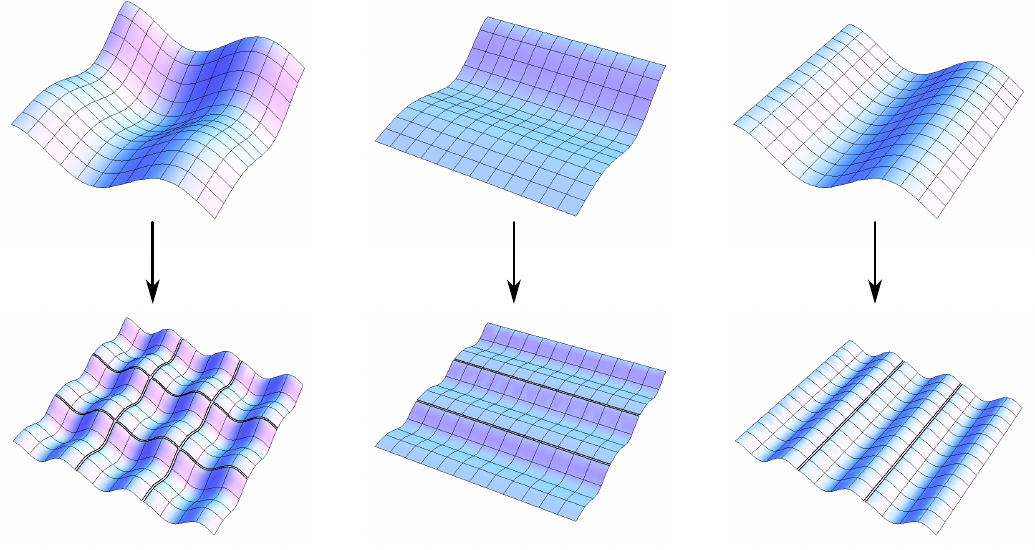}}%
    \put(0.65771808,0.39639377){\color[rgb]{0,0,0}\makebox(0,0)[lb]{\smash{$+$}}}%
    \put(0.30872483,0.39639377){\color[rgb]{0,0,0}\makebox(0,0)[lb]{\smash{$=$}}}%
    \put(0.65771808,0.10109178){\color[rgb]{0,0,0}\makebox(0,0)[lb]{\smash{$+$}}}%
    \put(0.30872483,0.10109178){\color[rgb]{0,0,0}\makebox(0,0)[lb]{\smash{$=$}}}%
    \put(0.60671137,0.26753473){\color[rgb]{0,0,0}\makebox(0,0)[lb]{\smash{averaging}}}%
  \end{picture}%
\endgroup

\caption{Staircase construction for the singular Jensen-type inequality.}
\label{fig:averaging}
\end{figure}
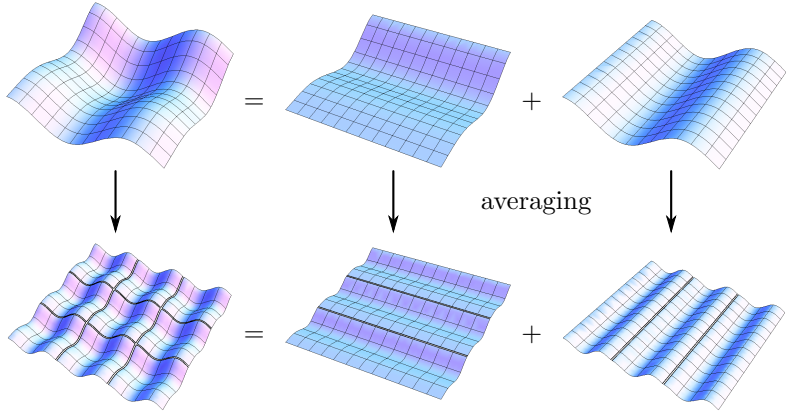

\begin{lemma} \label{lem:Jensen_singular}
Let $\nu \in \BDY(\Omega)$ be a BD-Young measure. Then, for $\lambda_\nu^s$-almost every $x_0 \in \Omega$ it holds that
\[
  g \bigl( \dprb{\id,\nu_{x_0}^\infty} \bigr) \leq \dprb{g,\nu_{x_0}^\infty}
\]
for all symmetric-quasiconvex and positively $1$-homogeneous $g \in \Crm(\R_\sym^{d \times d})$.
\end{lemma}

\begin{proof}
Theorem~\ref{thm:good_blowups} (which uses the singular localization principle, Proposition~\ref{prop:localize_sing}) on the existence of good blow-ups yields the existence of a singular tangent Young measure $\sigma \in \BDY_\loc(\R^d)$ to $\nu$ at $\lambda_\nu^s$-almost every $x_0 \in \Omega$. Let $[\sigma] = Ev$ for some $v \in \BD_\loc(\R^d)$ and define
\[
  A_0 := \dprb{\id,\nu_{x_0}^\infty}.
\]
Observe that by~\eqref{eq:loc_sing_ddpr}, $Ev = [\sigma] = A_0 \lambda_\sigma$. Moreover, depending on the value of $A_0$, one of the cases (i),~(ii),~(iii) in Theorem~\ref{thm:good_blowups} holds.

\proofstep{Case 1: $A_0 \notin \setn{a \odot b}{a,b \in \R^d \setminus \{0\}}$ (possibly $A_0 = 0$).}\\
By Theorem~\ref{thm:good_blowups}~(i), $v$ is affine, and multiplying $v$ by a constant, we may assume without loss of generality that $Ev = A_0 \Lcal^d$. Adding a rigid deformation if necessary, we may in fact assume $v(x) = A_0x$. Now restrict $\sigma$ to the unit ball $\Bbb^d$ and by virtue of Lemma~\ref{lem:boundary_adjust} take a sequence $(v_n) \subset (\Wrm^{1,1} \cap \Crm^{\infty})(\Bbb^d;\R^d)$ with $Ev_n \toY \sigma$ in $\Ybf(\Bbb^d;\R_\sym^{d \times d})$ and $v_n|_{\partial \Bbb^d}(x) = A_0 x$ on $\partial \Bbb^d$. Since $g$ is quasiconvex,
\[
  g(A_0) \leq \dashint_{\Bbb^d} g(\Ecal v_n) \dd z.
\]
Finally, we may use~\eqref{eq:loc_sing_ddpr} to get
\[
  g(A_0) \leq \limsup_{n \to \infty} \dashint_{\Bbb^d} g(\Ecal v_n) \dd x
  = \frac{1}{\omega_d} \ddprb{\ONE_{\Bbb^d} \otimes g,\sigma} = \dprb{g,\nu_{x_0}^\infty}.
\]
This proves the claim in this case.

\proofstep{Case 2: $A_0 = q(a \odot b)$, where $a,b \in \Sbb^{d-1}, q \in \R \setminus \{0\}$ and $a \neq b$.}\\
Let $P$ be an open unit parallelotope with its mid-point at the origin and with two face normals $a,b$. The other face normals are orthogonal to $a$ and $b$, yet otherwise arbitrary, i.e.\ if $\xi_3,\ldots,\xi_d \in \Sbb^{d-1}$ extend $a,b$ to a basis of $\R^d$ and satisfy $\xi_3,\ldots,\xi_d \perp \spn\{a,b\}$, then
\[
  P = \setb{ x \in \R^d }{ \abs{x \cdot a}, \abs{x \cdot b}, \abs{x \cdot \xi_3}, \ldots, \abs{x \cdot \xi_d} \leq \textstyle\frac{1}{2} }.
\]
We also set $P(x_0,r) := x_0 + rP$, where $x_0 \in \R^d$, $r > 0$. Put all the principal vectors of $P$ (i.e.\ the vectors lying in the edges) as columns into the matrix $X \in \R^{d \times d}$. See Figure~\ref{fig:parallelotope} for notation.

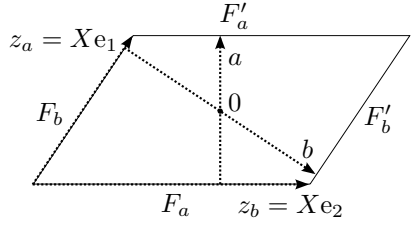
\begin{figure}[t]
\centering
\begingroup
  \makeatletter
  \providecommand\color[2][]{%
    \errmessage{(Inkscape) Color is used for the text in Inkscape, but the package 'color.sty' is not loaded}
    \renewcommand\color[2][]{}%
  }
  \providecommand\transparent[1]{%
    \errmessage{(Inkscape) Transparency is used (non-zero) for the text in Inkscape, but the package 'transparent.sty' is not loaded}
    \renewcommand\transparent[1]{}%
  }
  \providecommand\rotatebox[2]{#2}
  \ifx\svgwidth\undefined
    \setlength{\unitlength}{141.7984375pt}
  \else
    \setlength{\unitlength}{\svgwidth}
  \fi
  \global\let\svgwidth\undefined
  \makeatother
  \begin{picture}(1,0.54146511)%
    \put(0,0){\includegraphics[width=\unitlength]{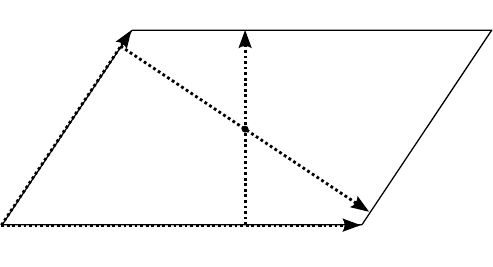}}%
    \put(0.51904657,0.40291018){\color[rgb]{0,0,0}\makebox(0,0)[lb]{\smash{$a$}}}%
    \put(0.71368906,0.1546705){\color[rgb]{0,0,0}\makebox(0,0)[lb]{\smash{$b$}}}%
    \put(0.49647933,0.51574641){\color[rgb]{0,0,0}\makebox(0,0)[lb]{\smash{$F_a'$}}}%
    \put(0.34132969,0.01080431){\color[rgb]{0,0,0}\makebox(0,0)[lb]{\smash{$F_a$}}}%
    \put(0.87730158,0.24211857){\color[rgb]{0,0,0}\makebox(0,0)[lb]{\smash{$F_b'$}}}%
    \put(0.01128366,0.25340219){\color[rgb]{0,0,0}\makebox(0,0)[lb]{\smash{$F_b$}}}%
    \put(0.51903555,0.27912091){\color[rgb]{0,0,0}\makebox(0,0)[lb]{\smash{$0$}}}%
    \put(-0.05974586,0.45086558){\color[rgb]{0,0,0}\makebox(0,0)[lb]{\smash{$z_a = X \ee_1$}}}%
    \put(0.54326669,0.00798341){\color[rgb]{0,0,0}\makebox(0,0)[lb]{\smash{$z_b = X \ee_2$}}}%
  \end{picture}%
\endgroup

\caption{Parallelotope notation.}
\label{fig:parallelotope}
\end{figure}

By Theorem~\ref{thm:good_blowups}~(ii), there exist functions $h_1,h_2 \in \BV_\loc(\R)$, a vector $v_0 \in \R^d$, and a skew-symmetric matrix $R \in \R_\skw^{d \times d}$ such that
\begin{equation} \label{eq:blowup_split_form}
  v(x) = v_0 + h_1(x \cdot a)b + h_2(x \cdot b)a + Rx.
\end{equation}
Without loss of generality we may assume that $v_0 = 0$ and $R = 0$. Moreover, we may additionally suppose that
\begin{equation} \label{eq:lambda_sigma_prop}
  \lambda_\sigma(P) > 0  \qquad\text{and}\qquad  \lambda_\sigma(\partial P) = 0.
\end{equation}
This can be achieved by taking a larger parallelotope $P' = tP \supset P$ ($t > 1$) with $\lambda_\sigma(P') > 0$, $\lambda_\sigma(\partial P') = 0$ if necessary, and then modifying the blow-up radii $r_n \todown 0$ to $r_n' := t r_n$.

Let $F_a, F_a' \subset \partial P$ be the two faces of $P$ with normal $a$ and such that $F_a$ lies in the affine hyperplane $H_a - a/2$, where $H_a := \set{ x \in \R^d }{ x \cdot a = 0 }$. Likewise define $F_b, F_b'$ and also $F_3, F_3', \ldots, F_d, F_d'$ for the remaining parallel face pairs. Then, the special form~\eqref{eq:blowup_split_form} of $v$ and the observation that the vectors $z_a = X \ee_1$, $z_b = X \ee_2$ (say) with $F_a' = F_a + z_a$, $F_b' = F_b + z_b$ satisfy
\[
  z_a \perp b  \qquad\text{and}\qquad  z_b \perp a,
\]
together yield
\[
  v|_{F_a'} - v|_{F_a}(\frarg - z_a) \equiv q_1 b,   \qquad v|_{F_b'} - v|_{F_b}(\frarg - z_b) \equiv q_2 a
\]
where $q_1 = h(1/2) - h(-1/2) = Dh_1((-1/2,1/2))$ and $q_2 = Dh_2((-1/2,1/2))$, as well as
\[
  v|_{F_k'} - v|_{F_k}(\frarg - X\ee_k) \equiv 0  \qquad\text{for $k = 3,\ldots,d$.}
\]
By the chain rule in $\BV$,
\begin{align*}
  Ev(P) = (q_1 + q_2) a \odot b,
\end{align*}
but on the other hand from the properties of $\sigma$, see~\eqref{eq:loc_sing_1}, we have
\[
  Ev(P) = [\sigma](P) = \dprb{\id,\nu_{x_0}^\infty} \lambda_\sigma(P) = A_0 \lambda_\sigma(P)
  = q (a \odot b) \lambda_\sigma(P),
\]
and so in particular
\[
  q \cdot \lambda_\sigma(P) = q_1 + q_2.
\]

By virtue of the Boundary Adjustment Lemma~\ref{lem:boundary_adjust}, we take a BD-norm bounded sequence $(v_n) \subset (\Wrm^{1,1} \cap \Crm^{\infty})(P;\R^d)$ with $v_n|_{\partial P} = v|_{\partial P}$ such that $Ev_n \toY \sigma$ in $\Ybf(P;\R_\sym^{d \times d})$. Extend $v_n$ to all of $\R^d$ by periodicity (with respect to the periodicity cell $P$) and define
\[
  w_n := v_n(x) + q_1 \floorBB{x \cdot a + \frac{1}{2}} b + q_2 \floorBB{x \cdot b + \frac{1}{2}} a,  \qquad x \in P.
\]
Clearly, $(w_n) \subset \BD(P)$ and one checks that the $Ew_n$ in fact do not charge the gluing surfaces. Indeed, the size of the jump incurred over the boundary of each copy of $P$ from the gluing of the $v_n$ is exactly compensated for by the staircase function. For example, over each $(F_a, F_a')$-interface, the first term in the definition of $w_n$ incurs a jump of magnitude $-q_1 b$, but at the same time the staircase term gives a jump of size $q_1 b$ over the same gluing interface, whence in $w_n$ no jump remains. Thus, $(w_n) \subset \LD_\loc(\R^d)$.

Now set
\[
  u_n(x) := \frac{w_n(nx)}{n}  \qquad x \in P,
\]
which lies in $\LD(P)$ and satisfies
\[
  \Ecal u_n(x) = \sum_{z \in \{0,\ldots,n-1\}^d} \Ecal v_n(nx - Xz) \ONE_{P(Xz/n,1/n)}(x).
\]

Next, we show that for some skew-symmetric matrix $R_0 \in \R_\skw^{d \times d}$,
\[
  u_n  \quad\to\quad  (\lambda_\sigma(P) A_0 + R_0)x  \qquad\text{in $\Lrm^1(P;\R^d)$.}
\]
To see this, first observe
\[
  \normBB{\frac{v_n(nx)}{n}}_{\Lrm^1(P;\R^d)} = \frac{1}{n}\norm{v_n}_{\Lrm^1(P;\R^d)}
  \quad\to\quad  0  \qquad\text{as $n \to \infty$}
\]
by a change of variables. On the other hand,
\[
  \frac{1}{n} \biggl( q_1 \floorBB{nx \cdot a + \frac{1}{2}} b + q_2 \floorBB{nx \cdot b + \frac{1}{2}} a \biggl) \quad\to\quad
  \bigl[ q_1 (b \otimes a) + q_2 (a \otimes b) \bigr] x
\]
uniformly. The symmetric part of the matrix on the right hand side is $(q_1 + q_2)a \odot b = \lambda_\sigma(P) A_0$ and so the claim follows. Subtracting $R_0 x$ from $v_n, v$, we may even assume that $R_0 = 0$.

We can now use Lemma~\ref{lem:boundary_adjust} again to get a sequence $(\tilde{u}_n) \subset \LD(P;\R^d)$ satisfying $\tilde{u}_n|_{\partial P}(x) = \lambda_\sigma(P) A_0 x$ on $\partial P$ such that for all $g$ as in the statement of the lemma,
\[
  \lim_{n \to \infty} \int_P g(\Ecal \tilde{u}_n) \dd x
  = \lim_{n \to \infty} \int_P g(\Ecal u_n) \dd x,
\]
by using the fact that $(E\tilde{u}_n)$ and $(Eu_n)$ generate the same (unnamed) Young measure.

The boundary conditions of $\tilde{u}_n$ together with the quasiconvexity of $g$ imply (notice $\abs{P} = 1$)
\[
  g \bigl( \lambda_\sigma(P) A_0 \bigr) \leq \int_P g(\Ecal \tilde{u}_n) \dd z.
\]
This allows us to calculate
\begin{align*}
  \lambda_\sigma(P) g(A_0) &\leq \lim_{n \to \infty} \int_P g(\Ecal \tilde{u}_n) \dd x
    = \lim_{n \to \infty} \int_P g(\Ecal u_n) \dd x \\
  &= \lim_{n \to \infty} \sum_{z \in \{0,\ldots,n-1\}^d} \int_{P(Xz/n,1/n)}
    g \bigl( \Ecal v_n(nx - Xz) \bigr) \dd x \\
  &= \lim_{n \to \infty} \sum_{z \in \{0,\ldots,n-1\}^d} \frac{1}{n^d} \int_P g(\Ecal v_n) \dd y \\
  &= \lim_{n \to \infty} \int_P g(\Ecal v_n) \dd y = \ddprb{\ONE_P \otimes g, \sigma}
  = \dprb{g,\nu_{x_0}^\infty} \lambda_\sigma(P),
\end{align*}
where the two last equalities follow from~\eqref{eq:loc_sing_ddpr} in conjunction with~\eqref{eq:lambda_sigma_prop}. Hence we have also shown the claim in this case.

\proofstep{Case 3: $A_0 = q(a \odot a)$, where $a \in \Sbb^{d-1}, q \in \R$.}\\
This case follows exactly like before, but using a parallelotope of which we only prescribe one face normal $a$ instead of $a,b$, and with
\[
  v(x) = v_0 + h(x \cdot a)a + Rx.
\]
in place of~\eqref{eq:blowup_split_form} by Theorem~\ref{thm:good_blowups}~(iii).
\end{proof}

\section{Lower semicontinuity and relaxation} \label{sc:lsc}

The Jensen-type inequalities from the previous Theorem~\ref{thm:BDY_Jensen} can be employed to easily yield lower semicontinuity and relaxation results in the space $\BD(\Omega)$, where here and in all of the following $\Omega \subset \R^d$ is a bounded Lipschitz domain with boundary unit inner normal $n_\Omega \colon \partial \Omega \to \Sbb^{d-1}$.

The main lower semicontinuity theorem of this work was already announced as Theorem~\ref{thm:BD_lsc_teaser} in the introduction:

\begin{theorem}[Lower semicontinuity in $\BD$] \label{thm:BD_lsc}
Let $f \colon \cl{\Omega} \times \R_\sym^{d \times d} \to \R$ satisfy the following assumptions:
\begin{itemize}
  \item[(i)] $f$ is a Carath\'{e}odory function,
  \item[(ii)] $\abs{f(x,A)} \leq M(1+\abs{A})$ for some $M > 0$ and all $x \in \cl{\Omega}$, $A \in \R_\sym^{d \times d}$,
  \item[(iii)] $f(x,\frarg)$ is symmetric-quasiconvex for all $x \in \cl{\Omega}$,
  \item[(iv)] the (strong) recession function $f^\infty(x,A)$ exists for all $x \in \cl{\Omega}$, $A \in \R_\sym^{d \times d}$ in the sense of~\eqref{eq:f_infty_def_Lip} and is (jointly) continuous on $\cl{\Omega} \times \R_\sym^{d \times d}$.
\end{itemize}
Then, the functional
\begin{equation} \label{eq:F_def}
\begin{aligned}
  \Fcal(u) &:= \int_\Omega f \bigl( x, \Ecal u(x) \bigr) \dd x + \int_\Omega f^\infty \Bigl( x,
    \frac{\di E^s u}{\di \abs{E^s u}}(x) \Bigr) \dd \abs{E^s u}(x) \\
  &\qquad + \int_{\partial \Omega} f^\infty \bigl( x, u|_{\partial \Omega}(x) \odot n_\Omega(x) \bigr)
    \dd \Hcal^{d-1}(x),  \qquad u \in \BD(\Omega),
\end{aligned}
\end{equation}
is sequentially lower semicontinuous with respect to weak*-convergence in the space $\BD(\Omega)$
\end{theorem}

\begin{remark}
Of course, in the above theorem the boundary term can be omitted if the boundary values of all $u_j$ are the same as the boundary value of the limit $u$, or if $f \geq 0$, see Remark~2 in~\cite{KriRin10RSIF} for more explanation.
\end{remark}

\begin{proof}
Let $u_j \toweakstar u$ in $\BD(\Omega)$ and consider $u_j,u$ to be extended by zero to $\R^d$. Assume also, taking a subsequence if necessary, that $Eu_j \toY \nu$ in $\BDY(\R^d)$. The operation of taking subsequences does not preclude our aim to prove lower semicontinuity since we will show an inequality for all such subsequences, which then clearly also holds for the original sequence.

For the barycenter of $\nu$ we have
\[
  [\nu] = Eu \restrict \Omega + (u|_{\partial \Omega} \odot n_{\Omega}) \, \Hcal^{d-1} \restrict \partial \Omega.
\]
Denote by $\lambda_\nu^*$ the singular part of $\lambda_\nu$ with respect to $\abs{E^s u} + \Hcal^{d-1} \restrict \partial\Omega$, i.e.\ $\lambda_\nu^*$ is concentrated in an $(\abs{E^s u} + \Hcal^{d-1} \restrict \partial\Omega)$-negligible set. We compute
\begin{align*}
  &\dprb{\id,\nu_x} + \dprb{\id,\nu_x^\infty} \frac{\di \lambda_\nu}{\di \Lcal^d}(x)
    = \frac{\di [\nu]}{\di \Lcal^d}(x) = \begin{cases}
      \Ecal u(x)  &\text{for $\Lcal^d$-a.e.\ $x \in \Omega$,} \\
      0           &\text{for $\Lcal^d$-a.e.\ $x \in \R^d \setminus \Omega$,}
    \end{cases} \\
  &\frac{\dpr{\id,\nu_x^\infty}}{\abs{\dpr{\id,\nu_x^\infty}}}
    = \frac{\di [\nu]^s}{\di \abs{[\nu]^s}}(x)
    = \begin{cases}
      \displaystyle\frac{\di E^s u}{\di \abs{E^s u}}(x)
        &\text{for $\abs{E^su}$-a.e.\ $x \in \Omega$,} \\
      \displaystyle\frac{u|_{\partial \Omega}(x) \odot n_{\Omega}(x)}
        {\abs{u|_{\partial \Omega}(x) \odot n_{\Omega}(x)}}
        &\text{for $\abs{u} \Hcal^{d-1}$-a.e.\ $x \in \partial \Omega$,}
    \end{cases} \\
  &\dprb{\id,\nu_x^\infty} = 0  \qquad\text{for $\lambda_\nu^*$-a.e.\ $x \in \R^d$,} \\
  &\abs{\dpr{\id,\nu_x^\infty}} \lambda_\nu^s = \abs{E^s u}
    + \absb{u|_{\partial \Omega} \odot n_{\Omega}} \, \Hcal^{d-1} \restrict \partial \Omega, \\
  &\dprb{\id,\nu_x} = 0  \qquad\text{for $x \in \R^d \setminus \cl{\Omega}$,} \\
  &\lambda_\nu \restrict (\R^d \setminus \cl{\Omega}) = 0.
\end{align*}
Moreover, consider $f$ to be extended to $\R^d \times \R_\sym^{d \times d}$ as follows: first extend $f^\infty$ restricted to $\cl{\Omega} \times \partial \Bbb_\sym^{d \times d}$ continuously to $\R^d \times \partial \Bbb_\sym^{d \times d}$ (where $\Bbb_\sym^{d \times d} := \Bbb^{d \times d} \cap \R_\sym^{d \times d}$) and then set $f(x,A) := \abs{A}f^\infty(x,A/\abs{A})$ for $x \in \R^d \setminus \cl{\Omega}$. Hence, the so extended $f$ is still a Carath\'{e}odory function, $f^\infty$ is jointly continuous and $f(x,0) = 0$ for all $x \in \R^d \setminus \cl{\Omega}$. The extended representation result for generalized Young measures~\eqref{eq:ext_repr_Caratheodory} in Section~\ref{ssc:YM} (the original result is in Proposition~2~(i) of~\cite{KriRin10CGGY}), together with Theorem~\ref{thm:BDY_Jensen} yields
\begin{align*}
  \liminf_{j\to\infty} \Fcal(u_j) &= \int \dprb{f(x,\frarg),\nu_x}
    + \dprb{f^\infty(x,\frarg),\nu_x^\infty} \frac{\di \lambda_\nu}{\di \Lcal^d}(x) \dd x \\ 
  &\qquad + \int \dprb{f^\infty(x,\frarg),\nu_x^\infty} \dd \lambda_\nu^s(x) \\
  &\geq \int f \biggl( x, \dprb{\id,\nu_x} + \dprb{\id,\nu_x^\infty}
    \frac{\di \lambda_\nu}{\di \Lcal^d}(x) \biggr) \dd x \\
  &\qquad + \int f^\infty \bigl( x, \dprb{\id,\nu_x^\infty} \bigr) \dd \lambda_\nu^s(x) \\
  &= \Fcal(u).
\end{align*}
Hence we have established lower semicontinuity.
\end{proof}

\begin{remark}
Symmetric quasiconvexity is also necessary for weak* lower semicontinuity, since it is already necessary for weak* lower semicontinuity of $\Fcal$ restricted to $\Wrm^{1,\infty}(\Omega;\R^d)$, which is a subspace of $\BD(\Omega)$.
\end{remark}

\begin{remark}[Recession functions] \label{rem:rec_func}
Notice that we needed to require the existence of the strong recession function $f^\infty$ in the previous result and could not just use the generalized recession function $f^\#$. Unfortunately, this cannot be avoided as long as no Alberti-type theorem is available in $\BD$. The reason is that for lower semicontinuity the \emph{lower} generalized recession function
\[
  f_\#(x,A) := \liminf_{t\to\infty} \frac{f(x,tA)}{t},  \qquad\text{$x \in \cl{\Omega}$, $A \in \R^{d \times d}$,}
\]
would be the natural choice of recession function, since for $f_\#$ it still holds that
\[
  \liminf_{j\to\infty} \Fcal(u_j) \geq \int \dprb{f(x,\frarg),\nu_x} \dd x
    + \int \dprb{f_\#(x,\frarg),\nu_x^\infty} \dd \lambda_\nu(x),
\]
see Theorem~2.5~(iii) in~\cite{AliBou97NUIG} (recall that $f$ is Lipschitz continuous by quasiconvexity). The problem with that choice, however, is that we cannot easily ascertain that $f_\#$ is symmetric-quasiconvex. For $f$ such that we know a-priori that $f_\#$ is symmetric-quasiconvex, the above theorem also holds with $f_\#$ in place of $f^\infty$. Indeed, take a sequence $(f_k) \subset \Ebf(\Omega;\R_\sym^{d \times d})$ with $f_k \toup f$, $f_k^\infty \toup f_\#$, and define $\Fcal_k$ like $\Fcal$, but with $f$ replaced by $f_k$. Also, let $\Fcal_\#$ be the functional with $f^\infty$ replaced by $f_\#$. Then,
\begin{align*}
  \liminf_{j\to\infty} \Fcal_\#(u_j) &\geq \lim_{k\to\infty} \lim_{j\to\infty} \Fcal_k(u_j)
    = \lim_{k\to\infty} \ddprb{f_k,\nu} \\
  &= \int \dprb{f(x,\frarg),\nu_x} \dd x + \int \dprb{f_\#(x,\frarg),\nu_x^\infty} \dd \lambda_\nu^s(x) \\
  &\geq \Fcal_\#(u)
\end{align*}
by the monotone convergence theorem and the Jensen-type inequalities from Lemmas~\ref{lem:Jensen_regular},~\ref{lem:Jensen_singular}. Hence, $\Fcal_\#$ is weakly* lower semicontinuous.
\end{remark}

\begin{remark}[Recession functions II]
In the $\BV$-case, most previous results were formulated for the (upper) generalized recession function $f^\#$, which by Fatou's Lemma we know to be quasiconvex whenever $f$ is. This is explained by the fact that $f^\infty = f_\# = f^\#$ on the rank-one cone, and by Alberti's Rank-One Theorem, we know that at $\abs{D^s u}$-almost every $x \in \Omega$, $\rank \Bigl( \frac{D^s u}{\abs{D^s u}}(x) \Bigr) \leq 1$, so the different recession functions are interchangeable. Of course, if we had an Alberti-type theorem in $\BD$, for which the natural conjecture is
\[
  \frac{E^s u}{\abs{E^s u}}(x) \in \setb{ a \odot b }{ a,b \in \R^d }
  \qquad\text{for $\abs{E^s u}$-a.e.\ $x \in \Omega$,}
\]
we could indeed use $f^\#$ instead of $f^\infty$. In fact, assuming that this conjecture in $\BD$ is true, we have $\Fcal_\# = \Fcal^\#$ and so, since we know from the previous remark that $\Fcal_\#$ is weakly* lower semicontinuous, we conclude the same for $\Fcal^\#$.
\end{remark}

The Direct Method of the Calculus of Variations together with the usual compactness results in $\BD$ immediately implies:

\begin{corollary}[Existence of minimizers]  \label{cor:min_existence}
Let $f \colon \cl{\Omega} \times \R_\sym^{d \times d}$ be as in Theorem~\ref{thm:BD_lsc}, and additionally assume the coercivity condition
\[
  m(\abs{A}-1) \leq f(x,A),  \qquad\text{$x \in \cl{\Omega}$, $A \in \R_\sym^{d \times d}$,}
\]
for some $m > 0$. Then, the variational problem
\[
  \Fcal(u) \quad\to\quad \min  \qquad\text{over $u \in \BD(\Omega)$}
\]
with $\Fcal$ defined as in~\eqref{eq:F_def}, has a solution.
\end{corollary}

\begin{remark}[Dirichlet boundary conditions]
Since the trace operator is not weakly* continuous in $\BD(\Omega)$, boundary conditions in general are not preserved under this convergence, and we need to switch to a suitable relaxed formulation of Dirichlet boundary conditions. However, since for linear growth integrands all parts of the symmetrized derivative may interact, this constraint is not easily formulated, and is probably only meaningful in connection with concrete problems. Some results for special $\BD$-functions can be found in~\cite{BeCoDa98CLSP}, Chapter~II.8 of~\cite{Tema85MPP} (also see Proposition~II.7.2) treats the case where additionally divergences converge weakly. Finally, Section~14 of~\cite{Gius83MSFB} contains general remarks on boundary conditions for linear growth functionals.
\end{remark}

Also, we immediately have the following relaxation theorem.

\begin{corollary}[Relaxation]  \label{cor:relaxation}
Let $f \in \Ebf(\Omega;\R_\sym^{d \times d})$ be symmetric-quasiconvex in its second argument. Then, the lower semicontinuos envelope of the functional
\[
  \int_\Omega f \bigl( x, \Ecal u(x) \bigr) \dd x
    + \int_{\partial \Omega} f^\infty \bigl( x, u(x) \odot n_\Omega(x) \bigr) \dd \Hcal^{d-1}(x),
    \qquad u \in \LD(\Omega),
\]
with respect to weak* convergence in $\BD(\Omega)$ is the functional $\Fcal$ from~\eqref{eq:F_def}.
\end{corollary}

Of course, for $f \geq 0$, we again may omit the boundary term.

\begin{proof}
Denote the $\Gcal$ the functional defined in the statement of the corollary and let $\Gcal_*$ be its weakly* (sequentially) lower semicontinuous envelope. By Reshetnyak's Continuity Theorem~\ref{thm:reshetnyak}, also see Corollary~\ref{cor:F_strictly_cont_ext}, $\Fcal$ is the $\langle\frarg\rangle$-strictly continuous extension of $\Gcal$ to $\BD(\Omega)$, in particular $\Gcal_* \leq \Fcal$. On the other hand, $\Fcal$ is weakly* lower semicontinuous, hence also $\Fcal \leq \Gcal_*$. 
\end{proof}

\begin{remark}
Of course it would be desirable to have a relaxation theorem for integrands $f$ that are not symmetric-quasiconvex. Then, the relaxed functional should be $\Fcal$ from~\eqref{eq:F_def}, but with $f$ replaced by its symmetric-quasiconvex envelope $\mathrm{SQ}f$. However, we do not know whether $(\mathrm{SQ}f)^\infty$ exists, and without an Alberti-type theorem in $\BD$, we cannot show lower semicontinuity for the functional with $(\mathrm{SQ}f)^\infty$ replaced by $(\mathrm{SQ}f)^\#$ within our framework, see the remarks above.
\end{remark}

\section{Concluding remarks} \label{sc:concl_remarks}

It should be remarked that most parts of the proof could also be reformulated in a more elementary fashion, circumventing the machinery of Young measures. However, without the use of tangent Young measures and working with blow-up sequences directly, several arguments would require additional technical steps. Particularly the construction of ``good'' blow-ups through the ``iterated blow-up'' trick in Theorem~\ref{thm:good_blowups} is not easily formulated with mere sequences instead of tangent Young measures. At the core of this lies the fact that in the blow-up technique, we are not primarily interested with the blow-up \emph{limit}, but with the behavior of the blow-up \emph{sequence}, just as represented in a (generalized) Young measure limit. This is precisely the idea behind the concept of tangent Young measures, and the Localization Principles, Propositions~\ref{prop:localize_reg} and~\ref{prop:localize_sing}, encapsulate all the technicalities of the blow-up process. Therefore, while Young measures are not in a strict sense necessary to formulate the proof, they provide an elegant conceptual framework for organizing the course of the argument by separating the technical aspects from the core ideas and allowing for a clearer exposition.

For integrands $f(x,u,Eu)$ depending also on the function $u$ itself, the results presented here (in particular the Jensen-type inequalities in Theorem~\ref{thm:BDY_Jensen}) should also yield a lower semicontinuity theorem for this extended situation together with some \enquote{freezing of $u$} idea for Young measures. One needs to be careful with the definition of a suitable recession function, though, and also jump points (where instead of $u(x)$ we have only the one-sided traces $u^-(x), u^+(x)$) need special attention. This is currently work in progress.

\appendix

\section{Existence of non-zero tangent measures} \label{ax:Tan_existence}

In this appendix we give a Preiss' proof on the existence of non-zero tangent measures, originally in Theorem~2.5 of~\cite{Prei87GMDR}.

\begin{lemma}
Let $\mu \in \Mbf_\loc(\R^d;\R^N)$. At $\abs{\mu}$-almost every $x_0 \in \R^d$, the set $\Tan(\mu,x_0)$ contains a non-zero measure.
\end{lemma}

\begin{proof}
Using~\eqref{eq:Tan_abs}, we may assume that $\mu$ is a positive measure. Moreover, restricting if necessary to a sufficiently large closed ball containing $x_0$, we can even assume $\mu \in \Mbf^+(K)$ for some compact set $K \subset \R^d$ with $x_0 \in K$.

\proofstep{Step~1.} First, we note that for all relatively compact Borel sets $A \subset \R^d$ it holds that
\begin{equation} \label{eq:measure_trafo}
  \mu(A) = \frac{1}{\omega_d r^d} \int \mu(A \cap B(x,r)) \dd x,
\end{equation}
where $\omega_d$ denotes the volume of the unit ball in $\R^d$. This follows with the aid of Fubini's Theorem:
\begin{align*}
  \int \mu(A \cap B(x,r)) \dd x &= \int \int \ONE_A(y) \ONE_{B(x,r)}(y) \dd \mu(y) \dd x \\
  &= \int \ONE_A(y) \int \ONE_{B(y,r)}(x) \dd x \dd \mu(y) = \omega_d r^d \mu(A).
\end{align*}

\proofstep{Step~2.} We now show that for all $t > 1$ it holds that
\begin{equation} \label{eq:doubling_est}
  \lim_{\beta \to \infty} \limsup_{r \todown 0} \, \mu \bigl( \setb{ x \in K }{ \mu(B(x,tr))
  \geq \beta \mu(B(x,r)) } \bigr) = 0.
\end{equation}
For this, let $\epsilon > 0$, $\beta > (2(t+1))^d \mu(K)/\epsilon$ and fix any $r > 0$. Also define
\[
  E := \setb{ x \in K }{\mu(B(x,tr)) \geq \beta \mu(B(x,r))}.
\]
Whenever $B(x,r/2) \cap E \neq \emptyset$ for some $r > 0$, take $z \in B(x,r/2) \cap E$ to estimate
\[
  \beta \mu(B(x,r/2)) \leq \beta \mu(B(z,r)) \leq \mu(B(z,tr)) \leq \mu(B(x,(t+1)r)).
\]
Hence we get from~\eqref{eq:measure_trafo},
\begin{align*}
  \mu(E) &= \frac{1}{\omega_d \cdot (r/2)^d} \int \mu(E \cap B(x,r/2)) \dd x \\
  &\leq \frac{(2(t+1))^d}{\beta} \cdot \frac{1}{ \omega_d \cdot ((t+1)r)^d} \int \mu(B(x,(t+1)r)) \dd x \\
  &=\frac{(2(t+1))^d}{\beta} \mu(K) < \epsilon.
\end{align*}
This clearly implies~\eqref{eq:doubling_est}. In fact, it even implies this assertion with the limes superior replaced by the supremum over all $r > 0$. This, however, is due to the fact that we without loss of generality restricted the measure $\mu$ to the compact set $K$, and so a smallness assumption on $r$ is already implicit.

\proofstep{Step~3.} From~\eqref{eq:doubling_est} we see that for all $\epsilon > 0$ and all $k = 2,3,\ldots$ there exists constants $\beta_k > 0$ and $t_k > 0$ such that
\[
  \mu \bigl( \setb{ x \in K }{ \mu(B(x,kr)) \geq \beta_k \mu(B(x,r)) } \bigr) \leq \frac{\epsilon}{2^k}
  \qquad\text{whenever $r \in (0,t_k)$.}
\]
Then, for $r > 0$ set
\begin{align*}
  A_r := \setB{ x \in K }{ &\text{there exists a $k \in \{2,3,\ldots\}$ with $r \in (0,t_k)$ such that} \\
                           &\text{$\mu(B(x,kr)) \geq \beta_k \mu(B(x,r))$} }
\end{align*}
and observe that $\mu(E_r) \leq \epsilon$ by the previous estimate. Hence, also
\[
  A := \bigcup_{i=1}^\infty \bigcap_{j=i}^\infty A_{1/j}
\]
satisfies $\mu(A) \leq \epsilon$. Since $\epsilon > 0$ was arbitrary, this implies $\mu(A) = 0$.

Let now $x \in K \setminus A$. Then, for all $i \in \N$ there exists $j \geq i$ such that $x \notin A_{1/j}$, i.e.\ for all $k \in \N$ with $1/j \leq t_k$,
\[
  \mu(B(x,k/j)) \leq \beta_k \mu(B(x,1/j)).
\]
Therefore, for $\mu$-almost every $x_0 \in \supp \mu$ (and hence $\mu$-almost every $x_0 \in K$), there exists a sequence $r_n \todown 0$ with
\[
  \limsup_{n \to \infty} \frac{\mu(B(x_0,kr_n))}{\mu(B(x_0,r_n))} \leq \beta_k
  \qquad\text{for all $k \in \N$.}
\]
This allows us to infer that the sequence $c_n T_*^{(x_0,r_n)}$ with $c_n := \mu(B(x_0,r_n))^{-1}$ is weakly* compact in $\Mbf_\loc(\R^d)$ and every weak* limit of a subsequence is a non-zero tangent measure to $\mu$ at $x_0$.
\end{proof}



\providecommand{\bysame}{\leavevmode\hbox to3em{\hrulefill}\thinspace}
\providecommand{\MR}{\relax\ifhmode\unskip\space\fi MR }
\providecommand{\MRhref}[2]{%
  \href{http://www.ams.org/mathscinet-getitem?mr=#1}{#2}
}
\providecommand{\href}[2]{#2}

\end{document}